\documentclass{article}
\usepackage{amssymb}
\usepackage{amsmath}
\usepackage{pifont}
\usepackage{amsfonts}
\usepackage{latexsym}
\usepackage{verbatim}
\usepackage{float}
\usepackage{exscale}
\usepackage[margin=0.95in]{geometry}
\usepackage{color}

\allowdisplaybreaks
\renewcommand{\ldots}{\ensuremath{\dotsc}}

\renewcommand{\theequation}{\arabic{section}.\arabic{equation}}

\newcommand{\BIGOP}[1]{\mathop{\mathchoice%
{\raise-0.22em\hbox{\huge $#1$}}%
{\raise-0.05em\hbox{\Large $#1$}}{\hbox{\large $#1$}}{#1}}}

\newcommand{\BIGboxplus}{\mathop{\mathchoice%
{\raise-0.35em\hbox{\huge $\boxplus$}}%
{\raise-0.15em\hbox{\Large $\boxplus$}}{\hbox{\large $\boxplus$}}{\boxplus}}}

\newcommand{\Rplus}{{\mathbb R}_{>0}}

\def\epsilon{\varepsilon}
\def\hat{\widehat}

\def\undertilde#1{{\baselineskip=0pt\vtop
{\hbox{$#1$}\hbox{$\scriptscriptstyle\sim$}}}{}}
\def\underdtilde#1{{\baselineskip=0pt\vtop
{\hbox{$#1$}\hbox{$\scriptscriptstyle\approx$}}}{}}
\def\epsilon{\varepsilon}
\def\hat{\widehat}

\def\b{\beta}
\def\nabq{\undertilde{\nabla}_{q}}
\def\nabqi{\undertilde{\nabla}_{q_i}}

\def\nabx{\undertilde{\nabla}_{x}}
\def\nabr1{\undertilde{\nabla}_{r_1}}
\def\nabr2{\undertilde{\nabla}_{r_2}}
\def\delx{\Delta_{x}}

\def\at{\undertilde{a}}
\def\ut{\undertilde{u}}

\def\utLDt{\ut_L^{\Delta t}}
\def\utLDtm{\ut_L^{\Delta t,-}}
\def\utLDtp{\ut_L^{\Delta t,+}}
\def\utLDtpm{\ut_L^{\Delta t(,\pm)}}
\def\vt{\undertilde{v}}
\def\wt{\undertilde{w}}
\def\xt{\undertilde{x}}
\def\qt{\undertilde{q}}

\def\ft{\undertilde{f}}
\def\gt{\undertilde{g}}

\def\nt{\undertilde{n}}
\def\yt{\undertilde{y}}

\def\Bt{\undertilde{B}}
\def\Ct{\undertilde{C}}
\def\Dt{\undertilde{D}}
\def\Ft{\undertilde{F}}

\def\Ht{\undertilde{H}}

\def\Lt{\undertilde{L}}
\def\Ltdiv{\Lt_{\rm div}}

\def\St{\undertilde{S}}

\def\Vt{\undertilde{V}}
\def\Wt{\undertilde{W}}

\def\dq{\,{\rm d}\undertilde{q}}
\def\dx{\,{\rm d}\undertilde{x}}

\def\dt{\,{\rm d}t}
\def\zerot{\undertilde{0}}
\def\zerott{\underdtilde{0}}
\def\tautt{\underdtilde{\tau}}

\def\xitt{\underdtilde{\xi}}
\def\utOB{\ut_{\rm OB}}

\def\hpsi{\widehat{\psi}}
\def\hpsiOB{\widehat{\psi}_{\rm OB}}

\def\hpsi{\widehat{\psi}}

\def\hpsiLDt{\widehat{\psi}_{L}^{\Delta t}}
\def\hpsiLDtp{\widehat{\psi}_{L}^{\Delta t,+}}
\def\hpsiLDtm{\widehat{\psi}_{L}^{\Delta t,-}}
\def\hpsiLDtpm{\widehat{\psi}_{L}^{\Delta t,\pm}}
\def\utst{\ut_\star}
\def\psist{\psi_\star}
\def\hpsist{\widehat{\psi}_\star}
\def\hpsistL{\widehat{\psi}_{\star,L}}
\def\hpsistLd{\widehat{\psi}_{\star,L,\delta}}
\def\hpsistLDt{\widehat{\psi}_{\star,L}^{\Delta t}}
\def\hpsistLDtp{\widehat{\psi}_{\star,L}^{\Delta t,+}}
\def\hpsistLDtm{\widehat{\psi}_{\star,L}^{\Delta t,-}}

\def\sigOB{\sigtt_{\rm OB}}

\def\Btt{\underdtilde{B}}
\def\Ctt{\underdtilde{C}}

\def\Dtt{\underdtilde{D}}
\def\Htt{\underdtilde{H}}
\def\Itt{\underdtilde{I}}

\def\Ltt{\underdtilde{L}}
\def\Wtt{\underdtilde{W}}
\def\nabxtt{\underdtilde{\nabla}_{x}\,}

\def\vnabtt{\underdtilde{\nabla}_{x}\,\vt}

\def\sigtt{\underdtilde{\sigma}}
\def\zetatt{\underdtilde{\zeta}}
\def\ztt{\underdtilde{z}}

\def\ptut{\frac{\partial \ut}{\partial t}}

\def\dd {{\,\rm d}}

\def\Ttt{\underdtilde{T}}

\def\ocirc#1{\ifmmode\setbox0=\hbox{$#1$}\dimen0=\ht0
    \advance\dimen0 by1pt\rlap{\hbox to\wd0{\hss\raise\dimen0
    \hbox{\hskip.2em$\scriptscriptstyle\circ$}\hss}}#1\else
    {\accent"17 #1}\fi}

\newcommand{\bet}{\noalign{\vskip6pt plus 3pt minus 1pt}}

\newcounter{ind}
\def\eqlabstart{%
 \setcounter{ind}{\value{equation}}\addtocounter{ind}{1}%
 \setcounter{equation}{0}%
 \renewcommand{\theequation}{\arabic{section}.\arabic{ind}\alph{equation}}%
}
\def\eqlabend{%
 \renewcommand{\theequation}{\arabic{section}.\arabic{equation}}%
 \setcounter{equation}{\value{ind}}%
}

\newtheorem{definition}{Definition}[section]
\newtheorem{lemma}{Lemma}[section]
\newtheorem{theorem}{Theorem}[section]
\newtheorem{remark}{Remark}[section]

\newenvironment{proof}[1][Proof]{\begin{trivlist}
\item[\hskip \labelsep {\bfseries #1}]}{ $\Box$\end{trivlist}}

\newcounter{appendix}\renewcommand\appendix{\par
        \refstepcounter{appendix}
        \setcounter{section}{0}%
        \setcounter{theorem}{0}
        \setcounter{equation}{0}
\renewcommand\thesection{\appendixname\ \Alph{section}}
\renewcommand\thesubsection{\Alph{section}.\arabic{subsection}}%
\renewcommand\theequation{\Alph{section}.\arabic{equation}}}%
\begin{document}

\title{\textbf{\large
EXISTENCE OF GLOBAL WEAK SOLUTIONS TO \\[3mm]
THE KINETIC HOOKEAN DUMBBELL MODEL FOR \\
INCOMPRESSIBLE DILUTE POLYMERIC FLUIDS
}}

\author{\textit{\normalsize John W. Barrett}\\
{\normalsize \textit{Department of Mathematics, Imperial College London, London SW7 2AZ, UK}}\\
{\small \texttt{j.barrett@imperial.ac.uk}}
\\
~\\
\textit{\normalsize Endre S\"uli}\\
{\normalsize \textit{Mathematical Institute, University of Oxford, Oxford OX1 3LB, UK}}\\
{\small \texttt{endre.suli@maths.ox.ac.uk}}}

\date{~}
\maketitle

\begin{abstract}
We explore the existence of global weak solutions to the Hookean dumbbell model,
a system of nonlinear partial differential equations that arises from the kinetic theory of
dilute polymers, involving the unsteady incompressible Navier--Stokes equations in a bounded domain in two or
three space dimensions, coupled to a Fokker--Planck-type parabolic equation. We prove the existence
of large-data global weak solutions in the case of two space dimensions. Indirectly,
our proof also rigorously demonstrates that, in two space dimensions at least, the Oldroyd-B model is the macroscopic closure
of the Hookean dumbbell model. In three space dimensions, we prove the existence of large-data global weak subsolutions to the model, which are weak solutions with a defect measure, where the defect measure appearing in the Navier--Stokes momentum equation is the divergence of a symmetric positive semidefinite matrix-valued Radon measure.

\medskip

\noindent
\textit{Keywords:} Kinetic polymer models, Hookean dumbbell model,
Navier--Stokes--Fokker--Planck system, dilute polymer, Oldroyd-B model
\end{abstract}

\section{Introduction}
\label{sec:1}
The aim of this paper is to explore the existence of global weak solutions to the Hookean dumbbell model,
--- a system of nonlinear partial differential equations involving the coupling of the time-dependent
incompressible Navier--Stokes equations to a parabolic Fokker--Planck type equation, --- which
arises from the kinetic theory of dilute polymeric fluids.
In this model the solvent is assumed to be an isothermal, viscous, incompressible,
Newtonian fluid in a bounded open Lipschitz domain $\Omega \subset \mathbb{R}^d$, $d=2$ or $3$. We will admit
both $d=2$ and $d=3$ for the vast majority of the paper, even though our main result concerning the existence
of large-data global weak solutions is, ultimately, restricted to the case of $d=2$.
In the model the equation for conservation of linear momentum in the Navier--Stokes system
involves, as a source term, an elastic {\em extra-stress} tensor $\tautt$ (i.e.,\ the polymeric part of the
Cauchy stress tensor), to be defined below in terms of the solution of a coupled Fokker--Planck type equation.

Given $T \in \mathbb{R}_{>0}$, we seek a nondimensional velocity field  $\ut\,:\,(\xt,t)\in
\overline\Omega \times [0,T] \mapsto \ut(\xt,t) \in {\mathbb R}^d$ (which, for simplicity,
we shall require to satisfy a no-slip boundary condition on $\partial \Omega \times (0,T]$)
and a nondimensional pressure $p\,:\, (\xt,t) \in \Omega \times (0,T] \mapsto p(\xt,t) \in {\mathbb R}$,
such that
\begin{subequations}
\begin{alignat}{2}
\ptut + (\ut \cdot \nabx)\, \ut - \nu \,\delx \ut + \nabx p
&= \ft + \nabx \cdot \tautt \qquad &&\mbox{in } \Omega \times (0,T],\label{ns1a}\\
\nabx \cdot \ut &= 0        \qquad &&\mbox{in } \Omega \times (0,T],\label{ns2a}\\
\ut &= \zerot               \qquad &&\mbox{on } \partial \Omega \times (0,T],\label{ns3a}\\
\ut(\xt,0)&=\ut_{0}(\xt)    \qquad &&\forall \xt \in \Omega.\label{ns4a}
\end{alignat}
\end{subequations}
In these equations $\nu\in \mathbb{R}_{>0}$ is the reciprocal of the Reynolds number
and $f$ is the nondimensional density of body forces.

The simplest kinetic model for a dilute polymeric fluid is the {\em dumbbell model}, where long polymer chains suspended
in the viscous incompressible Newtonian solvent are assumed not to interact with each other, and each chain
is idealized as a pair of massless beads, connected with an elastic spring.
The associated elastic extra-stress tensor
$\tautt$ is defined by the \textit{Kramers expression} in terms of $\psi$,
the probability density function of the (random) conformation vector
$\qt$ 
of the spring (cf. (\ref{tau1}) below), and the
domain $D$ of admissible conformation vectors $\qt$ is either the
whole of $\mathbb{R}^d$ or a bounded open $d$-dimensional ball
centred at the origin $\zerot \in \mathbb{R}^d$.
The evolution of $\psi$ from a given nonnegative initial
datum, $\psi_0$, is governed by a second-order parabolic partial differential equation,
the Fokker--Planck equation, whose transport coefficients depend on the velocity field $\ut$.

In \cite{BS2011-fene} we were concerned with models where $D$ 
is a bounded open ball in $\mathbb{R}^d$, $d=2,3$, resulting in, what are known as,
finitely extensible nonlinear (FENE) models. Here, as in
\cite{BS2010-hookean}, we shall be concerned with
the technically more subtle case
when $D = \mathbb{R}^d$, 
i.e., the spring between the beads is allowed to have an arbitrarily large extension.
In fact, in both \cite{BS2011-fene} and \cite{BS2010-hookean} we considered the more general
case of polymer models involving linear chains of $K+1$ beads coupled with $K$ springs, where $K \geq 1$.
Much of the analysis below carries across to $K>1$, but as our final result is restricted to $K=1$
we shall confine ourselves to this case from the start for simplicity.
Although springs with arbitrarily large extension are physically unrealistic,
thanks to their simplicity models of this kind are, nevertheless, frequently used in practice.
The elastic spring-force $\Ft\,:\, D=\mathbb{R}^d \rightarrow \mathbb{R}^d$ of the spring
is then defined by
\begin{equation}\label{eqF}
\Ft(\qt) = U'(\textstyle{\frac{1}{2}}|\qt|^2)\,\qt,
\end{equation}
where $U \!\in\! W^{2,\infty}_{\rm loc}([0,\infty);\mathbb {R}_{\geq 0})$,
$U(0)=0$ and $U$ is monotonic nondecreasing and unbounded on $[0,\infty)$.
We note that, among all such potentials $U$, only the Hookean potential
$U(s)=s$, with associated spring force
${\Ft}({\qt}) = {\qt}$, ${\qt} \in {D}=\mathbb{R}^d$,
yields, formally at least, closure to a macroscopic model, the Oldroyd-B model (cf. \cite{Oldroyd}), which
we shall state below.

In our paper \cite{BS2010-hookean}, we further assumed that
there exist constants $c_{j}>0$,
$j=1, 2, 3, 4$, 
such that the (normalized) Maxwellian $M$, is defined on $D = \mathbb{R}^d$
by
%
\begin{equation}
M(\qt) := \frac{1}{\mathcal{Z}}{\rm e}^{-U(\frac{1}{2}|\qt|^2)},
\quad \mbox{where}\quad \mathcal{Z}:=
{ \displaystyle \int_{D} {\rm e}^{-U(\frac{1}{2}|\qt|^2)} \dq},
\quad \mbox{yielding}\quad  \int_D M(\qt) \dq =1,
\label{MN}
\end{equation}
%
and the associated spring potential $U$ satisfies,
for $\vartheta>1$,
\eqlabstart
\begin{alignat}{2}
U(\textstyle{\frac{1}{2}}|\qt|^2) & = c_{1} \,(\textstyle{\frac{1}{2}}|\qt|^2)^{\vartheta}
\qquad &&\mbox{as } |\qt| \rightarrow \infty,
 \label{growth1}\\
U'(\textstyle{\frac{1}{2}}|\qt|^2) & \leq c_{2} + c_{3} \,(\textstyle{\frac{1}{2}}|\qt|^2)
^{\vartheta-1}
\qquad &&\forall \qt \in D,
 \label{growth2}\\
 \hspace*{-1.55in}
\mbox{and hence}\hspace{2in}&&\nonumber \\
 M(\qt)
& =
c_{4}\,
{\rm e}^{-c_{1}(\frac{1}{2}|\qt|^2)^{\vartheta}}
\qquad &&\mbox{as } |\qt| \rightarrow \infty.
 \label{growth3}
\end{alignat}
\eqlabend

Therefore, in \cite{BS2010-hookean} we were unable to cover the case of the Hookean
model, which corresponds to the choice $\vartheta =1$ in (\ref{growth1}--c).
This was due to the failure of a compactness argument,
used in passing to a limit on the extra-stress tensor
in the existence proof there, which required that the mapping
$s\mapsto U(s)$ 
had superlinear growth at infinity;
see Section \ref{remexist} below for further details.
We note that apart from this step, all the other results
in \cite{BS2010-hookean} remain valid with $\vartheta \geq 1$
in (\ref{growth1}--c).
Thus in \cite{BS2010-hookean}
we could deal
with a spring potential of the form
\begin{align}
U(s) = \left\{ \begin{array}{ll}
s \qquad &\mbox{for } s \in [0,s_\infty],\\
\frac{s_\infty}{\vartheta} \left[ \left(\frac{s}{s_\infty}\right)^{\vartheta} + (\vartheta-1) \right]
\qquad & \mbox{for } s \geq s_\infty
\end{array}
\right.
\label{eqHM}
\end{align}
for any $s_\infty >0$ and $\vartheta > 1$,
which approximates the Hookean potential $U(s)=s$.
Hence our use of the terminology {\em Hookean-type model} throughout the paper
\cite{BS2010-hookean} (instead
of {\em Hookean model}, which would have corresponded to taking $\vartheta=1$ in the above).

In this paper, we overcome the difficulties encountered with the compactness
result used on the extra-stress term
in \cite{BS2010-hookean}
and so we are able to address 
the Hookean model.
This is achieved by relating the Hookean model to its macroscopic closure,
the Oldroyd-B model,
and using an existence result, and establishing regularity results, for the latter.
We shall assume henceforth, in place of (\ref{growth1}--c), that
\begin{align}
U(\textstyle\frac{1}{2}|\qt|^2) = 
\frac{1}{2}|\qt|^2. 
\label{U}
\end{align}
By recalling (\ref{MN}), we observe that the Maxwellian satisfies
\begin{equation}
M(\qt)\,\nabq [M(\qt)]^{-1} = - [M(\qt)]^{-1}\,\nabq M(\qt) =
\nabq U(\textstyle{\frac{1}{2}}|\qt|^2)
=U'(\textstyle{\frac{1}{2}}|\qt|^2)\,\qt
=
\qt. \label{eqM}
\end{equation}
%
It follows from (\ref{MN}) and (\ref{U}) that, for any $r \in [0,\infty)$,
\begin{equation}\label{additional-1}
\int_{D} M(\qt)\,
|\qt|^r
\, \dd
\qt
< \infty. 
\end{equation}

The governing equations of the Hookean dumbbell model are
(\ref{ns1a}--d), where the extra-stress tensor $\tautt$, dependent
on the probability density function
$\psi :(\xt,\qt,t) \in \Omega \times D \times [0,T] \mapsto \psi(\xt,\qt,t) \in
{\mathbb R}_{\geq 0}$,
is defined by the \textit{Kramers expression}:
\begin{equation}\label{tau1}
\tautt(\psi) :=  k \left( \sigtt(\psi) - \rho(\psi)\, \Itt \right)\!,
\end{equation}
where
the dimensionless constant $k>0$ is a constant multiple of
the product of the Boltzmann constant $k_B$ and the absolute temperature $\mathcal{T}$,
$\Itt$ is the unit $d \times d$ tensor,
\eqlabstart
\begin{equation}
\sigtt(\psi)(\xt,t) :=
\int_{D} \psi(\xt,\qt,t)\, \qt\,
\qt^{\rm T}\, 
U'\left(\textstyle \frac{1}{2}|\qt|^2\right)
{\dd} \qt = 
\int_{D} \psi(\xt,\qt,t)\, \qt\,
\qt^{\rm T}\, 
{\dd} \qt 
\label{C1}
\end{equation}
and the density of polymer chains located at $\xt$ at time $t$ is given by
\begin{equation}\label{rho1}
\rho(\psi)(\xt,t) := \int_{D} \psi(\xt, \qt,t)
{\dd}\qt.
\end{equation}
\eqlabend
The probability density function $\psi$ is a solution, in $\Omega \times D \times (0,T]$, of the
Fokker--Planck equation
\begin{align}
\label{fp0}
\hspace{-3mm}\frac{\partial \psi}{\partial t} + (\ut \cdot\nabx) \psi +
\nabq
\cdot \left((\nabxtt \ut) \, \,\qt\, 
\psi \right)
=
\epsilon\,\Delta_x\,\psi +
\frac{1}{4 \,\lambda}\,
\nabq \cdot \left(
M\,\nabq\! \left(\frac{\psi}{M}\right)\right),
\end{align}
where, for $\vt=\vt(\xt,t) \in \mathbb{R}^d$, $(\vnabtt)(\xt,t) \in {\mathbb
R}^{d \times d}$ and $\{\vnabtt\}_{ij} := \textstyle
\frac{\partial v_i}{\partial x_j}$.
In \eqref{fp0}, $\varepsilon>0$ is the dimensionless centre-of-mass diffusion coefficient
and the dimensionless parameter $\lambda \in \Rplus$ is the Deborah number;
we refer the reader to \cite{BS2011-fene,BS2010-hookean,BS2010,BS2011-feafene} for further details.

Finally, we impose the following decay/boundary and initial conditions on $\psi$:
\begin{subequations}
\begin{alignat}{2}
&\left|M \left[ \frac{1}{4\,\lambda} 
\,\nabq \!\left(\frac{\psi}{M}\right)
- (\nabxtt \ut) \,\,\qt\,
\frac{\psi}{M}
\right] \right|
\rightarrow 0\quad \mbox{as} \quad |\qt| \rightarrow \infty 
\qquad  
&&\mbox{on }
\Omega 
\times (0,T],
\label{eqpsi3a}\\
&\epsilon\,\nabx \psi\,\cdot\,\nt =0 \qquad &&\mbox{on }
\partial \Omega \times D\times (0,T],\label{eqpsi3b}\\
&\psi(\cdot,\cdot,0)=\psi_{0}(\cdot,\cdot) \geq 0
\qquad&&\mbox{on $\Omega\times D$},\label{eqpsi3c}
\end{alignat}
\end{subequations}
where 
$\nt$ is the unit outward normal vector to $\partial \Omega$.

Here $\psi_0$ is a nonnegative function
defined on $\Omega\times D$,
with $\int_{D} \psi_0(\xt,\qt) \dd \qt = 1$ for a.e. $\xt  \in \Omega$.
The boundary conditions for $\psi$
on $\partial\Omega\times D\times(0,T]$ and the decay conditions  for $\psi$
on $\Omega
\times (0,T]$ as $|\qt| \rightarrow \infty$
have been chosen so as to ensure that
\begin{align}
\int_{D}\psi(\xt,\qt,t) \dd \qt  =
\int_{D} \psi(\xt,\qt,0) \dd \qt  \qquad \forall (\xt,t) \in \Omega \times (0,T].
\label{intDcon}
\end{align}


%
\begin{definition} \label{Pedef}
The set of equations and hypotheses
{\rm (\ref{ns1a}--d)}, \eqref{MN}, \eqref{U}
and \eqref{tau1}--{\rm (\ref{eqpsi3a}--c)} will be referred to henceforth as model
$({\rm P})$, or as the Hookean dumbbell
model
(with centre-of-mass diffusion).
\end{definition}

Next, for any $\at \in {\mathbb R}^d$, we note the following results:
\begin{align}
&(\at \cdot \nabq) \, \qt \,\qt^{\rm T} = \at \,\qt^{\rm T}
+\qt \,\at^{\rm T}, \label{adel} 
\quad
\Delta_q\,[ \qt\,\qt^{\rm T} ]   = 
2\,\Itt,
\quad(\at \cdot \nabq) \,|\qt|^2 = 2 \,\at \cdot \qt
\quad \mbox{and}\quad
\Delta_q\,|\qt|^2 
= 2\,d.
\end{align}

Multiplying (\ref{fp0}) by $\qt \,\qt^{\rm T}$, integrating over $D$,
performing integration by parts (assuming that $\psi$ and $\nabq \psi$ decay to zero,
sufficiently fast as $|\qt| \rightarrow \infty$) 
and noting (\ref{adel}),
and, similarly, integrating (\ref{fp0}) over $D$ and noting (\ref{eqpsi3a}),
yields formally that
$\sigtt(\psi)(\xt,t) \in {\mathbb R}^{d \times d}$
and $\rho(\psi)(\xt,t) \in {\mathbb R}$
satisfy
\eqlabstart
\begin{align}
&\frac{\partial \sigtt(\psi)}{\partial t} + (\ut \cdot\nabx) \sigtt(\psi)
- \left[ (\nabxtt \ut) \, \sigtt(\psi) + \sigtt(\psi) \, (\nabxtt \ut)^{\rm T}
\right] - {\epsilon}\,\Delta_x\,\sigtt(\psi) =
\frac{1}{2 \,\lambda}
\left[\rho(\psi)\,\Itt -\sigtt(\psi) \right]
\nonumber \\
&\hspace{4.5in} \mbox{in } \Omega \times (0,T],
\label{sigijeq}
\\
&\frac{\partial \rho(\psi)}{\partial t} + (\ut \cdot\nabx) \rho(\psi)
- \epsilon\,\Delta_x\,\rho(\psi)=0
\hspace{2.2in} \mbox{in } \Omega \times (0,T],
\label{rhoeq}
\end{align}
subject to the boundary and initial conditions
\begin{alignat}{2}
&\epsilon\,\nabx \sigtt(\psi)\,\cdot\,\nt =\zerott \quad\mbox{on }
\partial \Omega \times (0,T], \qquad
\sigtt(\psi)(\cdot,0)&&=\sigtt(\psi_{0})(\cdot)
\quad\mbox{on $\Omega$},
\nonumber \\
&\epsilon\,\nabx \rho(\psi)\,\cdot\,\nt =0 \quad\mbox{on }
\partial \Omega \times (0,T], \qquad \;  \; \,
\rho(\psi)(\cdot,0)&&=\rho(\psi_{0})(\cdot)
\quad\mbox{on $\Omega$}.
\label{eqpsig3c}
\end{alignat}
\eqlabend

As it is assumed that $\int_D \psi_0(\xt,\qt) \,\dq=1$ for a.e.\ $\xt \in \Omega$,
it follows that $\rho(\psi)\equiv 1$ is the unique solution of (\ref{rhoeq},c).

\begin{definition} \label{Qedef}
The collection of equations
{\rm (\ref{ns1a}--d)}, \eqref{tau1} 
and  {\rm (\ref{sigijeq}--c)}
will be referred to throughout the paper as model
$({\rm Q})$, or as the Oldroyd-B model (with stress-diffusion).
\end{definition}

A remark is in order concerning the evolution equation \eqref{sigijeq} for the extra stress tensor $\sigtt$.

\begin{remark}
By suppressing in our notation the dependence of $\sigtt$ and $\rho$ on $\psi$ and setting $\varepsilon=0$, equation \eqref{sigijeq} becomes
\begin{align}\label{eq:OB-nodiffusion}
&\frac{\partial \sigtt}{\partial t} + (\ut \cdot\nabx) \sigtt
- \left[ (\nabxtt \ut) \, \sigtt + \sigtt \, (\nabxtt \ut)^{\rm T}
\right] = \frac{1}{2 \,\lambda}
\big(\rho \,\Itt -\sigtt \big).
\end{align}
This is precisely the classical Oldroyd-B evolution equation for the elastic extra-stress, with our factor $1/2\lambda$ usually replaced by $1/\lambda$,
which is easily derived from equation (59) in Oldroyd's paper \cite{Oldroyd}. Indeed, by interpreting the convective derivative $\frac{\mathfrak{d}}{\mathfrak{d}t}$ in equation (59) in \cite{Oldroyd} as the upper convected derivative
\[ \overset{\tiny \nabla}{\cdot}\qquad \mbox{defined by}\qquad \overset{\tiny \nabla}{\sigtt}:= \frac{\partial \sigtt}{\partial t} + (\ut \cdot\nabx) \sigtt
- \left[ (\nabxtt \ut) \, \sigtt + \sigtt \, (\nabxtt \ut)^{\rm T}
\right],\]
and by additively splitting the (total) Cauchy stress tensor into a part $\Ttt$,
to be defined below, related to the
distortion (deformation at constant volume), and an isotropic tensor that is a scalar multiple of
$\Itt$, as in equation (52) in \cite{Oldroyd}, equation (59) in \cite{Oldroyd} states that
\[ \Ttt + \lambda\,\overset{\tiny \nabla}{\Ttt} = 2\mu\,(\Dtt + \alpha \,\overset{\tiny \nabla}{\Dtt}),\qquad \mbox{where}\qquad \Dtt = \Dtt(\ut):=\frac{1}{2}(\nabxtt \ut + (\nabxtt \ut)^{\rm T})\]
is the symmetric velocity gradient,
$\mu=\mu_s + \mu_p$ is the sum of the solvent viscosity $\mu_s>0$ and the polymeric viscosity $\mu_p>0$; and
$\lambda>0$, the relaxation time, and $\alpha>0$, the retardation time, are two constants
with the dimension of time. By splitting the tensor $\Ttt$ additively into its solvent part $\Ttt_s$ and
polymeric part $\Ttt_p$ as
$\Ttt = \Ttt_s + \Ttt_p$, where $\Ttt_s:= 2\mu_s \,\Dtt$, and setting $\alpha:=\lambda\mu_s/(\mu_p + \mu_s)$, it follows that
\[\Ttt_p + \lambda \,\overset{\tiny \nabla}{\Ttt_p} = 2\mu_p \,\Dtt.\]
As $\overset{\tiny \nabla}{\Itt} = -2\Dtt$, we then have that
$\Ttt_p + \lambda\,\overset{\tiny \nabla}{\Ttt_p}  = -\mu_p  \,\overset{\tiny \nabla}{\Itt}$.
Thus, by defining $ \sigtt  :=   \frac{\rho \lambda}{\mu_p} \,\Ttt_p + \rho \,\Itt$, where
$\rho$ is the solution of  $\frac{\partial \rho}{\partial t} + (\ut \cdot \nabx) \rho = 0$
(i.e. \eqref{rhoeq} with $\varepsilon=0$) subject to a given initial condition $\rho(\xt,0)=\rho_0(\xt)$,
we deduce that
\[ \overset{\tiny \nabla}{\sigtt}   =   \frac{1}{\lambda} \,(\rho \,\Itt - \sigtt),\]
which is precisely \eqref{eq:OB-nodiffusion} upon replacing Oldroyd's $\lambda$ by our $2\lambda$;
in particular, the choice of $\rho_0(\xt) \equiv 1$ yields $\rho(\xt,t)\equiv 1$ for all $(\xt,t)
\in \Omega \times (0,T]$, in agreement with the statement in the sentence preceding Definition
\ref{Qedef}. In summary then, $\Ttt:=\Ttt_s + \Ttt_p$, with $\Ttt_s:= 2\mu_s \,\Dtt$ and $\Ttt_p:=
\frac{\mu_p}{\rho\lambda}(\sigtt - \rho \Itt)$, where $\sigtt$ and $\rho$ are solutions to
the partial differential equations appearing in the previous sentence, with $\rho(x,t)\equiv 1$ being
an admissible special case. For an alternative, thermodynamically consistent, derivation of the
Oldroyd-B model we refer to \cite{MRT}.

\end{remark}

We continue with
a brief literature survey. In our paper \cite{BS2010-hookean} we proved the existence and equilibration of large-data global weak solutions to general noncorotational Hookean-type bead-spring chain models
with stress-diffusion in both two and three space dimensions, under the assumption that the spring potentials appearing in the model exhibit superlinear growth at infinity. We were, however, unable to cover the classical Hookean
dumbbell model, where the spring potential has
linear growth at infinity. Our objective here is to close this gap, in the case of two space dimensions at least.
The relevance of the Hookean dumbbell model is that it has a formal macroscopic closure: the Oldroyd-B model. Lions and Masmoudi \cite{LM}
proved global existence of large-data weak solutions to a corotational Oldroyd-B model (i.e. one where the gradient of the velocity field in
the stress evolution equation is replaced by the skew-symmetric part of the velocity gradient) without stress-diffusion,
in both two and three space dimensions. 
Returning to the general noncorotational case, Hu and Lin \cite{HuLin2016} proved the global existence of weak solutions to incompressible viscoelastic flows, including the Oldroyd-B model, without stress-diffusion, in two spatial dimensions, under the assumption that the initial deformation gradient is close to the identity matrix in $\Ltt^2(\Omega) \cap \Ltt^\infty(\Omega)$
and the initial velocity is small in $\Lt^2(\Omega)$ and bounded in $\Lt^p(\Omega)$ for some $p > 2$.
In \cite{barrett-boyaval-09}, Barrett \& Boyaval proved the existence of large-data global weak solutions to the Oldroyd-B model, in the presence of stress-diffusion, again in two spatial dimensions. Constantin and Kliegl \cite{CK12}  subsequently showed the global regularity of solutions to the Oldroyd-B model with stress-diffusion in two space
dimensions. Motivated by \cite{barrett-boyaval-09} and maximal regularity results for the unsteady Stokes and Navier--Stokes systems (cf. \cite{GigaSohr91}, \cite{Sol01}, \cite{FarwigGigaHsu16}, for example, and references therein), we revisit the classical
Hookean dumbbell model and prove, in the general noncorotational case, in two space dimensions, the existence of large-data global weak solutions. Indirectly, our argument also rigorously proves that, in two space dimensions at least, the Oldroyd-B model with stress-diffusion is the macroscopic closure of the Hookean dumbbell model with centre-of-mass diffusion.
In the case of three dimensions the question of existence of large-data global weak solutions to both the general noncorotational Hookean dumbbell model
and the noncorotational Oldroyd-B model with stress-diffusion remains open, although we will show here the existence of large-data global weak
\textit{subsolutions} to the general noncorotational Hookean dumbbell model for $d=3$.

The paper is structured as follows. In the next section we introduce
our notation and useful results, such as compactness theorems.
Henceforth, we shall frequently write
\[\hpsi:= \psi/M, \qquad \hpsi_0 := \psi_0/M.\]
In Section \ref{OB}, we recall from \cite{barrett-boyaval-09}
a global-in-time existence result for the Oldroyd-B model when $d=2$.
We then establish some regularity results for this solution, $(\utOB, \sigOB)$,
and a uniqueness result
for the stress equation.
In Section \ref{sec:FP}, we prove the existence of a global-in-time weak solution,
$\hpsist$,
to the Fokker--Planck equation, for a given velocity field
$\utst$, via regularization and time discretization. In addition,
we show that $\sigtt(M\,\hpsist)$ solves the corresponding
Oldroyd-B stress equation with given velocity field $\utst$.
In Section \ref{HDM} we combine the results of the previous
two sections to establish the the existence of
a global-in-time weak solution, $(\utOB,\hpsiOB)$, to the Hookean dumbbell
model when $d=2$. Moreover, we show that $(\utOB,\sigtt(M\,\hpsiOB))$
solves the Oldroyd-B model. Finally in Section
\ref{remexist} we explain why we had to resort to the reasoning upon which
our existence proof is based, and why a more direct argument is only capable
of showing, for both $d=2$ and $d=3$, the existence of large-data global weak
\textit{subsolutions} (in a sense to be made precise in Section
\ref{remexist}).


\section{Preliminaries} 
\label{sec:prem}
\setcounter{equation}{0}
Let $\Omega \subset {\mathbb R}^d$, $d=2$ or 3, be a bounded open set with a
Lipschitz-continuous boundary $\partial \Omega$, and
the set  
of elongation vectors $\qt 
\in D  \equiv {\mathbb R}^d$.
Let
\begin{eqnarray}
&\Ht :=\{\wt \in \Lt^2(\Omega) : \nabx \cdot \wt =0, \; (\wt \cdot \nt)|_{\partial \Omega}
=0\} \quad
\mbox{and}\quad \Vt :=\{\wt \in \Ht^{1}_{0}(\Omega) : \nabx \cdot
\wt =0\},&~~~ \label{eqVt}
\end{eqnarray}
where the divergence operator $\nabx\cdot$ is to be understood in
the sense of vector-valued distributions on $\Omega$.
Let $[\Ht^1_0(\Omega)]'$ denote the dual of  $\Ht^1_0(\Omega)$.
We recall the following well-known
Gagliardo--Nirenberg inequality. Let $r \in [2,\infty)$ if $d=2$,
and $r \in [2,6]$ if $d=3$ and $\theta = d \,\left(\frac12-\frac
1r\right)$. Then, there is a constant $C=C(\Omega,r,d)$, 
such that, for all $\eta \in H^{1}(\Omega)$:
\begin{equation}\label{eqinterp}
\|\eta\|_{L^r(\Omega)}
\leq C\,
\|\eta\|_{L^2(\Omega)}^{1-\theta} 
\,\|\eta\|_{H^{1}(\Omega)}^\theta. 
\end{equation}

Let $L^p_M(\Omega \times D)$, $p\in [1,\infty)$,
denote the Maxwellian-weighted $L^p$ space over $\Omega \times D$ with norm
$$
\| \hat \varphi\|_{L^{p}_M(\Omega\times D)} :=
\left\{ \int_{\Omega \times D} \!\!M\,
|\hat \varphi|^p \dq \dx
\right\}^{\frac{1}{p}}.
$$
Similarly, we introduce $L^p_M(D)$,
the Maxwellian-weighted $L^p$ space over $D$. On defining
%
\begin{align}
\| \hat \varphi\|_{H^{1}_M(\Omega\times D)} &:=
\left\{ \int_{\Omega \times D} \!\!M\, \left[
|\hat \varphi|^2 + \left|\nabx \hat \varphi \right|^2 + \left|\nabq \hat
\varphi \right|^2 \,\right] \dq \dx
\right\}^{\frac{1}{2}}\!\!, \label{H1Mnorm}
\end{align}
%
we then set
%
%
\begin{align}
\quad \hat X \equiv H^{1}_M(\Omega \times D)
&:= \left\{ \hat \varphi \in L^1_{\rm loc}(\Omega\times D): \|
\hat \varphi\|_{H^{1}_M(\Omega\times D)} < \infty \right\}. \label{H1M}
\end{align}
%
Similarly, we introduce $H^1_M(D)$, the Maxwellian-weighted $H^1$ space over $D$.
It is shown in Appendix A of \cite{BS2010-hookean} 
that
\begin{align}
C^\infty_0(D)
\mbox{ is dense in } H^1_M(D) \quad \mbox{and hence} \quad
C^{\infty}(\overline{\Omega},C^\infty_0(D))
\mbox{ is dense in } \hat X.
\label{cal K}
\end{align}
In addition, we note that the embeddings
\begin{subequations}
\begin{align}
 H^1_M(D) &\hookrightarrow L^2_M(D) ,\label{wcomp1}\\
H^1_M(\Omega \times D) \equiv
L^2(\Omega;H^1_M(D)) \cap H^1(\Omega;L^2_M(D))
&\hookrightarrow
L^2_M(\Omega \times D) \equiv L^2(\Omega;L^2_M(D))
\label{wcomp2}
\end{align}
\end{subequations}
are compact; 
see Appendix D and Appendix F of \cite{BS2010-hookean}, respectively.

Next, we note 
that (\ref{C1}) and
(\ref{additional-1}) yield, for
$\hat \varphi \in L^2_M(\Omega \times D)$ and any $r \in [0,\infty)$,
that
\begin{subequations}
\begin{align}
\label{eqCttbd}
\int_{\Omega} |\sigtt( M\,\hat \varphi)|^2\,\dx & =
\int_{\Omega} 
\left|
\int_{D} M\,\hat \varphi 
\,\qt\,\qt^{\rm T} 
\dq \right|^2 \dx
\leq 
\left(\int_{D} M\,
|\qt|^4 \,\dq\right)
\left(\int_{\Omega \times D} M\,|\hat \varphi|^2 \dq \dx\right)
\nonumber \\
& \leq C
\,\|\hat \varphi\|^2_{L^2_M(\Omega \times D)}, \\
\left(\int_{\Omega \times D} M\,|\qt|^r\,|\hat \varphi| \dq \dx \right)^2& \leq
\left(\int_{D} M\,|\qt|^{2r}  \dq \right)
\left(\int_{\Omega \times D} M\,|\hat \varphi|^2 \dq \dx\right) \leq C
\,\|\hat \varphi\|^2_{L^2_M(\Omega \times D)}.
\label{qrbd}
\end{align}
\end{subequations}
In addition, the following simple lemma will be useful (cf. Lemma 5.1 in \cite{BS2010} for the proof).

\begin{lemma}\label{le:supplementary}
Suppose that a sequence $\{\hat\varphi_n\}_{n=1}^\infty$ converges in
$L^1(0,T; L^1_{M}(\Omega \times D))$ to $\hat\varphi
\in L^1(0,T; L^1_{M}(\Omega \times D))$,
and is bounded in $L^\infty(0,T;L^1_{M}(\Omega \times D))$, i.e.,
there exists $K_0>0$ such that
$\|\varphi_n\|_{L^\infty(0,T; L^1_{M}(\Omega \times D))} \leq K_0$ for all $n \geq 1$.
Then, $\hat\varphi \in L^p(0,T; L^1_{M}(\Omega \times D))$ for all $p \in [1,\infty)$,
and the sequence $\{\hat\varphi_n\}_{n \geq 1}$ converges to $\hat\varphi$ in
$L^p(0,T; L^1_{M}(\Omega \times D))$ for all $p \in [1,\infty)$.
\end{lemma}

We shall use the symbol $|\cdot|$ to denote the absolute value when the argument is
a real number, the Euclidean norm when the argument is a vector, and the Frobenius
norm when the argument is a square matrix. For a square matrix $\Btt \in  \mathbb{R}^{d\times d}$,
we recall that the symbol $\mathfrak{tr}(\Btt)$ will signify the trace of $\Btt$.

We state a simple integration-by-parts formula (cf. Lemma 3.1 in \cite{BS2010-hookean}
and note that $U' \equiv 1$).

\begin{lemma}
Suppose that $\hat\varphi \in H^1_M(D)$ and let $\Btt \in \mathbb{R}^{d\times d}$ be a square
matrix such that $\mathfrak{tr}(\Btt)=0$; then,
\begin{align}
\int_{D} M(\qt)\,
(\Btt \,\qt) \cdot \nabq \hat\varphi(\qt) \dd \qt &= \int_{D} M(\qt)\,
\hat\varphi(\qt) \, 
\qt \,\qt^{\rm T}
: \Btt \dq.
\label{intbypartsi}
\end{align}
\end{lemma}

Let ${\cal F}\in C(\mathbb{R}_{>0})$ be defined by
\begin{align}
{\cal F}(s):= s\,(\log s -1) + 1, \qquad s>0; \qquad \mathcal{F}(0):=1.
\label{F}
\end{align}
Clearly, $\mathcal{F} \in C([0,\infty))$, and is a nonnegative,
strictly convex function. We note the following result.

\begin{lemma}\label{Fpsilem}
For all $\hat \varphi$ such that ${\cal F}(\hat \varphi) \in L_M^1(\Omega \times D)$
we have that
\begin{align}
\int_{\Omega \times D} M
\left(\frac{1}{2}\,|\qt|^2\right) \hat \varphi \dq \dx &\leq
\frac{2}{c}\, \left[\int_{\Omega \times D} M\,{\cal F}(\hat \varphi) \dq \dx
+ |\Omega| \int_{D} M\, {\rm e}^{\frac{c}{2} (\frac{1}{2}|\qt|^2)}\,
\dq \right],
\label{qt2bd}
\end{align}
where $c=1$.
\end{lemma}
\begin{proof}
See the proof of Lemma 4.1 in \cite{BS2010-hookean}, where it is proved
with $\frac{1}{2}\,|\qt|^2$ and $c$ replaced by
$(\frac{1}{2}\,|\qt|^2)^{\vartheta}$ and $c_{1}$, respectively,
and $M$, recall (\ref{MN}), based on $U$, 
satisfying (\ref{growth1}--c); but the proof given there is
valid for $\vartheta=1$ and hence for $M$ based on (\ref{U}).
\end{proof}

We recall the Aubin--Lions--Simon compactness theorem, see, e.g.,
Temam \cite{Temam} and Simon \cite{Simon}. Let ${\cal B}_0$, ${\cal B}$ and
${\cal B}_1$ be Banach
spaces, where ${\cal B}_i$, $i=0,1$, are reflexive, with a compact embedding ${\cal B}_0
\hookrightarrow {\cal B}$ and a continuous embedding ${\cal B} \hookrightarrow
{\cal B}_1$. Then, for $\alpha_i>1$, $i=0,1$, the embedding
\begin{eqnarray}
&\{\,\eta \in L^{\alpha_0}(0,T;{\cal B}_0): \frac{\partial \eta}{\partial t}
\in L^{\alpha_1}(0,T;{\cal B}_1)\,\} \hookrightarrow L^{\alpha_0}(0,T;{\cal B})
\label{compact1}
\end{eqnarray}
is compact.

We shall also require the following generalization of the Aubin--Lions--Simon compactness
theorem due to Dubinski{\u\i} \cite{DUB}; see also \cite{BS-DUB}.
Before stating the result, we recall the concept of a seminormed set (in the sense of
Dubinski{\u\i}).
A subset $\mathcal{M}$ of a linear space
$\mathcal{A}$ over $\mathbb{R}$ is said to be a seminormed set if
$c\,\varphi \in \mathcal{M}$, for any $c \in [0,\infty)$ and $\varphi \in \mathcal{A}$,
%
%
and there exists a functional
(namely the seminorm of $\mathcal M$), denoted by $[\varphi]_{\mathcal M}$,
such that:
\begin{enumerate}
\item[(i)] $[\varphi]_{\mathcal M} \geq 0$ $\forall \varphi \in \mathcal M$;
and $[\varphi]_{\mathcal M} = 0$ if,
and only if, $\varphi=0$;
\item[(ii)] $[c\, \varphi]_{\mathcal M}
= c\,[\varphi]_{\mathcal M}$ $\forall \varphi \in \mathcal M$, $\forall c \in [0,\infty)$.
\end{enumerate}
A subset $\mathcal{B}$ of a seminormed
set $\mathcal{M}$ is said to be bounded if there exists
a positive constant $K_0$ such that $[\varphi]_{\mathcal M} \leq K_0$
for all $\varphi \in \mathcal{B}$.
A seminormed set $\mathcal{M}$ contained in a normed linear space $\mathcal{A}$
with norm $\|\cdot\|_{\mathcal A}$ is said to be continuously embedded in $\mathcal{A}$,
and we write $\mathcal{M} \hookrightarrow \mathcal{A}$, if
there exists a $K_0 \in \mathbb{R}_{>0}$ such that
$\|\varphi \|_{\mathcal A} \leq K_0 [\varphi]_{\mathcal M}$
for all $\varphi \in \mathcal{M}$.
%
The embedding of a seminormed set $\mathcal{M}$ into a normed linear space
$\mathcal{A}$ is said to be compact if from any bounded infinite set of elements
of $\mathcal{M}$ one can extract a subsequence that converges in $\mathcal{A}$.

\begin{theorem}[Dubinski{\u\i} \cite{DUB}] \label{thm:Dubinski}
Suppose that $\mathcal{A}_0$ and $\mathcal{A}_1$ are Banach spaces, $\mathcal{A}_0
\hookrightarrow \mathcal{A}_1$, and $\mathcal{M}$ is a seminormed subset of
$\mathcal{A}_0$ such that the embedding $\mathcal{M} \hookrightarrow
\mathcal{A}_0$ is compact.
Then, for $\alpha_i>1$, $i=0,1$, the embedding
$$\{\,\eta \in L^{\alpha_0}(0,T;{\cal M}): \frac{\partial \eta}{\partial t}
\in L^{\alpha_1}(0,T;{\cal A}_1)\,\} \hookrightarrow L^{\alpha_0}(0,T;{\cal A}_0)$$
is compact.
\end{theorem}

\section{The Oldroyd-B model}\label{OB}
\setcounter{equation}{0}

We start with noting the following existence result for problem (Q), Oldroyd-B,
with $\rho \equiv 1$.
\begin{theorem}\label{Qex}
Let $d=2$ and $\partial \Omega \in C^{0,1}$. 
In addition, let $\ut_0 \in \Ht$, $\sigtt_0 \in \Ltt^2(\Omega)$ with
$\sigtt_0 = \sigtt_0^{\rm T} \geq 0$ a.e.\ in $\Omega$
and  $\ft \in \Lt^2(0,T;[\Ht^1_0(\Omega)]')$.
Then there exist
\begin{align}
&\utOB \in L^\infty(0,T;\Lt^2(\Omega))\cap L^2(0,T;\Vt)
\quad \mbox{and} \quad
\sigOB \in L^\infty(0,T;\Ltt^2(\Omega))\cap L^2(0,T;\Htt^1(\Omega))
\label{Qus}
\end{align}
with
$\sigOB = \sigOB^{\rm T} \geq 0$ a.e.\ in $\Omega \times (0,T]$
such that
\begin{subequations}
\begin{align}\label{equtOB}
&\displaystyle-\int_{0}^{T}
\int_{\Omega} \utOB \cdot \frac{\partial \wt}{\partial t}
\dx \dt
+ \int_{0}^T \int_{\Omega}
\left[ \left[ (\utOB \cdot \nabx) \utOB \right]\,\cdot\,\wt
+ \nu \,\nabxtt \utOB
:
\nabx \wt \right] \dx \dt
\nonumber
\\
&\qquad = \int_{0}^T
\left[ \langle \ft, \wt\rangle_{H^1_0(\Omega)}
- k\int_\Omega
\sigOB : \nabx
\wt  \dx
\right] \dt
+ \int_\Omega \ut_0(\xt) \cdot \wt(\xt ,0) \dx
\nonumber
\\
& \hspace{1.8in}
\forall \wt \in W^{1,1}(0,T;\Vt) \mbox{ {\rm with} $\wt(\cdot,T)=\zerot$},
\end{align}
where $ \langle \cdot , \cdot \rangle_{H^1_0(\Omega)}$ denotes the duality pairing between $[\Ht^1_0(\Omega)]'$ and $\Ht^1_0(\Omega)$, and
\begin{align}\label{eqsigOB}
&\displaystyle-\int_{0}^{T}
\int_{\Omega} \sigOB : \frac{\partial \xitt}{\partial t}
\dx \dt
+ \int_{0}^T \int_{\Omega}
\left[ \left[ (\utOB \cdot \nabx) \sigOB \right] : \xitt
+ \epsilon \,\nabx \sigOB
::
\nabx \xitt \right] \dx \dt
\nonumber
\\
&\hspace{1in} + \int_{0}^T \int_\Omega
\left[ \frac{1}{2\,\lambda}\,( \sigOB- \Itt) -
\left((\nabxtt \utOB)\,\sigOB + \sigOB\,(\nabxtt \utOB)^{\rm T} \right)
\right] :
\xitt  \dx  \dt \nonumber \\
&\qquad
=
\int_\Omega \sigtt_0(\xt) : \xitt(\xt ,0) \dx
\qquad
\forall \xitt \in W^{1,1}(0,T;\Htt^1(\Omega)) \mbox{ {\rm with} $\xitt(\cdot,T)=\zerott$}.
\end{align}
\end{subequations}
\end{theorem}
\begin{proof}
Existence of a solution to (Q) with $\rho \equiv 1$ is proved in
\cite{barrett-boyaval-09} via a finite element approximation for
a polygonal domain $\Omega \subset {\mathbb R}^2$
under the stronger assumption
$\sigtt_0 \in \Ltt^\infty(\Omega)$ with
$\sigtt_0 = \sigtt_0^{\rm T} > 0$ a.e.\ in $\Omega$.
The restriction on $\Omega$ was purely for ease of exposition
for the finite element approximation.
Existence of a solution to a compressible version of (Q) is proved in
\cite{COMPOB-arxiv} under the stronger assumptions
$C^{2,\mu}$, for $\mu \in (0,1)$,
and
$\ft \in L^\infty(\Omega \times (0,T))$,
with $\ut_0 \in \Lt^2(\Omega)$.
The proof there is easily adapted
to the far simpler incompressible case for $\partial \Omega \in C^{0,1}$
and $\ft \in \Lt^2(0,T;[\Ht^1_0(\Omega)]')$, with $\ut_0 \in \Ht$.
For example, in the existence proof for
the compressible version of (Q) discussed in \cite{COMPOB-arxiv} the assumption
$\partial \Omega \in C^{2,\mu}$ is only needed because one requires high regularity
for a parabolic Neumann problem; see Lemma \ref{lem-para} below, which, however, is
not needed in the existence proof for (Q) in the incompressible case.
\end{proof}

\begin{remark} \label{precrem}{\rm
Since the test functions in $\Vt$ are divergence-free, the
pressure has been eliminated in (\ref{equtOB});
it can be recovered in a very weak sense following the same
procedure as for the incompressible Navier--Stokes equations
discussed on p.~\!208 in \cite{Temam}; i.e., one obtains that
$\int_0^t p_{OB}(\cdot,t')\,{\rm d}t' \in C([0,T];L^2(\Omega))$.}
\end{remark}

We now deduce additional regularity for this Oldroyd-B model.

\begin{lemma}
\label{OBlem} Let $d=2$ and suppose that the hypotheses of Theorem \ref{Qex} hold; then,
\begin{align}
&\utOB \in L^4(0,T;\Lt^4(\Omega)), \quad \sigOB \in L^4(0,T;\Ltt^4(\Omega))
\quad (\utOB\cdot \nabx) \utOB
\in L^{\frac{4}{3}}(0,T;\Lt^{\frac{4}{3}}(\Omega))
\nonumber \\
 \mbox{and} \qquad
& (\utOB\cdot \nabx) \sigOB,\,(\nabxtt \utOB)\, \sigOB
\in L^{\frac{4}{3}}(0,T;\Ltt^{\frac{4}{3}}(\Omega)).
\label{Qus1}
\end{align}
\end{lemma}
\begin{proof}
The first two results in (\ref{Qus1}) follow directly from (\ref{Qus}) and (\ref{eqinterp}).
The remaining results in (\ref{Qus1}) follow directly from the first two results in (\ref{Qus1})
and (\ref{Qus}) using H\"{o}lder's inequality.
\end{proof}

In order to improve the regularity results (\ref{Qus}) and (\ref{Qus1}),
we consider the parabolic initial-boundary-value
problem:
\begin{subequations}
\begin{alignat}{2}\label{para1}
\frac{\partial \eta}{\partial t} - \epsilon\, \Delta_x \eta &= g
\qquad &&\mbox{in} \ \Omega \times (0,T],\\
\epsilon\,\nabx \eta\cdot \nt  &=0 \qquad &&\mbox{on} \ \partial \Omega \times (0,T],
\label{para2}\\
\eta(\cdot,0) &= \eta_0(\cdot)  \qquad  &&\mbox{on} \   \Omega,
\label{para3}
\end{alignat}
\end{subequations}
where $\epsilon \in \mathbb{R}_{>0}$.
We require the following definitions to state a regularity result for (\ref{para1}--c).
First we introduce fractional-order Sobolev spaces.
For any $k\in \mathbb{N}$, $\b \in (0,1)$ and $s\in (1,\infty) $,
we define
\[
 W^{k+\b,s}(\Omega):=\left\{\zeta\in W^{k,s}(\Omega) : \| \zeta \|_{W^{k+\b,s}(\Omega)}<\infty \right\}\!,
\]
where
$$
\| \zeta \|_{W^{k+\b,s}(\Omega)}:=\| \zeta \|_{W^{k,s}(\Omega)} + \sum_{|\alpha|=k}
\left(\int_\Omega\int_\Omega
\frac{|\partial^\alpha_x \zeta(\xt)- \partial^\alpha_y \zeta(\yt)|^s}{|\xt-\yt|^{d+\b s}} \,
\dx \,{\rm d}\yt\right)^{\frac{1}{s}}
$$
and $\partial^\alpha_x = \partial^\alpha/\partial_{x_1}^{\alpha_1}
\ldots \partial_{x_d}^{\alpha_d}$.
We then define $W^{2-\frac{2}{r},s}_{n}(\Omega)$, for $r,\, s \in (1,\infty)$,
to be the completion of
$\{\zeta \in  C^\infty(\overline \Omega): \nabx \zeta \,\cdot\, \nt = 0 \ \mbox{on } \partial \Omega\}$
 in the norm of $W^{2-\frac{2}{r},s}(\Omega)$.
We now recall the following regularity result for (\ref{para1}--c); see
e.g.\ Lemma 7.37 in \cite{NovStras}.

\begin{lemma}\label{lem-para}
Let $\Omega \subset {\mathbb R}^d$ with $\partial \Omega \in C^{2,\mu}$, for $\mu \in (0,1)$.
In addition, let
$\eta_0 \in W^{2-\frac{2}{r},s}_{n}(\Omega)$ and $g\in L^r(0,T;L^s(\Omega))$,
for $r,\,s \in (1,\infty)$.
Then, there exists a unique function
$$
\eta \in  C([0,T];W^{2-\frac{2}{r},s}(\Omega))
\cap L^r(0,T;W^{2,s}(\Omega))
\cap W^{1,r}(0,T;L^s(\Omega))
$$
solving (\ref{para1}--c).
Here (\ref{para2}) is satisfied in the sense of the normal trace,
which is well defined since $\Delta_x \eta \in L^r(0,T;L^{s}(\Omega))$.
Moreover, we have that
\begin{align*}
\|\eta\|_{L^\infty(0,T;W^{2-\frac{2}{r},s}(\Omega))}
+ 
\|\eta\|_{L^r(0,T; W^{2,s}(\Omega))}
+ \left\|\frac{\partial \eta}{\partial t} \right\|_{L^r(0,T; L^s(\Omega))}
\nonumber\\
\leq C(\epsilon,r,s,\Omega)\left[ 
\|\eta_0\|_{W^{2-\frac{2}{r},s}(\Omega)} + \|g\|_{L^r(0,T;L^s(\Omega))}\right]\!.
\end{align*}
\end{lemma}

We now apply Lemma \ref{lem-para} to the stress equation (\ref{eqsigOB}).
\begin{lemma}\label{lem-OBspara}
Let $d=2$, $\partial \Omega \in C^{2,\mu}$, for $\mu \in (0,1)$,  and $\sigtt_0 \in
\Wtt_{n}^{\frac{1}{2},\frac{4}{3}}(\Omega)$
with
$\sigtt_0 = \sigtt_0^{\rm T} \geq 0$ a.e.\ in $\Omega$. Then we have that
\begin{align}
\!\!\!\sigOB \in L^{\frac{4}{3}}(0,T;\Wtt^{2,\frac{4}{3}}(\Omega))
\Rightarrow \nabx \cdot \sigOB \in L^{\frac{4}{3}}(0,T;\Wt^{1,\frac{4}{3}}(\Omega))
\Rightarrow \nabx \cdot \sigOB \in L^{\frac{4}{3}}(0,T;\Lt^{4}(\Omega)).
\label{Qus2}
\end{align}
\end{lemma}
\begin{proof}
Applying Lemma \ref{lem-para} with $r=s=\frac{4}{3}$ to each component of
(\ref{eqsigOB}) and
noting (\ref{Qus1}) yields the first, and hence the second, result in (\ref{Qus2}).
The final result in (\ref{Qus2}) follows from the second result and Sobolev embedding as $d=2$.
\end{proof}

To improve the regularity results (\ref{Qus}), (\ref{Qus1}) and (\ref{Qus2}) further,
we
now consider the Stokes initial-boundary-value
problem for $\nu \in \mathbb{R}_{>0}$:
\begin{subequations}
\begin{alignat}{2}\label{stokes1}
\frac{\partial \vt}{\partial t} - \nu\,\Delta_x \vt + \nabx \pi&= \gt,
\qquad \nabx \cdot \vt = 0 \qquad &&\mbox{in} \ \Omega \times (0,T],\\
\vt  &=\zerot \qquad &&\mbox{on} \ \partial \Omega \times (0,T],
\label{stokes2}\\
\vt(\cdot,0) &= \vt_0(\cdot)  \qquad  &&\mbox{on} \   \Omega;
\label{stokes3}
\end{alignat}
\end{subequations}
and
the Navier--Stokes initial-boundary-value
problem, where (\ref{stokes1}) is replaced by
\begin{alignat}{2}\label{temam1}
\frac{\partial \vt}{\partial t} + (\vt \cdot \nabx) \vt - \nu\,\Delta_x \vt + \nabx \pi&= \gt,
\qquad \nabx \cdot \vt = 0 \qquad &&\mbox{in} \ \Omega \times (0,T].
\end{alignat}

In order to state a regularity result for (\ref{stokes1}--c),
we require the following definitions.
The first is a generalisation of $\Ht$, let
$\Ltdiv^s(\Omega)$ be the completion of $
\{
\wt \in \Ct^\infty_0(\Omega) : \nabx \cdot \wt = 0 \mbox{ in } \Omega \}$
in $\Lt^s(\Omega)$ for $s \in (1,\infty)$. So $\Ht \equiv \Ltdiv^2(\Omega)$.
Next, we introduce, for $\alpha \in (0,1)$ and $r,\,s \in (1,\infty)$,
\begin{align}
\Dt^{\alpha,r}_s(\Omega)
:= \left \{ \wt \in \Ltdiv^s(\Omega) :
\|\wt\|_{D^{\alpha,r}_s(\Omega)} := \|\wt\|_{L^s(\Omega)} +
\left( \int_0^\infty \|t^{1-\alpha}\,A_s\,{\rm e}^{-t\,A_s}\, \wt \|^r_{L^s(\Omega)}\,t^{-1} \dt
\right)^{\frac{1}{r}} <\infty
\right\},
\label{Dt}
\end{align}
where $A_s= -P_s\,\Delta$ is the Stokes operator with domain $D(A_s)=
\Ltdiv^s(\Omega)\cap \Wt^{1,s}_0(\Omega) \cap \Wt^{2,s}(\Omega)$
and $P_s:\Lt^s(\Omega) \rightarrow \Ltdiv^s(\Omega)$ is the Helmholtz
projection, see \cite{GigaSohr91} for details and Remark \ref{remHD} below.
We now recall the following regularity result for (\ref{stokes1}--c),
see Theorem 2.8 in \cite{GigaSohr91}.

\begin{lemma}\label{lem-stokes}
Let $\Omega \subset {\mathbb R}^d$ with $\partial \Omega \in C^{2,\mu}$, for $\mu \in (0,1)$.
In addition,  let
$\vt_0 \in \Dt^{1-\frac{1}{r},r}_{s}(\Omega)$ and $\gt
\in L^r(0,T;\Ltdiv^s(\Omega))$, for $r,\,s \in (1,\infty)$.
Then, there exist unique functions $\vt$ and $\nabx \pi$  satisfying
$$
\vt \in L^r(0,T;\Wt^{2,s}(\Omega))\cap W^{1,r}(0,T;\Lt^{s}(\Omega)),  \qquad
\nabx \pi \in L^{r}(0,T;\Lt^s(\Omega))
$$
and solving (\ref{stokes1}--c).
Moreover, we have that
\begin{align*}
\,\|\vt\|_{L^r(0,T;W^{2,s}(\Omega))}
+ \left\|\frac{\partial \vt}{\partial t} \right\|_{L^r(0,T; L^s(\Omega))}
+ \|\nabx \pi\|_{L^r(0,T; L^{s}(\Omega))}  \nonumber\\
\leq C(\nu,r,s,\Omega)\left[
\|\vt_0\|_{D^{1-\frac{1}{r},r}_s(\Omega)} + \|\gt\|_{L^r(0,T;L^s(\Omega))}\right].
\end{align*}
\end{lemma}


In addition, we recall the following regularity result for (\ref{temam1}), (\ref{stokes2},c),
see Theorem 3.10 on p.213 in \cite{Temam}.

\begin{lemma}\label{lem-temam}
Let $d=2$, $\partial \Omega \in C^2$, $\vt_0 \in \Vt$ and $\gt \in L^2(0,T;\Ltdiv^2(\Omega))$.
Then, there exists a function
$$
\vt \in L^\infty(0,T;\Vt) \cap L^2(0,T;\Ht^2(\Omega))\cap H^1(0,T;\Ht)
$$
solving (\ref{temam1}), (\ref{stokes2},c).
Moreover, we have that
\begin{align*}
&\|\vt\|_{L^\infty(0,T;H^{1}(\Omega))}
+\|\vt\|_{L^2(0,T;H^{2}(\Omega))}
+ \left\|\frac{\partial \vt}{\partial t} \right\|_{L^2(0,T; L^2(\Omega))}
\!\! \leq C(\nu,\Omega)\left[
\|\vt_0\|_{H^1(\Omega)} + \|\gt\|_{L^2(0,T;L^2(\Omega))}\right].
\end{align*}
\end{lemma}

\begin{remark} \label{remHD}
{\rm
We note that Lemmas \ref{lem-stokes} and \ref{lem-temam}
can be applied to $\gt \in L^r(0,T;\Lt^s(\Omega))$ (with $r,\,s \in (1,\infty)$
in the case of Lemma \ref{lem-stokes}, and $r=s=2$ in the case of Lemma \ref{lem-temam}),
by using the Helmholtz decomposition of $\gt$ and adjusting the pressure,
since any such $\gt$ can be uniquely decomposed
as
$$ \gt = \gt_0 + \nabx \eta, \quad \mbox{where}\quad \gt_0 \in L^r(0,T;\Ltdiv^s(\Omega))
\quad \mbox{and} \quad \nabx \eta \in L^r(0,T;\Lt^s(\Omega)).$$}
\end{remark}

On applying Lemmas \ref{lem-stokes} and \ref{lem-temam}
to the flow equation (\ref{equtOB}), we have the following result.

\begin{lemma}\label{lem-OBuGS1}
Let the assumptions of Lemma \ref{lem-OBspara} hold.
In addition let
$\ut_0 \in \Vt \cap
\Dt_{\mathfrak{s}}^{1-\frac{1}{\mathfrak{r}},\mathfrak{r}}(\Omega)$ and
$\ft \in L^2(0,T;\Lt^2(\Omega))$ $\cap
L^\mathfrak{r}(0,T;\Lt^\mathfrak{s}(\Omega))$,
for $\mathfrak{r} \in (1,\frac{4}{3}]$ and $\mathfrak{s} \in (2,4)$.
Then,
we have that
\begin{align}
&\utOB \in L^{\infty}(0,T;\Vt)\cap L^{2}(0,T;\Ht^{2}(\Omega))\cap H^1(0,T;\Ht)
\nonumber \\
&\hspace{1in}\Rightarrow \nabxtt \utOB \in L^{\infty}(0,T;\Ltt^{2}(\Omega))\cap L^{2}(0,T;\Htt^{1}(\Omega))
\nonumber \\
&\hspace{1in}\Rightarrow \nabxtt \utOB \in L^{4}(0,T;\Ltt^{4}(\Omega)).
\label{Qus4}
\end{align}
In addition, (\ref{Qus4}) yields, for any $\mathfrak{y} \in [1,\infty)$
and any $\mathfrak{z} \in [1,4)$, that
\begin{align}
\utOB \in L^\infty(0,T;\Lt^{\mathfrak{y}}(\Omega))
\qquad\mbox{and} \qquad
(\utOB \cdot \nabx) \utOB  \in L^4(0,T;\Lt^\mathfrak{z}(\Omega)).
\label{Qus5}
\end{align}
Moreover, we have that
\begin{align}
\utOB \in L^{\mathfrak{r}}(0,T;\Wt^{2,\mathfrak{s}}(\Omega))
&\Rightarrow \utOB \in L^{\mathfrak{r}}(0,T;\Wt^{1,\infty}(\Omega)).
\label{Qus6}
\end{align}
\end{lemma}
\begin{proof}
Applying Lemma \ref{lem-temam} to the flow equation (\ref{equtOB}),
on noting (\ref{Qus}) and Remark \ref{remHD}, yields the first result in (\ref{Qus4}).
The second result in (\ref{Qus4}) follows immediately from the first.
The third result in (\ref{Qus4}) follows from the second and (\ref{eqinterp}).
As $d=2$, it follows from the first result in (\ref{Qus4}) and Sobolev embedding that
the first result in (\ref{Qus5}) holds. The
second result in (\ref{Qus5}) follows from this and the last result in (\ref{Qus4}).

Applying Lemma \ref{lem-stokes} with $r=\mathfrak{r}$ and $s=\mathfrak{s}$ to the
flow equation (\ref{equtOB}),
on noting (\ref{Qus2}), (\ref{Qus5}) and Remark \ref{remHD},
yields the first result in (\ref{Qus6}).
The second result in (\ref{Qus6}) follows from the first and Sobolev embedding as $d=2$.
\end{proof}

\begin{theorem} \label{thOBreg}
Let $d=2$, $\partial \Omega \in C^{2,\mu}$, for $\mu \in (0,1)$,
$\ut_0 \in \Vt \cap
\Dt_{\mathfrak{s}}^{1-\frac{1}{\mathfrak{r}},\mathfrak{r}}(\Omega)$
and
$\ft \in L^2(0,T;\Lt^2(\Omega)) \cap
L^\mathfrak{r}(0,T;\Lt^\mathfrak{s}(\Omega))$,
for $\mathfrak{r} \in (1,\frac{4}{3}]$ and $\mathfrak{s} \in (2,4)$,
and $\sigtt_0  = \sigtt^{\rm T}_0  \in
\Wtt^{1,2}_n(\Omega)$.
Then,
\begin{subequations}
\begin{align}
&\ut_{\rm OB} \in L^\infty(0,T;\Vt)
\cap L^2(0,T;\Ht^2(\Omega))
\cap L^{\mathfrak{r}}(0,T;\Wt^{1,\infty}(\Omega))
\cap H^1(0,T;\Ht) \label{utOBregf}\\
\mbox{and} \qquad
&\sigOB \in C([0,T];\Htt^1(\Omega)) \cap L^2(0,T;\Htt^2(\Omega)) \cap H^1(0,T;\Ltt^2(\Omega))
\label{sigOBregf}
\end{align}
\end{subequations}
solves (\ref{equtOB},b).

Moreover, for given $\utOB$ satisfying (\ref{utOBregf}), the solution $\sigOB$ to (\ref{eqsigOB}) is unique.
\end{theorem}
\begin{proof}
The result (\ref{utOBregf}) for $\utOB$ follows immediately from (\ref{Qus4}) and (\ref{Qus6}).
It follows from
(\ref{Qus1}), (\ref{Qus4}),
(\ref{Qus5}) and  (\ref{Qus}) that
$(\nabxtt \utOB)\, \sigOB
\in L^{2}(0,T;\Ltt^{2}(\Omega))$ and
$(\utOB\cdot \nabx) \sigOB \in L^{2}(0,T;\Ltt^{\mathfrak{z}}(\Omega))$  for any
$\mathfrak{z} \in [1,2)$.
Applying Lemma \ref{lem-para} with $r=2$ and $s=\mathfrak{z}$ to each component of
(\ref{eqsigOB}) and
noting the above yields that
$\sigOB \in C([0,T];\Wtt^{1,\mathfrak{z}}(\Omega)) \cap L^2(0,T;\Wtt^{2,\mathfrak{z}}(\Omega))
\cap H^1(0,T;\Ltt^\mathfrak{z}(\Omega))$.
From this and Sobolev embedding, as $d=2$, we obtain that
$\sigOB \in L^2(0,T,\Wtt^{1,\mathfrak{y}}(\Omega))$ for any $\mathfrak{y} \in [1,\infty)$.
Hence, combining this with (\ref{Qus5}) we now have that
$(\utOB\cdot \nabx) \sigOB \in L^{2}(0,T;\Ltt^{2}(\Omega))$.
Applying Lemma \ref{lem-para} again with $r=s=2$ now to each component of
(\ref{eqsigOB}) yields the desired result (\ref{sigOBregf}).

If there existed another solution $\sigOB' \in L^\infty(0,T;\Ltt^2(\Omega))\cap
L^2(0,T;\Htt^1(\Omega))$ to (\ref{eqsigOB}) for given $\utOB$ satisfying
(\ref{utOBregf}), then, as in Lemma 7.1 and the above,
one could establish that $\sigOB'$ satisfies the same regularity
as $\sigOB$ in (\ref{sigOBregf}). Hence the difference
$\ztt= \sigOB-\sigOB'$ satisfies
\begin{align}\label{eqztt}
&\displaystyle\int_{0}^{T}
\int_{\Omega} \frac{\partial \ztt}{\partial t} : \xitt
\dx \dt
+ \int_{0}^T \int_{\Omega}
\left[ \left[ (\utOB \cdot \nabx) \ztt \right] : \xitt
+ \epsilon \,\nabx \ztt
::
\nabx \xitt \right] \dx \dt
\nonumber
\\
&\quad + \int_{0}^T \int_\Omega
\left[ \frac{1}{2\,\lambda}\,\ztt -
\left((\nabxtt \utOB)\,\ztt + \ztt\,(\nabxtt \utOB)^{\rm T} \right)
\right] :
\xitt  \dx  \dt = 0
\quad
\forall \xitt \in L^2(0,T;\Htt^1(\Omega))
\end{align}
and $\ztt(\cdot,0)=\zerott$.
Choosing $\xitt=\chi_{[0,t]}\,\ztt$, where $\chi_{[0,t]}$ denotes the characteristic
function of the interval $[0,t]$,  in the above yields
for all $t \in [0,T]$ that
\begin{align}
&\frac{1}{2}\,\| \ztt(t)\|_{L^2(\Omega)}^2 + \int_0^t
\left[ \epsilon \,\| \nabx \ztt(t')\|_{L^2(\Omega)}^2
+ \frac{1}{2\,\lambda}\,\|
\ztt(t')\|_{L^2(\Omega)}^2
\right] {\rm d}t'
\nonumber \\
& \qquad =  \int_0^t \int_\Omega
\left[(\nabxtt \utOB)\,\ztt + \ztt\,(\nabxtt \utOB)^{\rm T}
\right] :  \ztt \dx\,{\rm d}t'
\leq 2\int_0^t |\utOB(t')|_{W^{1,\infty}(\Omega)}\,\|\ztt(t')\|^2_{L^2(\Omega)}
\,{\rm d}t'.
\label{sigOBuniq}
\end{align}
Applying a Gr\"{o}nwall inequality yields that $\ztt \equiv \zerott$, and hence the required uniqueness
result.
\end{proof}

\begin{remark}\label{Solrem}
{\rm As is noted in \cite{GigaSohr91}, Solonnikov proved in \cite{Sol01} that, for $\Omega \in C^{2,1}$,
$\Dt_{\mathfrak{s}}^{1-\frac{1}{\mathfrak{r}},\mathfrak{r}}(\Omega)$
is the completion of $D(A_\mathfrak{s})$ in $\Wt^{2-\frac{2}{\mathfrak{r}},\mathfrak{s}}(\Omega)$
when $\mathfrak{r}=\mathfrak{s}\neq \frac{3}{2}$ and $d=3$. For the characterization of $\Dt_{\mathfrak{s}}^{1-\frac{1}{\mathfrak{r}},\mathfrak{r}}(\Omega)$ in terms
of Besov spaces of divergence-free vector functions, we refer to the  Appendix in \cite{FarwigGigaHsu16}, where $d=3$ and
$\Omega \in C^{2,1}$, and
Theorem 3.4 in Amann \cite{Amann00}, where $d \geq 2$ and $\Omega \in C^2$. Specifically, $\Dt_{\mathfrak{s}}^{1-\frac{1}{\mathfrak{r}},\mathfrak{r}}(\Omega) 
= \Bt^{2-\frac{2}{\mathfrak{r}}}_{\mathfrak{s},\mathfrak{r},0}(\Omega) \cap \Ltdiv^\mathfrak{s}(\Omega)$, where
$\Bt^{2-\frac{2}{\mathfrak{r}}}_{\mathfrak{s},\mathfrak{r},0}(\Omega) := \{\wt \in \Bt^{2-\frac{2}{\mathfrak{r}}}_{\mathfrak{s},\mathfrak{r}}(\Omega)\,:\, \wt|_{\partial\Omega} = \zerot\}$ for $2 -\frac{2}{\mathfrak{r}}>\frac{1}{\mathfrak{s}}$; $\Bt^{2-\frac{2}{\mathfrak{r}}}_{\mathfrak{s},\mathfrak{r},0}(\Omega) := \{\wt \in \Bt^{2-\frac{2}{\mathfrak{r}}}_{\mathfrak{s},\mathfrak{r}}(\mathbb{R}^d)\,:\, \mbox{supp }(\wt)\subset \overline{\Omega} \}$ for $2 -\frac{2}{\mathfrak{r}}=\frac{1}{\mathfrak{s}}$; and $\Bt^{2-\frac{2}{\mathfrak{r}}}_{\mathfrak{s},\mathfrak{r},0}(\Omega) := \Bt^{2-\frac{2}{\mathfrak{r}}}_{\mathfrak{s},\mathfrak{r}}(\Omega)$ for $0 \leq 2 -\frac{2}{\mathfrak{r}}<\frac{1}{\mathfrak{s}}$, with $\mathfrak{r}, \mathfrak{s} \in (1,\infty)$.
We note, for example,
that $\Bt^r_{s,s}(\Omega)=\Wt^{r,s}(\Omega)$, for $s \in (1, \infty)$ and fractional $r>0$,
and $\Bt^r_{2,2}(\Omega)=\Ht^{r}(\Omega)$, for $r \in {\mathbb N}$ (cf.\ Triebel \cite{T},
Sec.\ 4.4.1, Remark 2 and  Sec.\ 4.6.1, Theorem (b)).
 Consequently, for $\mathfrak{r} = \mathfrak{s}=2$ we have that $\Dt_{2}^{\frac{1}{2},2}(\Omega) = \Bt^{1}_{2,2,0}(\Omega)\cap \Ltdiv^2(\Omega) = \{\wt \in \Bt^{1}_{2,2}(\Omega)\,:\, \wt|_{\partial\Omega} = \zerot\}\cap \Ltdiv^2(\Omega) = \Ht^1_0(\Omega) \cap  \Ltdiv^2(\Omega) = \Vt$. As $\Bt^{1}_{2,2}(\Omega) \hookrightarrow \Bt^{1}_{2,\infty}(\Omega) \hookrightarrow \Bt^{1-\epsilon}_{2,1}(\Omega)
\hookrightarrow \Bt^{1-\epsilon}_{2,\mathfrak{r}}(\Omega)
\hookrightarrow \Bt^{2-\frac{2}{\mathfrak{r}}}_{\mathfrak{s},\mathfrak{r}}(\Omega)$ for all $\mathfrak{r},\mathfrak{s} \in (1,\infty)$ such that $1<\mathfrak{r}\leq \frac{2}{1+\epsilon}$ and $\frac{d}{\mathfrak{s}} + \frac{2}{\mathfrak{r}}\geq 1+ \epsilon + \frac{d}{2}$, $\epsilon \in (0,1)$ (cf.\ \cite{T}, Sec.\ 4.6.1, Theorem (a) and (c)),
it follows that $\Vt = \Dt_{2}^{\frac{1}{2},2}(\Omega) = \Bt^{1}_{2,2,0}(\Omega)\cap \Ltdiv^2(\Omega)
\subset \Bt^{2-\frac{2}{\mathfrak{r}}}_{\mathfrak{s},\mathfrak{r},0}(\Omega) \cap \Ltdiv^2(\Omega)
= \Bt^{2-\frac{2}{\mathfrak{r}}}_{\mathfrak{s},\mathfrak{r},0}(\Omega) \cap (\Lt^{\mathfrak{s}}(\Omega)\cap \Ltdiv^2(\Omega))
= \Bt^{2-\frac{2}{\mathfrak{r}}}_{\mathfrak{s},\mathfrak{r},0}(\Omega) \cap \Ltdiv^{\mathfrak{s}}(\Omega)$
(cf. Theorem III.2.3 in Galdi \cite{Galdi} for the final equality), and therefore $\Vt \subset \Dt_{\mathfrak{s}}^{1-\frac{1}{\mathfrak{r}},\mathfrak{r}}(\Omega)$ for all such
$\mathfrak{r}, \mathfrak{s}$. Thus, in particular, $\Vt \subset \Dt_{\mathfrak{s}}^{1-\frac{1}{\mathfrak{r}},\mathfrak{r}}(\Omega)$ for $d=2$, $\mathfrak{r} \in (1,2)$, and $\frac{1}{\mathfrak{r}} + \frac{1}{\mathfrak{s}}>1$; hence, for $d=2$ also $\Vt \cap
\Dt_{\mathfrak{s}}^{1-\frac{1}{\mathfrak{r}},\mathfrak{r}}(\Omega) = \Vt$
for $\mathfrak{r} \in (1,\frac{4}{3}]$ and $\mathfrak{s} \in (2,4)$.
}
\end{remark}

\section{The Fokker--Planck equation} \label{sec:FP}
\setcounter{equation}{0}
On setting $\ut=\utst \in L^\infty(0,T;\Vt)\cap
L^{1}(0,T;\Wtt^{1,\infty}(\Omega))$ in (\ref{fp0}),
we now want to prove the existence of a weak solution, $\psi=\psi_{\star}= M\, \hpsist$, to this
Fokker--Planck equation.
We restate this as the following problem.

\vspace{2mm}

{\boldmath $({\rm FP})$}: Find $\hpsist : (\xt,\qt,t) \mapsto \hpsist(\xt,\qt,t)$
such that
\begin{align}\label{hpsiOBeq}
&M\left(\frac{\partial \hpsist}{\partial t} + (\utst \cdot\nabx) \hpsist\right) +
\nabq
\cdot \left((\nabxtt \utst) \, \,\qt\,
M\,\hpsist \right)
=
\epsilon\,M\,\Delta_x\,\hpsist +
\frac{1}{4 \,\lambda}\,
\nabq \cdot \left(
M\,\nabq \hpsist\right)
\end{align}
subject to the following decay/boundary and initial conditions:
\begin{subequations}
\begin{alignat}{2}
&\left|M \left[ \frac{1}{4\,\lambda}
\,\nabq \hpsist
- (\nabxtt \utst) \,\qt\,\hpsist
\right] \right|
\rightarrow 0\quad \mbox{as} \quad |\qt| \rightarrow \infty
\qquad
&&\mbox{on }
\Omega
\times (0,T],
\label{hpsiOBa}\\
&\epsilon\,\nabx \hpsist \cdot \nt =0 \qquad &&\mbox{on }
\partial \Omega \times D\times (0,T],\label{hpsiOBb}\\
&\hpsist(\cdot,\cdot,0)=\hpsi_{0}(\cdot,\cdot)
\qquad&&\mbox{on $\Omega\times D$}.\label{hpsiOBc}
\end{alignat}
\end{subequations}

We shall assume that
\begin{align}
&\hpsi_0 \in L^2_M(\Omega \times D)
\quad \mbox{with} \quad
\hpsi_0
\geq 0 \mbox{ a.e.\ on } \Omega \times D, \quad \nonumber  \\
\mbox{and} \qquad
&[\rho(M\,\hpsi_0)](\xt)=\int_D M(\qt)\,\hpsi_0(\xt,\qt) \dq = 1 \mbox{ for a.e. } \xt \in \Omega.
\label{hpsi0}
\end{align}
In addition, we shall assume in this section that
\begin{align}
\Omega \subset \mathbb{R}^d, \ \mbox{$d=2$ or 3},
\mbox{ with }\partial \Omega \in C^{0,1}
\quad \mbox{and} \quad
\utst \in L^\infty(0,T,\Vt) \cap L^{1}(0,T,\Wt^{1,\infty}(\Omega)).
\label{hpsidata}
\end{align}

\subsection{A discrete-in-time regularized problem, {\boldmath $({\rm FP}_L^{\Delta t})$}}

Similarly to \cite{BS2011-fene} and \cite{BS2010-hookean},
in order to prove existence of a weak solution to (FP),
we consider a discrete-in-time approximation, (FP$_L^{\Delta t}$), of a regularization of (FP)
based on the parameter $L > 1$,
where the drag term, i.e.\ the term involving $\nabxtt \utst$, in (\ref{hpsiOBeq})
and the corresponding term in (\ref{hpsiOBa})
are modified using the cut-off function
$\beta^L \in C({\mathbb R})$ defined as
\begin{align}
\beta^L(s) := \min(s,L) = \left\{\begin{array}{ll}
s \qquad & \mbox{for $s \leq L$},
\\ L \qquad & \mbox{for $s \geq L$}.
\end{array} \right.
\label{betaLa}
\end{align}
The weak formulation of the regularization of (FP) leads to the following problem involving the cut-off function $\beta^L$.

{\boldmath $({\rm FP}_L)$}: Find $\hpsistL \in L^\infty(0,T;L^2_M(\Omega\times D))
\cap L^2(0,T;\hat X)$ 
such that
\begin{align}\label{hpsiOBLeq}
&-\int_{0}^T \int_{\Omega \times D}
M\, \hpsistL \frac{ \partial\hat \varphi}{\partial t}\,
 \dq \dx \dt
+ \int_{0}^T \int_{\Omega \times D} M\left[
\epsilon\, \nabx \hpsistL
- \utst\,\hpsistL \right]\cdot \nabx
\hat \varphi
\,\dq \dx \dt
\nonumber \\
& \quad  +
\int_{0}^T \int_{\Omega \times D} M\left[\frac{1}{4\,\lambda} \nabq \hpsistL - [\,(\nabxtt \utst)
\,\qt]\,
\beta^L(\hpsistL)\right] \cdot \nabq
\hat \varphi \,
\dq \dx \dt = \int_{\Omega \times D} M\,\beta^L(\hpsi_0)\, \hat \varphi
\dq \dx
\nonumber \\
& \hspace{2.8in}
\qquad \mbox{$\forall \hat \varphi \in W^{1,1}(0,T;\hat X)$ with  $\varphi(\cdot,\cdot,T)=0$.}
\end{align}


We now formulate our discrete-in-time  approximation of (FP$_{L}$).
We set, for   $n=1, \dots,  N$,
\begin{align}
\utst^{\Delta t, +}(\cdot,t) =
\utst^n(\cdot) := \frac{1}{\Delta t}\,\int_{t_{n-1}}^{t_n}
\utst(\cdot,t) \dt \in \Vt\cap\Wt^{1,\infty}(\Omega),
\qquad t \in (t_{n-1}, t_n].
\label{utOBn}
\end{align}
It follows from (\ref{hpsidata}) and (\ref{utOBn}) that
\begin{subequations}
\begin{align}
&\|\utst^{\Delta t, +}\|_{L^\infty(0,T;H^1(\Omega))}
\leq \|\utst\|_{L^\infty(0,T;H^1(\Omega))}
\label{utOBDtsatb}\\
\mbox{and} \qquad
&\utst^{\Delta t, +} \rightarrow \utst \quad \mbox{strongly in }
L^{1}(0,T;\Wt^{1,\infty}(\Omega))
\mbox { as } \Delta t \rightarrow 0_{+}.
\label{utOBncon}
\end{align}
\end{subequations}

Next, we shall assign a certain `smoothed' initial
datum,
\[\hpsi^0 = \hpsi^0(L,\Delta t) \in H^1_M(\Omega \times D),\]
to the given initial datum $\hpsi_0$ such that
\begin{align}
&\int_{\Omega \times D} M \left[ \widehat \psi^0\, \widehat \varphi +
\Delta t\, \left( \nabx \widehat \psi^0 \cdot \nabx \widehat \varphi +
\nabq \widehat \psi^0 \cdot \nabq \widehat \varphi
\right) \right] \dq \dx
= \int_{\Omega \times D} M \,\beta^{L}(\widehat \psi_0)\, \widehat \varphi \dq \dx
\nonumber \\
& \hspace{4in} \forall \widehat \varphi \in
H^1_M(\Omega \times D).
\label{psi0}
\end{align}
For $r\in [1,\infty)$, let
\begin{align}
\hat Z_r
:= \{ \hat \varphi \in L^r_M(\Omega \times D) :
\hat \varphi
\geq 0 \mbox{ a.e.\ on } \Omega \times D
\mbox{; }
\int_D M(\qt)\,\hat \varphi(\xt,\qt) \dq \leq 1 \mbox{ for a.e. } \xt \in \Omega
\}.
\label{hatZ}
\end{align}

In the Appendix of \cite{BS2011-feafene}
it is proved for FENE-type potentials
and $\hpsi_0$ satisfying (\ref{hpsi0}), with $\hpsi_0 \in L^2_M(\Omega \times D)$
replaced by the weaker assumption $\mathcal{F}(\hpsi_0)\in L^1_M(\Omega \times D)$,
recall (\ref{F}),
that 
$\widehat \psi^0 \in H^1_M(\Omega \times D)$, satisfying \eqref{psi0}, is such that
$\widehat \psi^0 \in \widehat Z_2$,
\begin{subequations}
\begin{align}
&\int_{\Omega \times D} M\,\mathcal{F}(\widehat \psi^0) \dq \dx
+ 4\,\Delta t \int_{\Omega \times D} M\,\left[
\big|\nabx \sqrt{\widehat \psi^0} \big|^2 + \big|\nabq \sqrt{\widehat \psi^{0}}\big|^2
\right]\!\dq \dx
\leq \int_{\Omega \times D} M \,\mathcal{F}(\widehat \psi_0) \dq \dx
\label{inidata-1}
\end{align}
and
\begin{align}
\widehat \psi^0 = \beta^L(\widehat \psi^0) \rightarrow \widehat \psi_0 \quad
\mbox{weakly in }L_M^1(\Omega \times D) \quad \mbox{as} \quad L \rightarrow \infty, \quad
\Delta t \rightarrow 0_+.
\label{psi0conv}
\end{align}
\end{subequations}
The proof given in \cite{BS2011-feafene} for FENE-type potentials carries across
immediately to potentials satisfying (\ref{U}). 
In addition, with the stronger assumption (\ref{hpsi0})
on $\hat \psi_0$,
it is easy to show that
\begin{subequations}
\begin{align}
&\int_{\Omega \times D} M\,|\widehat \psi^0|^2 \dq \dx
+ \Delta t \int_{\Omega \times D} M\,\left[
\big|\nabx \widehat \psi^0 \big|^2 + \big|\nabq \widehat \psi^{0}\big|^2
\right]\!\dq \dx
\leq \int_{\Omega \times D} M \,|\widehat \psi_0|^2 \dq \dx
\label{inidata-1str}
\end{align}
and
\begin{align}
\widehat \psi^0 = \beta^L(\widehat \psi^0) \rightarrow \widehat \psi_0 \quad
\mbox{weakly in }L_M^2(\Omega \times D) \quad \mbox{as} \quad L \rightarrow \infty, \quad
\Delta t \rightarrow 0_+.
\label{psi0convstr}
\end{align}
\end{subequations}
Moreover, it follows from (\ref{qrbd}), (\ref{inidata-1str}) and (\ref{hpsi0}),
for any $r \in [0,\infty)$ that
\begin{align}
\int_{\Omega \times D} M\,|\qt|^{r}\,\hpsi^0  \dq \dx
\leq C
\,\|\hpsi^0\|_{L^2_M(\Omega \times D)} \leq C
\,\|\hpsi_0\|_{L^2_M(\Omega \times D)} \leq C.
\label{hpsi0rbd}
\end{align}

Our discrete-in-time  approximation of (FP$_{L}$) is then defined as follows.



{\boldmath $({\rm FP}_{L}^{\Delta t})$}:
Let $\hpsistL^0 = \hat \psi^0 \in \hat Z_2$.
Then, for $n =1, \dots,  N$, given
$\hpsistL^{n-1} \in \hat Z_2$,
find $\hpsistL^n \in \hat X \cap \hat Z_2$ such that
\begin{align}
&\int_{\Omega \times D} M\,\frac{\hpsistL^n
- \hpsistL^{n-1}}
{\Delta t}\,\hat \varphi \,\dq \dx
+ \int_{\Omega \times D} M
\left[\epsilon\,\nabx \hpsistL^n
- \utst^{n}\,
\hpsistL^n \right] \cdot \nabx
\hat \varphi\,\dq \dx
\nonumber \\
\bet
&\qquad
+ \int_{\Omega \times D} M
\left[\, \frac{1}{4\, \lambda}\,
\nabq \hpsistL^n
-[\,(\nabxtt \utst^n) \,\qt\,
]\,
\,\beta^L(\hpsistL^{n})\right]\cdot \nabq
\hat \varphi \,\dq \dx
=0
\qquad \forall \hat \varphi \in
\hat X.
\label{hpsiOBLn}
\end{align}

We note that if $\beta^L(\hpsistL^{n})$ in (\ref{hpsiOBLn}) is replaced by
$\hpsistL^{n}$ then the resulting integral is not well-defined.

\begin{lemma}
\label{hpsiOBLnex}
Let the assumptions (\ref{hpsi0}) and (\ref{hpsidata}) hold;
then, there exists a solution $\{
\hpsistL^{n}\}_{n=1}^N$ to (FP$^{\Delta t}_{L}$).
\end{lemma}
\begin{proof}
It is convenient to rewrite (\ref{hpsiOBLn}) as
\begin{align}
a(\hpsistL^n,\hat \varphi) = \ell(\beta^L(\hpsistL^n))
(\hat \varphi) \qquad \forall \hat \varphi \in \hat X,
\label{genLM}
\end{align}
where, for all $\hat \varphi_1,\,\hat \varphi_2 \in \hat X$,
\begin{subequations}
\begin{align} \label{agen}
a(\hat \varphi_1,\hat \varphi_2) &:= \int_{\Omega \times D} M \left(
\hat \varphi_1\,\hat \varphi_2 + \Delta t \left[
\,\epsilon\,\nabx
\hat \varphi_1 - \utst^n\,\hat \varphi_1\right]
\cdot \nabx
\hat \varphi_2
+\, \frac{\Delta t}{4\,\lambda} \,
\nabq
\hat \varphi_1 \cdot \nabqi
\hat \varphi_2 \right) \dq \dx,
\end{align}
and, for all $\hat \eta \in L^\infty(\Omega\times D)$
and $\hat \varphi \in \hat X$,
\begin{align}
\ell(\hat \eta)(\hat \varphi) &:=
\int_{\Omega \times D}
M \left[\hpsistL^{n-1}
\,\hat \varphi
+ \Delta t\,\hat \eta\, [(\nabxtt \utst^n)
\,\qt] \cdot\nabq
\hat \varphi \right]\!\dq \dx.
\label{lgen}
\end{align}
\end{subequations}
On noting (\ref{utOBn}) and that $\hpsistL^{n-1} \in \hat Z_2$,
it is easily deduced that $a(\cdot,\cdot)$ is a
continuous nonsymmetric coercive bilinear functional on $ \hat X \times \hat X$,
and
$\ell(\hat \eta)(\cdot)$ is a continuous linear functional on $\hat X$
for all $\hat \eta \in L^\infty(\Omega \times D)$.

In order to prove existence of a solution to (\ref{hpsiOBLn}),
i.e., (\ref{genLM}),
we
consider a regularized system for a given $\delta \in (0,1)$:
Find $\hpsistLd^n \in \hat X$ such that
\begin{alignat}{2}
a(\hpsistLd^n,\hat \varphi) &= \ell(\beta^L_\delta(\hpsistLd^n))
(\hat \varphi) \qquad &&\forall \hat \varphi \in \hat X,
\label{genLMd}
\end{alignat}
where $\beta^L_\delta(s):=\max(\beta^L(s),\delta)$.
In order to prove the existence of a solution to (\ref{genLMd}),
we consider a fixed-point argument. Given
$\hpsi \in L^2_M(\Omega \times D)$,
let $G(\hpsi) \in \hat X$ be such that
\begin{alignat}{2}
a(G(\hpsi),\hat \varphi) &=
\ell(\beta^L_\delta(\hpsi))(\hat \varphi) \qquad
&&\forall \hat \varphi \in \hat X. \label{fix3}
\end{alignat}
%
The Lax--Milgram theorem yields the
existence of a unique solution
$G(\hpsi) \in
\hat X$ to (\ref{fix3}) for each $\hpsi \in L^2_M(\Omega\times D)$.
Thus the nonlinear map $G: L^2_M(\Omega \times D)
\rightarrow \hat X \subset L^2_M(\Omega \times D)$ is well-defined.
On recalling (\ref{wcomp2}), we have that $G$ is compact.
Next, we show that $G$ is continuous.
Let $\{\hat{\psi}^{(p)}\}_{p \geq 0}$ be such that
$\hpsi^{(p)}
\rightarrow \hpsi$ strongly in
$L^{2}_M(\Omega\times D)$ as $p \rightarrow \infty$.
It follows immediately that
$\beta^L_\delta(\hpsi^{(p)}) \rightarrow
\beta^L_\delta(\hat{\psi})$ strongly in
$L^{2}_M(\Omega\times D)$ as $p \rightarrow \infty$.
As $\|G(\hpsi^{(p)})\|_{\hat X} \leq C(L,(\Delta t)^{-1})$, independent of $p$,
it follows from (\ref{wcomp2}) that there exists a
subsequence $\{G(\hpsi^{(p_k)})\}_{p_k \geq 0}$ and
a function $\hat{\eta}\in \hat X$ such that
$G(\hpsi^{(p_k)}) \rightarrow \hat \eta$ weakly in $\hat X$ and strongly in
$L^2_M(\Omega \times D)$, as $p_k \rightarrow \infty$; see the argument
on p.\ 1233 in \cite{BS2011-fene} for details.
We deduce from the above, the definition of $G$ and the density result (\ref{cal K})
that
\begin{align}
a(\hat{\eta},\hat \varphi)
&= \ell(\beta^L_\delta(\hat \psi))(\hat \varphi)
\qquad \forall \hat \varphi \in C^\infty(\overline{\Omega};C^\infty_0(D)).
\label{Gcont11}
\end{align}
Noting again (\ref{cal K}) yields that (\ref{Gcont11}) holds
for all $\hat \varphi \in \hat X$, and hence
$\hat \eta = G(\hat \psi)\in \hat X$. Therefore
the whole sequence
$G(\hat{\psi}^{(p)})
\rightarrow G(\hat{\psi})$ strongly in $L^2_M(\Omega\times D)$,
as $p \rightarrow \infty$, and so $G$ is continuous.

Finally, to show that $G$ has a fixed point, i.e.\ there exists a solution to
(\ref{genLMd}), using Schauder's fixed point theorem we need to show that there
exists a $C_{\star} \in {\mathbb R}_{>0}$ such that
$\|\hat{\psi}\|_{L^{2}_M(\Omega\times D)} \leq C_{\star}$
for every $\hat{\psi} \in L^2_M(\Omega \times D)$ and $\kappa \in (0,1]$
satisfying $\hat{\psi} = \kappa\, G(\hat{\psi})$; that is,
\begin{align}
a(\hpsi,\hat \varphi) = \kappa\,\ell(\beta^L_\delta(\hpsi),\hat \varphi)
\qquad \forall \hat \varphi \in \hat X.
\label{akappa}
\end{align}
In order to prove this, we introduce the following convex regularization
${\cal F}_{\delta}^L \in
C^{2,1}({\mathbb R})$ of ${\cal F}$ defined, for any $\delta \in (0,1)$ and
$L>1$, by
\begin{align}
 &{\cal F}_{\delta}^L(s) := \left\{
 \begin{array}{ll}
 \textstyle\frac{s^2 - \delta^2}{2\,\delta}
 + s\,(\log \delta - 1) + 1
 \quad & \mbox{for $s \le \delta$}, \\
{\cal F}(s)\ \equiv
s\,(\log s - 1) + 1 & \mbox{for $s \in [\delta,L]$}, \\
  \textstyle\frac{s^2 - L^2}{2\,L}
 + s\,(\log L - 1) + 1
 & \mbox{for $s \ge L$}.
 \end{array} \right. \label{GLd}
\end{align}
We note that
\begin{subequations}
\begin{align}
{\cal F}^L_\delta(s) &\geq \left\{
\begin{array}{ll}
\frac{s^2}{2\,\delta} &\qquad \mbox{for $s \leq 0$},
\\[2mm]
\frac{s^2}{4\,L} - C(L)&\qquad \mbox{for $s \geq 0$}
\end{array}
\right.
\label{cFbelow}\\
\mbox{and} \qquad
([{\cal F}_{\delta}^{L}]'')(s)&=(\beta^L_\delta(s))^{-1} \geq L^{-1}
\qquad \forall s \in \mathbb{R}.
\label{FLdbLd}
\end{align}
\end{subequations}
Choosing $\hat \varphi = [{\cal F}_\delta^L]'(\hat{\psi})$ in (\ref{akappa}),
noting (\ref{FLdbLd}), (\ref{utOBn}) and that $\hpsi\,\nabx[{\cal F}_\delta^L]'(\hat{\psi})
= \nabx {\cal G}_\delta^L(\hat{\psi})$, where $[{\cal G}_\delta^L]'(s)=s/\beta^L_\delta(s)$,
yields that
\begin{align}
&\int_{\Omega \times D} M\left({\cal F}_\delta^L (\hat{\psi})
+ \Delta t\,L^{-1}
\left[\epsilon
\,|\nabx \hat{\psi}|^2
+ \frac{1}{4\,\lambda} \,|\nabq
\hat{\psi}|^2
\right] \right) \dq \dx
\leq
\int_{\Omega \times D} M\,{\cal F}_\delta^L (\kappa \,\hpsistL^{n-1})
\dq \dx.
\label{Ek}
\end{align}
It is easy to show that $\mathcal{F}^L_\delta(s)$ is nonnegative for all
$s \in \mathbb{R}$, with  $\mathcal{F}^L_\delta(1)=0$.
In addition, for any $\kappa \in (0,1]$,
$\mathcal{F}^L_\delta(\kappa\, s) \leq
\mathcal{F}^L_\delta(s)$ if $s<0$ or $1 \leq \kappa\, s$, and also
$\mathcal{F}^L_\delta(\kappa\, s) \leq \mathcal{F}^L_\delta(0)
\leq 1$ if $0 \leq \kappa\, s \leq 1$.
Thus we deduce that
${\cal F}_\delta^L(\kappa\, s)
\leq {\cal F}_\delta^L(s)+ 1$ for all $s \in {\mathbb R}$,
and $\kappa \in (0,1]$.
Hence, on applying the above bound and (\ref{cFbelow}) to (\ref{Ek})  yields that
$\|\hat{\psi}\|_{L^{2}_M(\Omega\times D)} \leq C_{\star}$
with $C_*$ dependent only on
$\delta$, $L$, $\Delta t$, $\utst^n$ and $\hpsistL^{n-1}$.
Therefore $G$ has a fixed point.
Thus we have proved the existence of a solution
to (\ref{genLMd}).

Choosing $\hat \varphi = [{\cal F}_\delta^L]'(\hpsistLd^n)$ in (\ref{genLMd})
yields, similarly to (\ref{Ek}), that
\begin{align}
&\int_{\Omega \times D} M\left({\cal F}_\delta^L (\hpsistLd^n)
+ \Delta t\,L^{-1}
\left[\epsilon
\,|\nabx \hpsistLd^n|^2
+ \frac{1}{4\,\lambda} \,|\nabq
\hpsistLd^n|^2
\right] \right) \dq \dx
\nonumber \\
& \hspace{3in}\leq
\int_{\Omega \times D} M\,{\cal F}_\delta^L (\hpsistL^{n-1})
\dq \dx \leq C,
\label{Eka}
\end{align}
where $C$ is independent of $\delta$ as $\hpsistL^n \in \hat Z_2$.
We obtain from (\ref{Eka}) and (\ref{cFbelow}) that $\|\hpsistLd^n\|_{\hat X} \leq
C$. Similarly to the continuity argument for the mapping $G$ above,
it follows from (\ref{wcomp2}) that there exists a
subsequence $\{\hpsi_{\star,L,\delta_k}^n\}_{\delta_k \geq 0}$ and
a function $\hpsistL^n \in \hat X$ such that
$\hpsi_{\star,L,\delta_k}^n \rightarrow \hpsistL^n$ weakly in $\hat X$ and strongly in
$L^2_M(\Omega \times D)$, as $\delta_k \rightarrow 0_+$.
The fact that $\hpsistL^n \geq 0$ follows from the first term on the left-hand side
in (\ref{Eka}) and the bound (\ref{cFbelow}).
Hence, we have that
$\beta^L_{\delta_k}(\hpsi_{\star,L,\delta_k}^n) \rightarrow \beta^L(\hpsistL^n)$
strongly in
$L^2_M(\Omega \times D)$, as $\delta_k \rightarrow 0_+$.
Therefore, we can pass to the limit $\delta_k \rightarrow 0_+$ in
(\ref{genLMd}) for $\hat \varphi \in C^\infty(\overline{\Omega};C^\infty_0(D))$
to obtain (\ref{genLM}) for $\hat \varphi \in C^\infty(\overline{\Omega};C^\infty_0(D))$.
The desired result (\ref{genLM}) for all $\hat \varphi \in \hat X$
then follows from the density result (\ref{cal K}).
Finally, to conclude that $\hpsistL^n \in \hat X \cap \hat Z_2$, we need to show
the integral constraint,
$\rho(M\,\hpsistL^n)(\xt) \in [0,1]$ for a.e. $\xt \in \Omega$,
on recalling (\ref{hatZ}) and (\ref{rho1}).
This follows from a maximum principle, see p.\ 1234 in
\cite{BS2011-fene} for details.
\end{proof}

Next, we note the following result.

\begin{lemma} \label{lemsigdef}
Under the assumptions of Lemma \ref{hpsiOBLnex}
the solution $\{\hpsistL^{n}\}_{n=0}^N$ to (FP$^{\Delta t}_{L}$)
is such that
\[\mbox{$\sigtt(M\,\hpsistL^{n}),
\,\sigtt(M\,\beta^L(\hpsistL^{n})) \in \Htt^1(\Omega)$,
\,$\rho(M\,\hpsistL^{n})  \in H^1(\Omega)$,}\]
for $n = 0, \ldots, N$, 
and satisfy
\begin{subequations}
\begin{align}
&\int_{\Omega} \frac{ \sigtt(M\,\hpsistL^n)
- \sigtt(M\,\hpsistL^{n-1})}
{\Delta t}:\zetatt \dx
\nonumber \\
&\qquad + \int_{\Omega}\left[
\epsilon\,\nabx \,\sigtt(M\,\hpsistL^n)
: \nabx \,\zetatt
-\sigtt(M\,\hpsistL^n) :
 (\utst^{n} \cdot \nabx)\,\zetatt
\right] \! \!\dx
\nonumber \\
&\qquad - \int_{\Omega}
\left[ (\nabxtt \utst^{n}) \,
\sigtt(M\,\beta^L(\hpsistL^n)) +
\sigtt(M\,\beta^L(\hpsistL^n)) \,
(\nabxtt \utst^{n})^{\rm T}
\right] : \zetatt \dx
\nonumber
\\
&\qquad +
\frac{1}{2\, \lambda} \,
\int_{\Omega} \left[ 
\sigtt(M\,\hpsistL^{n})
-\rho(M\,\hpsistL^{n})\,\Itt
\right]
: \zetatt \,\dx =0
\qquad\forall \zetatt \in
\Htt^1(\Omega),
\label{psiGijbd5}\\
&\int_{\Omega} \frac{ \rho(M\,\hpsistL^n)
- \rho(M\,\hpsistL^{n-1})}
{\Delta t}\,\eta \dx
\nonumber \\
&\qquad + \int_{\Omega}\left[
\epsilon\,\nabx \,\rho(M\,\hpsistL^n)
\cdot \nabx \,\eta
-\rho(M\,\hpsistL^n) \,
 (\utst^{n} \cdot \nabx)\,\eta
\right] \!\dx
=0
\qquad \forall \eta \in
H^1(\Omega).
\label{rhoneq}
\end{align}
\end{subequations}
\end{lemma}
\begin{proof}
On noting (\ref{H1M}), (\ref{H1Mnorm}) and (\ref{additional-1}), we have that $\widehat \varphi =
\eta \in \widehat{X}$, for any $\eta \in C^{\infty}(\overline{\Omega})$, and
$\widehat \varphi \in \qt\,\qt^{\rm T}
: \zetatt \in \widehat{X}$, for any $\zetatt \in
\Ctt^{\infty}(\overline{\Omega})$.
The first choice of $\widehat \varphi$
in (\ref{hpsiOBLn})
immediately yields (\ref{rhoneq}) for any $\eta \in  C^{\infty}(\overline{\Omega})$.

The second choice of $\widehat{\varphi}$
in (\ref{hpsiOBLn})
yields, on noting (\ref{C1}) and (\ref{adel}), that,
for $n = 1, \ldots, N$, 
\begin{align}
&\int_{\Omega} \left[\frac{ \sigtt(M\,\hpsistL^n)
- \sigtt(M\,\hpsistL^{n-1})}
{\Delta t}:\zetatt
+ \epsilon\,\nabx \,\sigtt(M\,\hpsistL^n)
:: \nabx \,\zetatt
-\sigtt(M\,\hpsistL^n) :
 (\utst^{n} \cdot \nabx)\,\zetatt
\right] \! \!\dx
\nonumber \\
&\hspace{2cm} - \int_{\Omega}
\left[ (\nabxtt \utst^{n}) \,
\sigtt(M\,\beta^L(\hpsistL^n)) +
\sigtt(M\,\beta^L(\hpsistL^n)) \,
(\nabxtt \utst^{n})^{\rm T}
\right] : \zetatt \dx
\nonumber
\\
&\hspace{0.9cm} =   - \frac{1}{4\, \lambda} \,
\int_{\Omega \times D}
M\, \left( 
\nabq \hpsistL^n
\cdot \nabq (\qt \,\qt^{\rm T}) \right) : \zetatt  \dq \dx
\qquad \forall \zetatt \in
\Ctt^\infty(\overline{\Omega}).
\label{psiGijbd}
\end{align}
Noting (\ref{cal K}), we can approximate $\hpsistL^n$, for fixed $L, \,\Delta t$
and
$n$, by a sequence $\{\hpsistL^{n,m}\}_{m \geq 1}$ such that
\begin{align}
\hpsistL^{n,m} \in C^\infty(\overline{\Omega};C^{\infty}_0(D))
\qquad \mbox{and} \qquad \hpsistL^{n,m} \rightarrow \hpsistL^{n} \quad
\mbox{strongly in } \widehat{X}
\mbox{ as }
m \rightarrow \infty.
\label{psiGijbd1}
\end{align}
Then we have 
for any $\zetatt \in
\Ctt^\infty(\overline{\Omega})$ that
\begin{align}
\int_{\Omega \times D}
M\,\left( 
\nabq \hpsistL^n
\cdot 
\nabq (\qt \,\qt^{\rm T}) \right) : \zetatt  \dq \dx
&  =
-
\int_{\Omega \times D} \hpsistL^{n,m}
\left[
\nabq \cdot \left( M \,\nabq (\qt \,\qt^{\rm T}) \right) \right] : \zetatt  \dq \dx
\nonumber \\
& \quad\ +
\int_{\Omega \times D}
M\,
\nabq (\hpsistL^n-\hpsistL^{n,m})
\cdot \nabq (\qt \,\qt^{\rm T}) 
: \zetatt  \dq \dx
\nonumber \\
& 
=: T_1 +T_2.
\label{psiGijbd2}
\end{align}
It follows from (\ref{adel}), (\ref{eqM}) and (\ref{C1},b) that
\begin{align}
T_1 &= 2 \int_{\Omega} \left[ \,
\sigtt(M\,\hpsistL^{n,m})- \rho(M\,\hpsistL^{n,m})\,\Itt
\right]
: \zetatt
\dx \nonumber \\
&= 2 \int_{\Omega} \left[ \,
\sigtt(M\,\hpsistL^{n})- \rho(M\,\hpsistL^{n})\,\Itt \right]
: \zetatt
\dx
\nonumber \\
&\qquad+ 2 \int_{\Omega} \left[ \,
\sigtt(M\,(\hpsistL^{n,m}-\hpsistL^{n}))- \rho(M\,(\hpsistL^{n,m}-\hpsistL^{n}))\,\Itt \right]
: \zetatt
\dx \nonumber \\
& =: T_3 + T_4.
\label{psiGijbd3}
\end{align}
Next we note that (\ref{C1},b), (\ref{H1Mnorm}) and (\ref{additional-1}) yield
\begin{align}
|T_2| + |T_4| \leq C\,\|\zetatt\|_{L^\infty(\Omega)}\,
\|\hpsistL^{n} - \hpsistL^{n,m}\|_{H^1_M(\Omega \times D)}.
\label{psiGijbd4}
\end{align}
Hence, it follows from (\ref{psiGijbd})--(\ref{psiGijbd4}) that
(\ref{psiGijbd5}) holds for any $\zetatt \in \Ctt^\infty(\overline{\Omega})$.

Finally,
similarly to (\ref{eqCttbd}), we have, for $\hat \varphi \in L^2_M(\Omega \times D)$, that
\begin{subequations}
\begin{align}
\|\sigtt(M\,\hat \varphi)\|_{L^2(\Omega)} + \|\sigtt(M\,\beta^L(\hat \varphi))\|_{L^2(\Omega)}
+ \|\rho(M\,\hat \varphi)\|_{L^2(\Omega)} \leq C\,\|\hat \varphi\|_{L^2_M(\Omega\times D)},
\label{sigrhoL2}
\end{align}
and, in addition, we have, for $\nabx \hat \varphi \in L^2_M(\Omega\times D)$, that
\begin{align}
\|\nabx \sigtt(M\,\hat \varphi)\|_{L^2(\Omega)}
+ \|\nabx \sigtt(M\,\beta^L(\hat \varphi))\|_{L^2(\Omega)}
+ \|\nabx \rho(M\,\hat \varphi)\|_{L^2(\Omega)}
\leq C\,\|\nabx \hat \varphi\|_{L^2_M(\Omega \times D)}.
\label{sigrhoH1}
\end{align}
\end{subequations}
Hence $\hpsistL^{n} \in \hat X$, recall (\ref{inidata-1str}) for $n=0$, yields that
$\sigtt(M\,\hpsistL^{n}),
\,\sigtt(M\,\beta^L(\hpsistL^{n})) \in \Htt^1(\Omega)$,
\,$\rho(M\,\hpsistL^{n})  \in H^1(\Omega)$,
for $n = 0, \ldots, N$.
Combining these, the fact that
$\utst^n \in \Wt^{1,\infty}(\Omega)$,
$n = 1, \ldots, N$,
and that
$C^\infty(\overline{\Omega})$ is dense in $H^1(\Omega)$ yield
that (\ref{psiGijbd5},b) hold. 
\end{proof}

\subsection{Uniform bounds on the solution of {\boldmath $({\rm FP}_L^{\Delta t})$}}

We note the following result.

\begin{lemma} \label{lem:qalphabd}
Let the assumptions of Lemma \ref{hpsiOBLnex} hold.
Then, we have, for any $r \in {\mathbb R}_{\geq 0}$, that
\begin{align}
\int_{\Omega \times D} M \,|\qt|^{r}\, \hpsistL^n \,\dq \dx \leq C,
\qquad n=0,\ldots,N.
\label{qalphabd}
\end{align}
\end{lemma}
\begin{proof}
We first prove (\ref{qalphabd}) for any $r \geq 2$.
Similarly to the proof of Lemma \ref{lemsigdef}, we can choose,
on noting (\ref{additional-1}),
$\widehat{\varphi} = |\qt|^{r}$, for any $r \geq 2$,
in (\ref{hpsiOBLn}). This yields, on noting (\ref{adel}), that,
for $n = 1, \ldots, N$,
\begin{align}
&\int_{\Omega\times D} M\,|\qt|^{r}
\left[\frac{\hpsistL^n
- \hpsistL^{n-1}}
{\Delta t}\right] \!\dq \dx
\nonumber \\&\qquad
=  r\,\int_{\Omega \times D}
M \,|\qt|^{r-2}\left[ [(\nabxtt \utst^{n}) \,\qt]\,\beta^L(\hpsistL^n) -
\frac{1}{4\, \lambda}\,\nabq \hpsistL^n \right]
\cdot \qt\,  \dq \dx.
\label{qab0}
\end{align}
We then approximate $\hpsistL^n$, for fixed $L, \,\Delta t$
and
$n$, by a sequence $\{\hpsistL^{n,m}\}_{m \geq 1}$ satisfying
(\ref{psiGijbd1}). Hence
\begin{align}
&\int_{\Omega \times D}
M \,|\qt|^{r-2}\,\nabq \hpsistL^n
\cdot \qt\,\dq \dx \nonumber \\
& \qquad =
-
\int_{\Omega \times D}
\hpsistL^{n,m}\,\nabq \cdot
(M \,|\qt|^{r-2} \,\qt)\,\dq \dx
+\int_{\Omega \times D}
M \,|\qt|^{r-2}\,\nabq (\hpsistL^n-\hpsistL^{n,m})
\cdot \qt\,\dq \dx  \nonumber \\
& \qquad =: T_1 + T_2.
\label{qab1}
\end{align}
It follows from (\ref{eqM}) and (\ref{U}) that
\begin{align}
T_1 &=
\int_{\Omega \times D} M
\left[|\qt|^{r} - (d+r-2)\,|\qt|^{r-2}\right]
\hpsistL^{n,m}\,
\dq \dx \nonumber \\
&=
\int_{\Omega \times D} M
\left[|\qt|^{r} - (d+r-2)\,|\qt|^{r-2}\right]
\hpsistL^{n}\,
\dq \dx
\nonumber \\
& \qquad +
\int_{\Omega \times D} M
\left[|\qt|^{r} - (d+r-2)\,|\qt|^{r-2}\right]
(\hpsistL^{n,m}-\hpsistL^n)\,
\dq \dx =: T_3 + T_4.
\label{qab2}
\end{align}
Next we note from (\ref{H1Mnorm}) and (\ref{additional-1}) that
\begin{align}
|T_2|+|T_4| \leq C\,\|\hpsistL^n-\hpsistL^{n,m}\|_{H^1_M(\Omega\times D)}.
\label{qab3}
\end{align}
Hence, on setting, for any $\mathfrak{z} \in \mathbb{R}_{\geq 0}$,
\begin{align}
A^n_\mathfrak{z} &= \int_{\Omega \times D} M\,|\qt|^\mathfrak{z}\,\hpsistL^n
\,\dq\,\dx, \qquad n=0,\ldots,N,
\label{qab4}
\end{align}
combining (\ref{qab0})--(\ref{qab3}) and noting (\ref{psiGijbd1}),
$\hpsistL^n \geq 0$
and (\ref{rho1}) yields,
for $n=1,\ldots,N$ and any $\delta \in {\mathbb R}_{>0}$, that
\begin{align}
&\left[1 +  \frac{\Delta t\,r}{4\lambda} \right]A^n_{r}
\nonumber \\ &\qquad
\leq A^{n-1}_{r}
+ \Delta  t \,r \left[
\int_{\Omega \times D} M\,|\qt|^{r}\,\beta^L(\hpsistL^n)
\,|\nabxtt \utst^n| \dq \dx
+ \frac{d+r-2}{4 \lambda}\,A^{n}_{r-2}
\right]
\nonumber \\
&\qquad
\leq A^{n-1}_{r}
+ \Delta  t \,r \left[
\int_{\Omega \times D} M\,|\qt|^{r}\,\beta^L(\hpsistL^n)
\,|\nabxtt \utst^n|
\dq \dx
\right]
\nonumber \\
& \hspace{2in} + \Delta  t \,r\,\frac{d+r-2}{4 \lambda}
\left[\delta\,A^n_{r} + \delta^{\frac{2-r}{2}}\,
\rho(M\,\hpsistL^n) \right].
\label{qab5}
\end{align}
Choosing $\delta =\frac{1}{2(d+r-2)}$ in (\ref{qab5}),
and summing for $n=1,\ldots,m$, yields, on noting
$\hpsistL^n \in \hat Z_2$,
(\ref{additional-1}) and (\ref{betaLa}),
that, for $m=1,\ldots,N$
\begin{align}
&A^m_{r} + \frac{r}{8\lambda}
\sum_{n=1}^m \Delta t \,A_{r}^n
\nonumber \\
&\qquad \leq A^0_{r}+ C\left(\Delta t\,L\,|\nabxtt \utst^m|_{L^1(\Omega)}
+ \sum_{n=1}^{m-1} \Delta t \, |\nabxtt \utst^n|_{L^\infty(\Omega)}\,A_{r}^n
+t_m \right).
\label{qab6}
\end{align}
Therefore, applying a discrete Gr\"{o}nwall inequality to (\ref{qab6}) yields,
on noting (\ref{hpsi0rbd}) and (\ref{utOBncon}),
the desired result (\ref{qalphabd}) for $r \geq 2$.
The result (\ref{qalphabd}) for $r \in (0,2)$ follows immediately from
(\ref{qalphabd}) for $r=2$ and $r=0$, the latter holding as
$\hpsistL^n \in \hat Z_2$, $n=0,\ldots,N$.
\end{proof}

We now introduce the following definitions, in line with (\ref{utOBn}):
\begin{subequations}
\begin{equation}
\hpsistL(\cdot,t):=\,\frac{t-t_{n-1}}{\Delta t}\,
\hpsistL^n(\cdot)+
\frac{t_n-t}{\Delta t}\,\hpsistL^{n-1}(\cdot),
\quad t\in [t_{n-1},t_n], \quad n=1,\dots,N, \label{hpsiOBLlin}
\end{equation}
\begin{equation}
\hpsistLDtp(\cdot,t):=\hpsistL^n(\cdot),\quad
\hpsistLDtm(\cdot,t):=\hpsistL^{n-1}(\cdot),
\quad t\in(t_{n-1},t_n], \quad n=1,\dots,N. \label{hpsiOBLpm}
\end{equation}
\end{subequations}
We shall adopt $\hpsistL^{\Delta t (,\pm)}$ as a collective symbol
for $\hpsistL^{\Delta t}$, $\hpsistL^{\Delta t,\pm}$.
We note that
\begin{equation}
\hpsistLDt-\hpsistL^{\Delta t,\pm}= (t-t_{n}^{\pm})
\,\frac{\partial \hpsistLDt}{\partial t},
\quad t \in (t_{n-1},t_{n}), \quad n=1,\dots,N, \label{hpsiOBLdt}
\end{equation}
where $t_{n}^{+} := t_{n}$ and $t_{n}^{-} := t_{n-1}$.

Using the above notation,
(\ref{hpsiOBLn}) summed for $n=1, \dots,  N$ can be restated in the following form.

{\boldmath $({\rm FP}_{L}^{\Delta t})$}:
$\hpsistL^{\Delta t,+}(t) \in
\hat X \cap \hat Z_2$
satisfy
\begin{align}\label{hpsiOBLncon}
&\int_{0}^T \int_{\Omega \times D}
M\,\frac{ \partial \hpsistL^{\Delta t}}{\partial t}\,
\hat \varphi \dq \dx \dt
+ \int_{0}^T \int_{\Omega \times D} M\left[
\epsilon\, \nabx \hpsistL^{\Delta t,+}
- \utst^{\Delta t,+}\,\hpsistL^{\Delta t,+} \right]\cdot \nabx
\hat \varphi
\,\dq \dx \dt
\nonumber \\
& \qquad  +
\int_{0}^T \int_{\Omega \times D} M\left[
\frac{1}{4\,\lambda}
\nabq \hpsistL^{\Delta t,+}
-[\,(\nabxtt \utst^{\Delta t,+})
\,\qt]\,
\beta^L(\hpsistL^{\Delta t,+})  \right]\cdot \nabq
\hat \varphi \,
\dq \dx \dt = 0
\nonumber \\
& \hspace{4in}
\qquad \forall \hat \varphi \in L^1(0,T;\hat X)
\end{align}
and the initial condition
$\hpsistL^{\Delta t}(\cdot,\cdot,0) = \hat \psi^0(\cdot,\cdot) \in \hat Z_2$.
We emphasize that (\ref{hpsiOBLncon})
is an equivalent restatement $({\rm FP}_{L}^{\Delta t})$
for which existence of a solution has been established
under assumptions (\ref{hpsi0}) and (\ref{hpsidata}) on the data
(cf. Lemma \ref{hpsiOBLnex}).

Similarly, we 
rewrite (\ref{psiGijbd5},b), $n=1,\ldots,N$, using the notation
(\ref{utOBn}) and (\ref{hpsiOBLlin},b)
to obtain:

{\bf (S$^{\Delta t}_{L}$)}:
$\sigtt(M\,\hpsistL^{\Delta t,+})(t),\,\sigtt(M\,\beta^L(\hpsistL^{\Delta t,+}))(t),
\,\rho(M\,\hpsistL^{\Delta t,+})(t) \in
\Htt^1(\Omega)$ satisfy
\begin{subequations}
\begin{align}
&\int_0^T \int_{\Omega} \left[ \frac{ \partial \sigtt(M\,\hpsistL^{\Delta t})}{\partial t} :\zetatt
+ \epsilon\, \nabx \,\sigtt(M\,\hpsistL^{\Delta t,+})
:: \nabx \,\zetatt
-\sigtt(M\,\hpsistL^{\Delta t,+}) :
 (\utst^{\Delta t,+} \cdot \nabx)\,\zetatt
\right] \! \! \dx \dt
\nonumber \\
&\hspace{0.1cm} - \int_0^T \int_{\Omega}
\left[ (\nabxtt \utst^{\Delta t,+}) \,
\sigtt(M\,\beta^L(\hpsistL^{\Delta t,+})) +
\sigtt(M\,\beta^L(\hpsistL^{\Delta t,+})) \,
(\nabxtt \utst^{\Delta t,+})^{\rm T}
\right] : \zetatt \dx \dt
\nonumber
\\
&\hspace{0.1cm} -
\frac{1}{2\, \lambda} \, \int_0^T
\int_{\Omega} \left[ 
\rho(M\,\hpsistL^{\Delta t,+})\,\Itt
-
\sigtt(M\,\hpsistL^{\Delta t,+})  
\right]
: \zetatt \,\dx \dt
=0
\qquad \forall \zetatt \in
L^2(0,T;\Htt^1(\Omega)),
\label{psiGijcont}  \\[2mm]
&\int_0^T \int_{\Omega} \left[ \frac{ \partial \rho(M\,\hpsistL^{\Delta t})}
{\partial t}\eta
+ \epsilon\, \nabx \,\rho(M\,\hpsistL^{\Delta t,+})
\cdot \nabx \eta
-\rho(M\,\hpsistL^{\Delta t,+}) \,
 (\utst^{\Delta t,+} \cdot \nabx)\,\eta
\right] \! \! \dx \dt
=0 \nonumber \\
& \hspace{3.7in}
\qquad \forall \eta \in
L^2(0,T;H^1(\Omega))
\label{rhocont}
\end{align}
\end{subequations}
and the initial conditions
$\sigtt(M\,\hpsistL^{\Delta t}(\cdot,\cdot,0))
= \sigtt(M\,\hat \psi^0(\cdot,\cdot))$,
$\rho(M\,\hpsistL^{\Delta t}(\cdot,\cdot,0))
= \rho(M\,\hat \psi^0(\cdot,\cdot))$.


On noting
(\ref{hpsiOBLlin},b), (\ref{rho1}), (\ref{qalphabd}) and that $\hpsistL^{n} \in \widehat{Z}_2$,
$ n=0,\ldots,N$,
we have, for all $r \in [0,\infty)$, that
\begin{subequations}
\begin{equation}\label{additionalOB1}
\hpsistL^{\Delta t(,\pm)} \geq 0\qquad \mbox{a.e. on $\Omega \times D \times [0,T]$,}
\qquad \|\,|\qt|^r\,\hpsistL^{\Delta t(,\pm)}\|_{L^1_M(\Omega \times D)} \leq C
\end{equation}
and
\begin{equation}\label{additionalOB2}
[\rho( M\, \hpsistL^{\Delta t(,\pm)})](\xt,t)=
\int_D M(\qt)\, \hpsistL^{\Delta t(,\pm)}
(\xt,\qt,t)\dq \leq 1\quad \mbox{for a.e. $(\xt,t) \in \Omega \times [0,T]$.}
\end{equation}
\end{subequations}
Moreover, we have the following result.
\begin{lemma}\label{lemhpsiOBLnent}
Under the assumptions of Lemma \ref{hpsiOBLnex}
we have 
that
\begin{align}
& {\rm ess.sup}_{t\in [0,T]}
\int_{\Omega \times D}
M \,\mathcal{F}(\hpsistL^{\Delta t(,\pm)}(t)) \dq \dx + \frac{1}{\Delta t\,L}
\int_0^T \int_{\Omega \times D} M \,(\hpsistL^{\Delta t, +}
- \hpsistL^{\Delta t, -})^2 \dq \dx \dd t
\nonumber \\
&\hspace{1in} + \int_0^T \int_{\Omega \times D} M \left[
|\nabx \sqrt{\hpsistL^{\Delta t(,\pm)}} |^2 
+ 
|\nabq \sqrt{\hpsistL^{\Delta t(, \pm}}|^2 \right] \dq \dx \dd t
\leq C.
\label{hpsiOBLent}
\end{align}
In addition, we have that
\begin{align}\label{hpsiOBLdtbd}
&\left|\int_{0}^T\int_{\Omega \times D} M\, \frac{\partial \hpsistL^{\Delta t}}{\partial t}\,
\hat \varphi\,
\dq \dx \dt\right|
\leq C\,
\|\hat \varphi\|_{L^2(0,T;W^{1,\infty}(\Omega \times D))}
\qquad
\forall
\hat \varphi \in L^2(0,T;W^{1,\infty}(\Omega \times D)).
\end{align}
\end{lemma}
\begin{proof}
Similarly to (\ref{GLd}), we introduce
the following convex regularization
${\cal F}^L \in
C^{2,1}({\mathbb R})$ of ${\cal F}$ defined, for any $L>1$, by
\begin{align}
{\cal F}^L(s) := \left\{
\begin{array}{ll}
{\cal F}(s) \equiv
s\,(\log s - 1) + 1 & \mbox{for $s \in [0,L]$}, \\
  \textstyle\frac{s^2 - L^2}{2\,L}
 + s\,(\log L - 1) + 1
 & \mbox{for $s \ge L$}.
 \end{array} \right. \label{GL}
\end{align}
We have the following analogues of (\ref{cFbelow},b) for all $s \in \mathbb{R}_{>0}$:
\begin{align}
{\cal F}^{L}(s) \ge {\cal F}(s), \qquad
([{\cal F}^{L}]'')(s)&=(\beta^L(s))^{-1} \geq L^{-1} \qquad
\mbox{and} \qquad
([{\cal F}^{L}]'')(s)\geq s^{-1} \label{FLbL}.
\end{align}
For any $\alpha \in {\mathbb R}_{>0}$,
choosing $\hat \varphi = \chi_{[0,t_n]}\,[{\cal F}^L]'(\hpsistLDtp+\alpha)$,
$n=1,\ldots,N$,
in (\ref{hpsiOBLent}),
noting (\ref{FLbL}), (\ref{utOBn}) and that
$(\hpsistLDtp+\alpha)\,\nabx[{\cal F}^L]'(\hpsistLDtp+\alpha)
= \nabx {\cal G}^L(\hpsistLDtp+\alpha)$,
where $[{\cal G}^L]'(s)=s/\beta(s)$ for $s>0$,
yields, similarly to (\ref{Eka}), that
\begin{align}
&\int_{\Omega \times D} M \,\mathcal{F}^L(\hpsistL^{\Delta t,+}(t_n) + \alpha) \dq \dx
+\,\frac{1}{2 \Delta t\,L}\int_0^{t_n} \int_{\Omega \times D}
M\,(\hpsistL^{\Delta t,+} - \hpsistL^{\Delta t,-})^2 \dq \dx \dd t \nonumber\\
&\qquad \qquad +  \int_0^{t_n} \int_{\Omega \times D} M\left[
\varepsilon\,
\frac{|\nabx \hpsistL^{\Delta t,+}|^2}{\hpsistL^{\Delta t,+} + \alpha} 
+\,\frac{1}{4\,\lambda}  
(\mathcal{F}^L)''(\hpsistL^{\Delta t,+} + \alpha)\,|\nabq \hpsistL^{\Delta t,+} |^2
\right] \dq \dx \dd t\nonumber \\
&\qquad \leq \int_{\Omega \times D} M \,\mathcal{F}^L(\hat\psi^0 + \alpha)
\dq \dx \nonumber \\
& \qquad \qquad
+ \int_0^{t_n} \int_{\Omega \times D} M\,[\,(\nabxtt \utst^{\Delta t,+})\,\qt\,]
\frac{\beta^L(\hpsistL^{\Delta t,+})}{\beta^L(\hpsistL^{\Delta t,+} + \alpha)}
\cdot \nabq \hpsistL^{\Delta t,+}\, \dq \dx \dd t =: T_1 + T_2.
\label{eq:energy-psi-summ1}
\end{align}
As $\utst^{\Delta t,+} \in L^\infty(0,T;\Vt)$, it follows from (\ref{intbypartsi})
that
\begin{align}
T_2 &= \int_0^{t_n} \int_{\Omega \times D} M\,\qt\,\qt^{\rm T}\,
\hpsistL^{\Delta t,+} : \nabxtt \utst^{\Delta t,+} \dq \dx \dd t\nonumber\\
&\qquad - \int_0^{t_n} \int_{\Omega \times D} M\,[\,(\nabxtt \utst^{\Delta t,+})\,\qt\,]
\left[1 -\frac{\beta^L(\hpsistL^{\Delta t,+})}{\beta^L(\hpsistL^{\Delta t,+} + \alpha)}\right] 
\cdot \nabq \hpsistL^{\Delta t,+}\, \dq \dx \dd t.
\label{entT2}
\end{align}
%
As, for all $s \in \mathbb{R}_{\geq 0}$,
\begin{align*}
0 &\leq \left(1 - \frac{\beta^L(s)}{\beta^L(s+ \alpha)}\right)
\frac{1}{\sqrt{({\mathcal F}^L)''(s + \alpha)}}
= \frac{\beta^L(s + \alpha) - \beta^L(s)}{\sqrt{\beta^L(s + \alpha)}}
\leq \frac{\beta^L(s + \alpha) - \beta^L(s)}{\sqrt{\alpha}} \leq \sqrt{\alpha},
\end{align*}
we have, on noting 
(\ref{additionalOB1}), (\ref{utOBDtsatb},b)
and (\ref{additional-1})
 that
%
\begin{align}
|T_2| &\leq
\int_0^{t_n} \|\utst^{\Delta t,+}\|_{W^{1,\infty}(\Omega)}\left(
\int_{\Omega \times D} M\,|\qt|^2\,
\hpsistL^{\Delta t,+} \dq \dx \right)\dd t \nonumber \\
&\qquad + \sqrt{\alpha} \int_0^{t_n}
\int_{\Omega}  |\nabxtt \utst^{\Delta t, +}| \left(\int_D M
\,|\qt|
\sqrt{(\mathcal{F}^L)''(\hpsistL^{\Delta t, +} + \alpha)}
\; |\nabq \hpsistL^{\Delta t, +}|\, \dq \right)\dx \dd t\nonumber
\\
&\leq 
\frac{1}{8 \lambda}
\left(\int_0^{t_n} \int_{\Omega \times D} M (\mathcal{F}^L)''(\hpsistL^{\Delta t, +} + \alpha)
\, |\nabq \hpsistL^{\Delta t, +}|^2\, \dq \dx \dd t\right) + C\,(1+\alpha).
\label{eq:lastterm}
\end{align}

Combining (\ref{eq:energy-psi-summ1}) and (\ref{eq:lastterm}) yields,
on noting (\ref{FLbL}), that, for $n=1,\ldots,N$ and any $\alpha
\in {\mathbb R}_{>0}$,
\begin{align}
&\int_{\Omega \times D} M \,\mathcal{F}^L(\hpsistL^{\Delta t,+}(t_n) + \alpha) \dq \dx
+\,\frac{1}{2 \Delta t\,L}\int_0^{t_n} \int_{\Omega \times D}
M\,(\hpsistL^{\Delta t,+} - \hpsistL^{\Delta t,-})^2 \dq \dx \dd t \nonumber\\
&\qquad \qquad +  \int_0^{t_n} \int_{\Omega \times D} M\left[
\varepsilon\,
\frac{|\nabx \hpsistL^{\Delta t,+}|^2}{\hpsistL^{\Delta t,+} + \alpha} 
+\,\frac{1}{8\,\lambda}  
\frac{|\nabq \hpsistL^{\Delta t,+}|^2}{\hpsistL^{\Delta t,+} + \alpha}
\right] \dq \dx \dd t\nonumber \\
&\qquad \leq \int_{\Omega \times D} M \,\mathcal{F}^L(\hat\psi^0 + \alpha)
\dq \dx 
+ C\,(1+\alpha).
\label{enta}
\end{align}
Passing to the limit $\alpha \rightarrow 0_+$ in (\ref{enta}), 
noting  that
${\cal F}^L(\hpsi^0)=
{\cal F}^L(\beta^L(\hpsi^0))={\cal F}(\beta^L(\hpsi^0)) \leq {\cal F}(\hpsi^0)$,
as $L>1$, (\ref{FLbL}), \eqref{inidata-1} and (\ref{hpsi0})
yield the desired result (\ref{hpsiOBLent})
with
$\hpsistL^{\Delta t(,\pm)}$ in
the first and third terms
replaced by $\hpsistL^{\Delta t,+}$.
It follows from (\ref{inidata-1str}) and (\ref{hpsi0}) that
(\ref{hpsiOBLent}) holds with
$\hpsistL^{\Delta t(,\pm)}$ in
the first and third terms
replaced by $\hpsistL^{\Delta t,-}$.
It is then a simple matter to derive the desired result (\ref{hpsiOBLent})
on recalling (\ref{hpsiOBLlin},b), the convexity of ${\cal F}$ and that
$|\nabx \sqrt{\hpsistL^{\Delta t}} |^2
\leq 2(|\nabx \sqrt{\hpsistL^{\Delta t,+} }|^2
+ |\nabx \sqrt{\hpsistL^{\Delta t,-}}|^2)$, see p.\ 44 in \cite{BS2010} for details of the
latter
result.



Finally, to obtain the bound (\ref{hpsiOBLdtbd}) from (\ref{hpsiOBLncon}),
we have, on noting
(\ref{hpsiOBLent}), (\ref{additionalOB1},b) and (\ref{utOBDtsatb}),
that
\begin{subequations}
\begin{align}
&\left|\int_0^T \int_{\Omega \times D} M \,\nabx \hpsistLDtp \cdot \nabx \hat \varphi
\,\dq\dx\dt\right|
\nonumber \\
& \qquad \leq 2\,\|\hpsistLDtp\|_{L^\infty(0,T;L^1_M(\Omega \times D))}^{\frac{1}{2}}
\|\nabx \sqrt{\hpsistLDtp}\|_{L^2(0,T;L^2_M(\Omega \times D))}
\,\|\nabx \hat \varphi\|_{L^2(0,T;L^\infty(\Omega \times D))}
\nonumber\\
&\qquad \leq C
\,\|\nabx \hat \varphi\|_{L^2(0,T;L^\infty(\Omega \times D))}
\label{dterm1}
\end{align}
and
\begin{align}
&\left|\int_0^T \int_{\Omega \times D}
M\,[\,(\nabxtt \utst^{\Delta t,+})
\,\qt]\,
\beta^L(\hpsistL^{\Delta t,+})\cdot \nabq
\hat \varphi\,
\dq \dx \dt \right|
\nonumber \\
& \qquad \leq \|\int_D M\,|\qt|\,\beta^L(\hpsistLDtp)\dq\|_{L^\infty(0,T;L^2(\Omega))}
\,\|\nabxtt \utst^{\Delta t,+}\|_{L^2(0,T;L^2(\Omega))}
\,\|\nabq \hat \varphi\|_{L^2(0,T;L^\infty(\Omega \times D))}
\nonumber\\
& \qquad \leq C\,
\|\int_D M\,\hpsistLDtp \dq\|_{L^\infty(0,T;L^\infty(\Omega))}
\,\||\qt|^2\,\hpsistLDtp\|_{L^\infty(0,T;L^1_M(\Omega\times D))}
\,\|\nabq \hat \varphi\|_{L^2(0,T;L^\infty(\Omega \times D))}
\nonumber\\
& \qquad \leq C\,
\,\|\nabq \hat \varphi\|_{L^2(0,T;L^\infty(\Omega \times D))}.
\label{dterm2}
\end{align}
\end{subequations}
The remaining two terms in (\ref{hpsiOBLncon}) are bounded similarly.
\end{proof}

\begin{lemma}
\label{psiGijlem}
Let the assumptions of Lemma \ref{hpsiOBLnex} hold.
In addition, if $4\,L^2\,\Delta t \leq 1$ then we have that
\begin{subequations}
\begin{align}
&{\rm ess.sup}_{t\in [0,T]} 
\|\sigtt(M\,\hpsistL^{\Delta,+})(t)\|^2_{L^2(\Omega)}
+ \int_0^T
\|\nabx \,\sigtt(M\,\hpsistL^{\Delta t,+})\|^2_{L^2(\Omega)} \dt
\leq C,
\label{psiGijbdfinlem} \\
&{\rm ess.sup}_{t\in [0,T]}
\|\rho(M\,\hpsistL^{\Delta,+})(t)\|^2_{L^2(\Omega)}
+ \int_0^T
\|\nabx \,\rho(M\,\hpsistL^{\Delta t,+})\|^2_{L^2(\Omega)} \dt
\leq C.
\label{rhobds}
\end{align}
\end{subequations}
\end{lemma}
\begin{proof}
Choosing $\zetatt = \chi_{[0,t_n]}\,\sigtt(M\,\hpsistL^{\Delta t,+})$,
$n=1,\ldots,N$,
 in (\ref{psiGijbd5})
yields, on noting 
(\ref{eqCttbd}), (\ref{inidata-1str}), (\ref{hpsi0}), (\ref{utOBDtsatb},b),
(\ref{additionalOB1},b),
(\ref{betaLa}) and (\ref{additional-1})
that
\begin{align}
&\frac{1}{2} 
\,\|\sigtt(M\,\hpsistL)(t_n)\|^2_{L^2(\Omega)} 
+ \frac{1}{2}\,\int_0^{t_n} 
\|\sigtt(M\,\hpsistL^{\Delta t,+})-\sigtt(M\,\hpsistL^{\Delta t,-})\|^2_{L^2(\Omega)} \dt
\nonumber \\
& \qquad
+ 
\int_0^{t_n}
\left[ \frac{1}{2\,\lambda}\,\|\sigtt(M\,\hpsistL^{\Delta t,+})\|^2_{L^2(\Omega)}
+ \epsilon\,\|\nabx \,\sigtt(M\,\hpsistL^{\Delta t,+})\|^2_{L^2(\Omega)} \right] \dt
\nonumber \\
& \quad =
\frac{1}{2}\,
\|\sigtt(M\,\hpsi^{0})\|^2_{L^2(\Omega)} 
+ \frac{1}{2\,\lambda} 
\int_0^{t_n}
\int_{\Omega}
\rho(M\,\hpsistL^{\Delta t,+})\,\Itt : \sigtt(M\hpsistL^{\Delta t,+}) \dx \dt
\nonumber \\
& \qquad
+ 
\int_0^{t_n}
\int_{\Omega}
\left[ (\nabxtt \utst^{\Delta t,+}) \,
\sigtt(M\,\beta^L(\hpsistL^{\Delta t,+})) +
\sigtt(M\,\beta^L(\hpsistL^{\Delta t,+})) \,
(\nabxtt \utst^{\Delta t,+})^{\rm T}
\right] :
\sigtt(M\,\hpsistL^{\Delta t,+}) \dx \dt
\nonumber \\
& \quad \leq
C + C\,\Delta t \,L\,\|\nabxtt \utst^{\Delta t,+}(t_n)\|_{L^2(\Omega)}\,
\|\sigtt(M\,\hpsistL^{\Delta t,+})(t_n)\|_{L^2(\Omega)}
\nonumber \\
& \qquad
+ 2 
\int_0^{t_{n-1}}
\|\nabxtt \utst^{\Delta t,+}\|_{L^\infty(\Omega)} \,
\|\sigtt(M\,\hpsistL^{\Delta t,+})\|^2_{L^2(\Omega)} \dt
\nonumber \\
& \quad \leq
C
+ \Delta t\,L^2\, 
\|\sigtt(M\,\hpsistL^{\Delta t,+})(t_n)\|^2_{L^2(\Omega)}
+ 2\int_0^{t_{n-1}}
\|\nabxtt \utst^{\Delta t,+}\|_{L^\infty(\Omega)} \,
\|\sigtt(M\,\hpsistL^{\Delta t,+})\|^2_{L^2(\Omega)}
\dt.
\label{psiGijbd6}
\end{align}
Therefore, 
applying a discrete Gr\"{o}nwall inequality to (\ref{psiGijbd6}) yields,
on noting (\ref{eqCttbd}), (\ref{inidata-1str}), (\ref{hpsi0}) and (\ref{utOBncon})
that
the bounds in
(\ref{psiGijbdfinlem}) hold.

Similarly to the above, choosing
 $\eta = \chi_{[0,t_n]}\,\rho(M\,\hpsistL^{\Delta t,+})$,
$n=1,\ldots,N$,
 in (\ref{rhoneq}) yields the bounds (\ref{rhobds}).
\end{proof}

\begin{remark} \label{rempsiGij}
{\em We note that Lemma \ref{psiGijlem} is valid for $d=3$, as well as $d=2$, as it exploits
the $L^\infty(0,T;\Vt) \cap L^{1}(0,T;\Wt^{1,\infty}(\Omega))$ regularity of $\utst$, recall
(\ref{hpsidata}).
If $\utst$ were only in $L^2(0,T;\Vt)$, then one
has to use the Gagliardo-Nirenberg inequality, (\ref{eqinterp}), in the proof of Lemma
\ref{psiGijlem},
and this will lead to a restriction to $d=2$.
This is the same reason why the existence proof for the Oldroyd-B model,
via the convergence of  a finite element approximation,
is restricted to $d=2$ in \cite{barrett-boyaval-09}, see Theorem 7.1 there.
}
\end{remark}

\subsection{Passage to the limit $L \rightarrow \infty$ $(\Delta t \rightarrow 0_+)$ } 

We are now ready to pass to the limit
$L \rightarrow \infty$, $\Delta t \rightarrow 0_+$ in (FP$^{\Delta t}_L$), (\ref{hpsiOBLncon}),
and (S$^{\Delta t}_L$), (\ref{psiGijcont},b). In view of the assumption on $L$ and $\Delta t$
in Lemma \ref{psiGijlem} we shall choose $\Delta t \leq (4\,L^2)^{-1}$ as $L \rightarrow \infty$.

\begin{theorem}
\label{convfinal} 
Let the assumptions (\ref{hpsi0}) and (\ref{hpsidata}) hold on the data, and
let $\Delta t \leq (4\,L^2)^{-1}$  as $L \rightarrow \infty$.
Then,
there exists a subsequence of $\{\hpsistL^{\Delta t}\}_{L >1}$ (not indicated),
and a function $\hpsist$ such that
%
%
%
\begin{subequations}
\begin{alignat}{2}
|\qt|^r \, \hpsist &\in L^\infty(0,T;L^1_{M}(\Omega \times D)), \qquad
&&\mbox{for any } r \in [0,\infty),
\label{hpsiOBqr}
\\
\hpsist &\in H^1(0,T; M^{-1}[H^s(\Omega \times D)]'),\qquad
&&\mbox{for any } s>d+1,
\label{hpsiOBqdt}
\end{alignat}
\end{subequations}
with
\begin{equation}\label{mass-conserved}
\hpsist \geq 0 \mbox{ a.e. on $\Omega \times D \times [0,T]$} \quad \mbox{and} \quad
\int_D M(\qt)\,\hpsist(\xt,\qt,t) \dq \leq
1 \mbox{ for a.e. $(x,t) \in \Omega \times [0,T]$},
\end{equation}
%
and finite relative entropy and Fisher information, with
\begin{equation}\label{relent-fisher}
\mathcal{F}(\hpsist) \in L^\infty(0,T;L^1_M(\Omega\times D))\quad
\mbox{and}\quad \sqrt{\hpsist} \in L^{2}(0,T;H^1_M(\Omega \times D));
\end{equation}
%
such that, as $L\rightarrow \infty$ (and thereby $\Delta t \rightarrow 0_+$),
\begin{subequations}
\begin{alignat}{2}
\bet
M^{\frac{1}{2}}\,\nabx \sqrt{\hpsistL^{\Delta t(,\pm)}}
&\rightarrow M^{\frac{1}{2}}\,\nabx \sqrt{\hpsist}
&&\qquad \mbox{weakly in } L^{2}(0,T;\Lt^2(\Omega\times D)), \label{5-psiwconH1a}\\
\bet
M^{\frac{1}{2}}\,\nabq \sqrt{\hpsistL^{\Delta t(,\pm)}}
&\rightarrow M^{\frac{1}{2}}\,\nabq \sqrt{\hpsist}
&&\qquad \mbox{weakly in } L^{2}(0,T;\Lt^2(\Omega\times D)), \label{5-psiwconH1xa}\\
\bet
M\,\frac{\partial \hpsistL^{\Delta t}} {\partial t} &\rightarrow
M\,\frac{\partial \hpsist}{\partial t}
&&\qquad \mbox{weakly in }
L^2(0,T;[H^s(\Omega\times D)]'), \label{5-psitwconL2a}\\
\bet
|\qt|^r\,\beta^L(\hpsistL^{\Delta t (,\pm)}),\,
|\qt|^r\,\hpsistL^{\Delta t (,\pm)} &\rightarrow
|\qt|^r\,\hpsist
&&\qquad \mbox{strongly in }
L^{p}(0,T;L^{1}_{M}(\Omega\times D)),\label{5-psisconL2a}
\end{alignat}
for any $p \in [1,\infty)$. 
\end{subequations}

%

In addition, for $s>d+1$, the function $\hpsist$
satisfies
\begin{align}\label{eqpsinconP}
&-\int_{0}^T
\int_{\Omega \times D} M\,\hpsist\, \frac{\partial \hat\varphi}{\partial t}
\dq \dx \dt
+ \int_{0}^T \int_{\Omega \times D} M\,\left[
\epsilon\, \nabx \hpsist - \utst \,\hpsist \right]\cdot\, \nabx
\hat \varphi
\,\dq \dx \dt
\nonumber \\
&\quad +
\int_{0}^T \int_{\Omega \times D} M\,\left[
\frac{1}{4\,\lambda}\,\nabq \hpsist
-
\left[(\nabxtt \utst)
\,\qt\right] \hpsist \right] \cdot \nabq
\hat \varphi
\dq \dx \dt \nonumber \\
& \quad \quad
= \int_{\Omega \times D} M\, \hat\psi_0(\xt,\qt)\,\hat\varphi(\xt,\qt,0)
\dq \dx
\quad \forall \hat \varphi \in W^{1,1}(0,T;H^s(\Omega\times D))
\mbox{ with $\hat\varphi(\cdot,\cdot,T)=0$}.
\end{align}
%
%
\end{theorem}
\begin{proof}
We shall apply
Dubinski{\u\i}'s theorem,
Theorem \ref{thm:Dubinski}, to the sequence $\{\hpsistL^{\Delta t}\}_{L>1}$.
On noting (\ref{hpsiOBLent}),
we select
%
$\mathcal{A}_0 = L^1_{M}(\Omega \times D)$
%
and
\begin{align}
&\mathcal{M} = \left\{ \hat\varphi \in \mathcal{A}_0\,: \hat\varphi \geq 0 \quad \mbox{with}
\quad \int_{\Omega \times D} M\left[ \left|\nabx \sqrt{\hat\varphi}\right|^2
+ \left|\nabq \sqrt{\hat\varphi}\right|^2\right]\dq \dx < \infty \right\},
\label{calM}
\end{align}
and, for $\hat\varphi \in \mathcal{M}$,  we define
\[ [\hat\varphi]_{\mathcal M}:= \|\hat\varphi\|_{\mathcal A_0} +
\int_{\Omega \times D} M\left[ \left|\nabx
\sqrt{\hat\varphi}\right|^2 + \left|\nabq
\sqrt{\hat\varphi}\right|^2\right] \dq \dx.
\]
Note that $\mathcal{M}$ is a seminormed subset of the Banach
space $\mathcal{A}_0$.
As $H^1_M(\Omega \times D)
$ is compactly embedded in $ L^2_{M}(\Omega \times D)$,
recall (\ref{wcomp2}),
one can deduce that the embedding
$\mathcal{M} \hookrightarrow \mathcal{A}_0$ is compact
by applying the argument on p.\ 1251 in \cite{BS2011-fene}.
On noting (\ref{hpsiOBLdtbd}) and
as Sobolev embedding yields that $H^{s}(\Omega \times D) \hookrightarrow
W^{1,\infty}(\Omega \times D)$, for $s>d+1$, we choose
$\mathcal{A}_1 := M^{-1} [H^{s}(\Omega \times D)]' := \{\hat \varphi
: M \hat\varphi \in [H^{s}(\Omega \times D)]'\}$, where
$[H^{s}(\Omega \times D)]'$ is the dual of  $H^{s}(\Omega \times D)$,
equipped with the norm
$\|\hat\varphi\|_{\mathcal A_1} := \|M \hat \varphi\|_{[H^{s}(\Omega \times D)]'}.$
For such $s$, it follows from Sobolev embedding,
for any $\hat\varphi \in \mathcal{A}_0 = L^1_M(\Omega \times D)$, that
%
\begin{align*}
\|\hat\varphi\|_{\mathcal A_1} &= \sup_{\!\!\!\chi
\in H^s(\Omega \times D)}\! \frac{|(M\hat\varphi, \chi)|}{\|\chi\|_{H^s(\Omega \times D)}}
\leq \sup_{\!\!\!\chi \in H^s(\Omega \times D)}\!
\frac{\|\hat\varphi\|_{L^1_M(\Omega \times D)}
\|\chi\|_{L^\infty(\Omega \times D)}}{\|\chi\|_{H^s(\Omega \times D)}}
\leq 
C\, \|\hat\varphi\|_{\mathcal{A}_0}.
\end{align*}
%
Hence,
we have that
  $\mathcal{A}_0
  \hookrightarrow
   \mathcal{A}_1$.
Thus, our choices of  
$\mathcal{A}_0$, $\mathcal{M}$ and
$\mathcal{A}_1$ satisfy the conditions of Theorem \ref{thm:Dubinski}.
Applying this theorem with
$\alpha_0=1$ and $\alpha_1=2$, 
implies that the embedding
%
\begin{align*}
&\left\{\varphi : [0,T] \rightarrow \mathcal{M}\,:\,
[\varphi]_{L^1(0,T;\mathcal M)} + \left\|\frac{{\rm d}\varphi}{{\rm d}t} \right\|_{L^{2}(0,T;\mathcal{A}_1)}
\!\!\!< \infty   \right\}
\hookrightarrow L^1(0,T;\mathcal{A}_0)
= L^1(0,T;L^1_{M}(\Omega\times D))
\end{align*}
is compact.
Using this compact embedding, together with the bounds
(\ref{hpsiOBLent}) and (\ref{hpsiOBLdtbd}), in conjunction with
\eqref{additionalOB1} and Sobolev embedding,
we deduce (upon extraction of a subsequence)
strong convergence of $\{\hpsistL^{\Delta t}\}_{L>1}$ in $L^1(0,T;
L^1_{M}(\Omega \times D))$ to an element
$\hpsist \in L^1(0,T; L^1_{M}(\Omega \times D))$, as $L \rightarrow \infty$.

Thanks to the bound on the second term in \eqref{hpsiOBLent}, (\ref{hpsiOBLlin},b) and
(\ref{additional-1}),
we have that
\begin{align}
\|\hpsistL^{\Delta t}-\hpsistL^{\Delta t,\pm}\|_{L^1(0,T;L^1_M(\Omega \times D))}
&\leq
\frac{T\, |\Omega|}{3}
\left(\int_D M \dq \right) \displaystyle \int_0^T
\displaystyle \int_{\Omega \times D} M\, (\hpsist^{\Delta t,+} - \psist^{\Delta t,-})^2
\dq \dx \dt\nonumber\\
\label{difference}
&\leq C\,\Delta t\,L.
\end{align}
On 
recalling  that $\Delta t \leq  4\,L^{-2}$, 
together
with \eqref{difference} and the strong  convergence of $\{\hpsistL^{\Delta t}\}_{L>1}$ to
$\hpsist$ in $L^1(0,T; L^1_{M}(\Omega \times D))$,
we deduce, as $L \rightarrow \infty$,
strong convergence of $\{\hpsistL^{\Delta t,\pm}\}_{L>1}$ in
$L^1(0,T; L^1_{M}(\Omega \times D))$ to the same element
$\hpsist \in L^1(0,T; L^1_{M}(\Omega \times D))$.
This completes the proof of \eqref{5-psisconL2a} for $\hpsist^{\Delta t(,\pm)}$
with $r=0$ and $p=1$.

From 
%
the first bound in (\ref{hpsiOBLdtbd}) we have that
$\{\hpsistL^{\Delta t (,\pm)}\}_{L>1}$ are bounded in
$L^\infty(0,T;L^1_{M}(\Omega \times D))$.
By  Lemma \ref{le:supplementary},  the strong convergence of these
to $\hpsist$ in $L^1(0,T;L^1_{M}(\Omega \times D))$,
shown above, then implies
strong convergence in $L^p(0,T;L^1_{M}(\Omega \times D))$
to the same limit for all values of $p \in [1,\infty)$. That completes the proof of
\eqref{5-psisconL2a} for $\hpsist^{\Delta t(,\pm)}$ with $r=0$ and any $p\in [1,\infty)$.

Strong convergence in $L^p(0,T;L^1_{M}(\Omega \times D))$
for $p \geq 1$, implies convergence almost everywhere on
$\Omega \times D \times [0,T]$ of a subsequence.
Hence it follows from \eqref{additionalOB1} that
$\hpsist \geq 0$ a.e.\ on $\Omega \times D \times [0,T]$.
Applying Fubini's theorem, one can deduce from the above and (\ref{additionalOB2}) that
\begin{equation}\label{1boundonpsi}
\int_D M(\qt)\,\hpsist(\xt,\qt,t)\dq \leq 1
\qquad \mbox{for a.e. $(x,t) \in \Omega \times [0,T]$.}
\end{equation}
%
Since $\mathcal{F}$ is nonnegative, one can deduce from Fatou's lemma
and (\ref{hpsiOBLent}) that, for a.e. $t \in [0,T]$,
\begin{align}
\int_{\Omega \times D} M(\qt)\, \mathcal{F}(\hpsist(\xt,\qt,t))\dq \dx
&\leq \mbox{lim inf}_{L \rightarrow \infty}
\int_{\Omega \times D} M(\qt)\, \mathcal{F}(\hpsistL^{\Delta t(,\pm)}(\xt,\qt,t)) \dq \dx
\leq C.
\label{fatou-app}
\end{align}
%
As the
expression on the left-hand side of \eqref{fatou-app} is nonnegative, we deduce
the first result in (\ref{relent-fisher}).
Similarly, one can deduce from (\ref{additionalOB1}) that,
for any $r \in [0,\infty)$ and a.e.\ $t\in [0,T]$,
\begin{align}
\int_{\Omega \times D} M(\qt)\,|\qt|^r\,\hpsist(\xt,\qt,t)\dq \dx
&\leq \mbox{lim inf}_{L \rightarrow \infty}
\int_{\Omega \times D} M(\qt)\,|\qt|^r\, \hpsistL^{\Delta t,+}(\xt,\qt,t) \dq \dx
\leq C.
\label{qt2-app}
\end{align}
Hence (\ref{hpsiOBqr}) holds.

Since $|\sqrt{c_1} - \sqrt{c_2}\,
|\leq \sqrt{|c_1-c_2 |}$ for any $c_1,\,c_2 \in {\mathbb R}_{\geq 0}$,
we have, for any $r \in [0,\infty)$ and $p \in [2,\infty)$, on noting (\ref{qt2-app}) that
\begin{align}
&\|\,|\qt|^{r}\,(\hpsist-\hpsistL^{\Delta t(,\pm)})\|_{L^p(0,T;L^1_M(\Omega \times D))}
\nonumber \\
& \qquad \leq \left\|\,|\qt|^{2r}\,\left(\sqrt{\hpsist}+\sqrt{\hpsistL^{\Delta t(,\pm)}}
\right)^2\,
\right\|_{L^{\infty}(0,T;L^1_M(\Omega \times D))}^{\frac{1}{2}}\,
\left\|\left(\sqrt{\hpsist}-\sqrt{\hpsistL^{\Delta t(,\pm)}}\right)^2\,
\right\|_{L^{\frac{p}{2}}(0,T;L^1_M(\Omega \times D))}^{\frac{1}{2}}
\nonumber \\
& \qquad \leq 2^{\frac{1}{2}}\,\|\,|\qt|^{2r}\,(\hpsist+\hpsistL^{\Delta t(,\pm)})\,
\|_{L^\infty(0,T;L^1_M(\Omega \times D))}^{\frac{1}{2}}\,
\|\hpsist-\hpsistL^{\Delta t(,\pm)}\|_{L^{\frac{p}{2}}(0,T;L^1_M(\Omega \times D))}
^{\frac{1}{2}} \nonumber \\
& \qquad \leq C\,
\|\hpsist-\hpsistL^{\Delta t(,\pm)}\|_{L^{\frac{p}{2}}(0,T;L^1_M(\Omega \times D))}
^{\frac{1}{2}}.
\label{qalpha}
\end{align}
Hence, for any $p\in [2,\infty)$,
the desired result (\ref{5-psisconL2a}) for $\hpsist^{\Delta t(,\pm)}$
with any $r\in (0,\infty)$ follows from
(\ref{5-psisconL2a}) for $\hpsist^{\Delta t(,\pm)}$
with $r=0$; thus, by H\"older's inequality, it is also true for $p \in [1,2)$.
Therefore, we have completed the proof of (\ref{5-psisconL2a})
for $\hpsist^{\Delta t(,\pm)}$.
The proof of (\ref{5-psisconL2a})
for $\beta^L(\hpsist^{\Delta t(,\pm)})$ follows, on noting (\ref{betaLa}), that,
for any $r\in [0,\infty)$ and $p\in [1,\infty)$,
\begin{align}
&\|\,|\qt|^{r}\,(\hpsist-\beta^L(\hpsistL^{\Delta t(,\pm)}))\|_{L^p(0,T;L^1_M(\Omega \times D))}
\nonumber \\
& \qquad \leq
\|\,|\qt|^{r}\,(\hpsist-\beta^L(\hpsist))\|_{L^p(0,T;L^1_M(\Omega \times D))}
+ \|\,|\qt|^{r}\,(\beta^L(\hpsist)-
\beta^L(\hpsistL^{\Delta t(,\pm)}))\|_{L^p(0,T;L^1_M(\Omega \times D))}
\nonumber \\
& \qquad \leq
\|\,|\qt|^{r}\,(\hpsist-\beta^L(\hpsist))\|_{L^p(0,T;L^1_M(\Omega \times D))}
+ \|\,|\qt|^{r}\,(\hpsist-
\hpsistL^{\Delta t(,\pm)})\|_{L^p(0,T;L^1_M(\Omega \times D))}.
\label{betacon}
\end{align}
The first term converges to zero using Lebesgue's dominated convergence
and the convergence of $\beta^L(\hpsist)$ to $\hpsist$ a.e.\ on
$\Omega \times D \times (0,T)$ as $L \rightarrow \infty$.
Hence, (\ref{5-psisconL2a})
for $\beta^L(\hpsist^{\Delta t(,\pm)})$ follows from
(\ref{5-psisconL2a})
for $\hpsist^{\Delta t(,\pm)}$.

It follows from $|\sqrt{c_1} - \sqrt{c_2}\,
|\leq \sqrt{|c_1-c_2 |}$ for any $c_1,\,c_2 \in {\mathbb R}_{\geq 0}$
and (\ref{5-psisconL2a}) with $r=0$ that
\begin{align}
M^{\frac{1}{2}}\,\sqrt{\hpsistL^{\Delta t(,\pm)}} \rightarrow
M^{\frac{1}{2}}\,\sqrt{\hpsist} \qquad \mbox{strongly in } L^p(0,T;L^2(\Omega \times D)),
\label{ML2conv}
\end{align}
as $L \rightarrow \infty$.
The weak convergence results (\ref{5-psiwconH1a}--c) are then easily deduced,
see p.\ 1268 in \cite{BS2011-fene} for details. In addition,
(\ref{hpsiOBqdt}) and the second result in (\ref{relent-fisher}) hold.



\smallskip

We now pass to the limit $L \rightarrow \infty$
(and $\Delta t \rightarrow 0_+$) in (FP$^{\Delta t}_L$),
\eqref{hpsiOBLncon}.
We shall take at first
$\hat\varphi \in {\cal E}:= \{ \hat \varphi \in
C^1([0,T];C^\infty(\overline{\Omega};C^\infty_0(D)))
: \hat \varphi(\cdot,\cdot,T)=0\}$.
Integration by parts with respect to $t$ on the first term in \eqref{hpsiOBLncon} gives
\begin{align}
\int_{0}^T \int_{\Omega \times D} M\,\frac{ \partial \hpsistL^{\Delta t}}{\partial t}\,
\hat \varphi \,\dq \dx \dt
= &- \int_0^T \int_{\Omega \times D}
M\, \hpsistL^{\Delta t}\, \frac{\partial \hat\varphi}{\partial t} \dq \dx \dt \nonumber\\
&\quad - \int_{\Omega \times D} M(\qt) \,\hpsi^0(\xt,\qt)\,
\hat\varphi(\xt, \qt, 0) \dq \dx
\qquad \forall \hat \varphi \in {\cal E}.
\label{partialint}
\end{align}
%
Using \eqref{5-psisconL2a}
and (\ref{psi0convstr}), we immediately
have that, as $L \rightarrow \infty$ (and $\Delta t \rightarrow 0_+$),
the first term on the right-hand
side of \eqref{partialint} converges to the first term on the left-hand side of
\eqref{eqpsinconP} and the second term on the right-hand side of \eqref{partialint}
converges to $-\int_{\Omega \times D} \hat\psi_0 (\xt, \qt)\, \hat\varphi(\xt,\qt,0)
\dq \dx$, resulting
in the first term on the right-hand side of \eqref{eqpsinconP}.
On rewriting $\nabx \hpsistL^{\Delta t,+} = 2\,\sqrt{\hpsistL^{\Delta t,+}}\,
\nabx \sqrt{\hpsistL^{\Delta t,+}}$, and similarly $\nabq \hpsistL^{\Delta t,+}$,
it is a simple matter to pass to the limit
$L \rightarrow \infty$ (and $\Delta t \rightarrow 0_+$) in the remaining terms of
\eqref{hpsiOBLncon} using (\ref{hpsiOBqr}), (\ref{5-psiwconH1a},b,d), (\ref{ML2conv})
and (\ref{utOBDtsatb},b) to obtain (\ref{eqpsinconP})
for all $\hat \varphi \in {\cal E}$.

Finally, we note
that for any $s \geq 0$, $C^\infty(\overline{\Omega};C^\infty_0(D))$ is dense
in $L^2(\Omega;H^s(D))\cap H^s(\Omega;L^2(D))
=H^s(\Omega \times D)$, and so ${\cal E}$
is a dense linear subspace
of the linear space of functions
$W^{1,1}(0,T;H^s(\Omega \times D))$ vanishing at $t=T$.
It follows from this, (\ref{hpsiOBqr},b), (\ref{relent-fisher})
and (\ref{hpsidata}) that
(\ref{eqpsinconP})
holds for all $\hat \varphi \in W^{1,1}(0,T;H^s(\Omega \times D))$,
for any $s > d+1$, vanishing at $t=T$.
\end{proof}

\begin{lemma}\label{convsig}
Let the assumptions of Theorem \ref{convfinal} hold.
Then we have, on possibly extracting a further subsequence of $\{\hpsistL^{\Delta t}\}_{L>1}$,
that, as $L\rightarrow \infty$ (and thereby $\Delta t \rightarrow 0_+$),
\begin{subequations}
\begin{alignat}{2}
\sigtt(M\,\beta^L(\hpsistL^{\Delta t,+})), \,\sigtt(M\,\hpsistL^{\Delta t,+})
&\rightarrow
\sigtt(M\,\hpsist)
&&\quad \mbox{strongly in }
L^p(0,T;\Ltt^1(\Omega)),
\label{CsconL1a}\\
\sigtt(M\,\hpsistL^{\Delta t,+})
&\rightarrow
\sigtt(M\,\hpsist)
&&\quad \mbox{weak* in }
L^{\infty}(0,T;\Ltt^2(\Omega)),
\label{CwconLinfa}\\
\sigtt(M\,\hpsist^{\Delta t,+})
&\rightarrow 
\sigtt(M\,\hpsist)
&&\quad \mbox{weakly in }
L^{2}(0,T;\Htt^1(\Omega));
\label{CwconL2a}
\end{alignat}
\end{subequations}
and
\begin{subequations}
\begin{alignat}{2}
\rho(M\,\hpsistL^{\Delta t,+})
&\rightarrow
\rho(M\,\hpsist)
&&\quad \mbox{strongly in }
L^p(0,T;L^1(\Omega)),
\label{rsconL1}\\
\rho(M\,\hpsistL^{\Delta t,+})
&\rightarrow
\rho(M\,\hpsist)
&&\quad \mbox{weak* in }
L^{\infty}(0,T;L^2(\Omega)),
\label{rwconL2a}\\
\rho(M\,\hpsistL^{\Delta t,+})
&\rightarrow 
\rho(M\,\hpsist)
&&\quad \mbox{weakly in }
L^{2}(0,T;H^1(\Omega)),
\label{rwconWH1}
\end{alignat}
\end{subequations}
for any $p \in [1,\infty)$. In addition, it follows that
\begin{align}
&\sigtt(M\,\hpsist) \in
L^\infty(0,T;\Ltt^2(\Omega)) \cap L^2(0,T;\Htt^1(\Omega))
\nonumber \\
\quad\mbox{and}\quad
&\rho(M\,\hpsist) \in
L^\infty(0,T;\Ltt^2(\Omega)) \cap L^2(0,T;\Htt^1(\Omega))
\label{sigreq}
\end{align}
satisfy
\begin{subequations}
\begin{align}\label{eqsigpsi}
&\displaystyle-\int_{0}^{T}
\int_{\Omega} \sigtt(M\,\hpsist)  :  \frac{\partial \xitt}{\partial t}
\dx \dt
\nonumber \\
& \qquad + \int_{0}^T \int_{\Omega}
\left[ \left[ (\utst \cdot \nabx) \sigtt(M\,\hpsist) \right] : \xitt
+ \epsilon \,\nabx \sigtt(M\,\hpsist)
::
\nabx \xitt \right] \dx \dt
\nonumber
\\
&\qquad + \int_{0}^T \int_\Omega
\left[ \frac{1}{2\,\lambda}\left( \sigtt(M\,\hpsist)- \rho(M\,\hpsist)\,\Itt\right) -
\left((\nabxtt \utst)\,\sigtt(M\,\hpsist) + \sigtt(M\,\hpsist)\,(\nabxtt \utst)^{\rm T} \right)
\right] :
\xitt  \dx  \dt \nonumber \\
&\quad
=
\int_\Omega [\sigtt(M\,\hpsi_0)](\xt) :  \xitt(\xt ,0) \dx
\qquad
\forall \xitt \in W^{1,1}(0,T;\Htt^1(\Omega)) \mbox{ {\rm with} $\xitt(\cdot,T)=\zerott$},
\\
&\displaystyle-\int_{0}^{T}
\int_{\Omega} \rho(M\,\hpsist) \, \frac{\partial \eta}{\partial t}
\dx \dt
+ \int_{0}^T \int_{\Omega}
\left[ \left[ (\utst \cdot \nabx) \rho(M\,\hpsist) \right] \,\eta
+ \epsilon \,\nabx \rho(M\,\hpsist)
\cdot
\nabx \eta \right] \dx \dt
\nonumber\\
&\quad
=
\int_\Omega [\rho(M\,\hpsi_0)](\xt) \,\eta(\xt ,0) \dx
\qquad
\forall \eta \in W^{1,1}(0,T;H^1(\Omega)) \mbox{ {\rm with} $\eta(\cdot,T)=0$}.
\label{eqrho}
\end{align}
\end{subequations}
\end{lemma}
\begin{proof}
The desired result (\ref{CsconL1a}) follows immediately from (\ref{C1}) and
(\ref{5-psisconL2a}) as
\begin{align}
\|\sigtt(M\,\hat \varphi)\|_{L^p(0,T;L^1(\Omega))}
\leq \|\,|\qt|^2\, \hat \varphi\|_{L^p(0,T;L^1_M(\Omega \times D))}
\quad \forall \hat \varphi\mbox{ s.t. } |\qt|^2\, \hat \varphi \in L^p(0,T;L^1_M(\Omega \times D)).
\label{sigL1bd}
\end{align}
The desired result (\ref{rsconL1}) follows immediately from
(\ref{5-psisconL2a}) as
\begin{align}
\|\rho(M\,\hat \varphi)\|_{L^p(0,T;L^1(\Omega))}
\leq \|\hat \varphi\|_{L^p(0,T;L^1_M(\Omega \times D))}
\qquad \forall \hat \varphi \in L^p(0,T;L^1_M(\Omega \times D)).
\label{rhoL1bd}
\end{align}
The desired results
(\ref{CwconLinfa},c) and
(\ref{rwconL2a},c) follow immediately from (\ref{psiGijbdfinlem},b), (\ref{CsconL1a})
and (\ref{rsconL1}). Hence, (\ref{sigreq}) holds.

We now pass to the limit $L \rightarrow \infty$
(and $\Delta t \rightarrow 0_+$) in (S$^{\Delta t}_L$),
(\ref{psiGijcont},b).
We shall take at first
$\xitt \in
C^1([0,T];\Ctt^\infty(\overline{\Omega}))$
with $\xitt(\cdot,\cdot,T)=\zerott$ and
$\eta \in
C^1([0,T];C^\infty(\overline{\Omega}))$
with $\eta(\cdot,\cdot,T)=0$.
 Integration by parts with respect to $t$ on the first term in \eqref{psiGijcont} gives,
for all $\xitt \in C^1([0,T];\Ctt^\infty(\overline{\Omega}))$
with $\xitt(\cdot,\cdot,T)=\zerott$,
\begin{align}
\int_{0}^T \int_{\Omega} \frac{ \partial \sigtt(M\,\hpsistL^{\Delta t})}{\partial t} :
\xitt \dx \dt
= &- \int_0^T \int_{\Omega}
\sigtt(M\, \hpsistL^{\Delta t})\, \frac{\partial \xitt}{\partial t} \dx \dt
- \int_{\Omega} [\sigtt( M\,\hpsi^0)](\xt)\,
\xitt(\xt, 0) \dx.
\label{partialinta}
\end{align}
%
Using \eqref{CwconLinfa}, (\ref{eqCttbd})
and (\ref{psi0convstr}), we immediately
have that, as $L \rightarrow \infty$ (and $\Delta t \rightarrow 0_+$),
the first term on the right-hand
side of \eqref{partialint} converges to the first term on the left-hand side of
\eqref{eqsigpsi} and the second term on the right-hand side of \eqref{partialinta}
converges to $$-\int_{\Omega \times D} [\sigtt(M \hpsi_0)](\xt):\xitt(\xt,0)
\dx,$$ resulting
in the first term on the right-hand side of \eqref{eqsigpsi}.
It is a simple matter to pass to the limit
$L \rightarrow \infty$ (and $\Delta t \rightarrow 0_+$) in the remaining terms of
(\ref{psiGijcont}) using (\ref{CsconL1a}--c), (\ref{rwconL2a})
and (\ref{utOBDtsatb},b) to obtain (\ref{eqsigpsi})
for all $\xitt \in
C^1([0,T];\Ctt^\infty(\overline{\Omega}))$ such that
$\xitt(\cdot,T)=\zerott$.
Similarly to the above, we can pass to the limit
$L \rightarrow \infty$ (and $\Delta t \rightarrow 0_+$) in
(\ref{rhocont}) using (\ref{rsconL1}--c) and
(\ref{utOBDtsatb},b) to obtain (\ref{eqrho})
for all $\eta \in
C^1([0,T];C^\infty(\overline{\Omega}))$ such that
$\eta(\cdot,T)=0$.

Finally, we note
that
$C^1([0,T];C^\infty(\overline{\Omega}))$
is a dense linear subspace of the linear space $W^{1,1}(0,T;H^1(\Omega))$.
It follows from this, (\ref{sigreq})
and (\ref{hpsidata}) that
(\ref{eqsigpsi})
holds for all $\xitt \in W^{1,1}(0,T;\Htt^1(\Omega))$ vanishing at $t=T$
and (\ref{eqrho})
holds for all $\eta \in W^{1,1}(0,T;H^1(\Omega))$ vanishing at $t=T$.
\end{proof}

\begin{remark}\label{rho=1rem}
{\rm As we assume that
$[\rho(M\,\hpsi_0)](\xt)=\int_D M(\qt)\,\hpsi_0(\xt,\qt) \dq = 1
\mbox{ for a.e. } \xt \in \Omega$, recall (\ref{hpsi0}), it follows that
$\rho(M\,\hpsist) \equiv 1$ is the unique solution to the linear problem
(\ref{eqrho}). Similarly, $\rho(M\,\hpsistL^n)\equiv 1$, $n=0,\ldots,N$ is the unique
solution of (\ref{rhoneq}).}
\end{remark}

\section{The Hookean dumbbell model}\label{HDM}
\setcounter{equation}{0}

Putting together the results in the previous two sections we have the following result.

\begin{theorem}\label{EXP}
Let $d=2$ and $\partial \Omega \in C^{2,\mu}$, for $\mu \in (0,1)$.
Let
\begin{align}
\ut_0 \in \Vt 
\qquad
\mbox{and} \qquad
\ft \in L^2(0,T;\Lt^2(\Omega)) \cap
L^\mathfrak{r}(0,T;\Lt^\mathfrak{s}(\Omega)),
\end{align}
for $\mathfrak{r} \in (1,\frac{4}{3}]$ and $\mathfrak{s} \in (2,4)$.
Let $\hat \psi_0$ satisfy (\ref{hpsi0}) with $\sigtt_0 :=\sigtt(M\,\hpsi_0)
\in W^{1,2}_n(\Omega)$.
It follows that there exist $\utOB$ and $\sigOB=\sigOB^{\rm T}$
satisfying (\ref{utOBregf},b) and solving the Oldroyd-B system (\ref{equtOB},b).

In addition, there exists $\hpsiOB$ satisfying
\begin{subequations}
\begin{alignat}{2}
|\qt|^r \, \hpsiOB &\in L^\infty(0,T;L^1_{M}(\Omega \times D)), \qquad
&&\mbox{for any } r \in [0,\infty),
\label{hpsiOBqrF}
\end{alignat}
%
with
\begin{equation}\label{mass-conservedF}
\hpsiOB \geq 0 \mbox{ a.e. on $\Omega \times D \times [0,T]$} \quad \mbox{and} \quad
\int_D M(\qt)\,\hpsiOB(\xt,\qt,t) \dq =
1 \mbox{ for a.e. $(x,t) \in \Omega \times [0,T]$},
\end{equation}
%
%
\begin{equation}\label{relent-fisherF}
\mathcal{F}(\hpsiOB) \in L^\infty(0,T;L^1_M(\Omega\times D))\quad
\mbox{and}\quad \sqrt{\hpsiOB} \in L^{2}(0,T;H^1_M(\Omega \times D));
\end{equation}
\end{subequations}
and solving, for $s>d+1$,
\begin{align}\label{eqpsinconPF}
&-\int_{0}^T
\int_{\Omega \times D} M\,\hpsiOB\, \frac{\partial \hat\varphi}{\partial t}
\dq \dx \dt
+ \int_{0}^T \int_{\Omega \times D} M\,\left[
\epsilon\, \nabx \hpsiOB - \utOB \,\hpsiOB \right]\cdot\, \nabx
\hat \varphi
\,\dq \dx \dt
\nonumber \\
&\quad +
\int_{0}^T \int_{\Omega \times D} M\,\left[
\frac{1}{4\,\lambda}\,\nabq \hpsiOB
-
\left[(\nabxtt \utOB)
\,\qt\right] \hpsiOB \right] \cdot \nabq
\hat \varphi
\dq \dx \dt \nonumber \\
& \quad \quad
= \int_{\Omega \times D} M\, \hat\psi_0(\xt,\qt)\,\hat\varphi(\xt,\qt,0)
\dq \dx
\quad \forall \hat \varphi \in W^{1,1}(0,T;H^s(\Omega\times D))
\mbox{ with $\hat\varphi(\cdot,\cdot,T)=0$}.
\end{align}

Moreover, we have that $\sigOB=\sigtt(M\,\hpsiOB)$.
Hence, $(\utOB,\hpsiOB)$ solve the Hookean dumbbell model, (\ref{equtOB})
and (\ref{eqpsinconPF}); and $(\utOB,\sigOB=\sigtt(M\,\hpsiOB))$ solve the
Oldroyd-B model, (\ref{equtOB},b).
\end{theorem}
\begin{proof}
Theorem \ref{thOBreg} immediately yields the existence of
$\utOB$ and $\sigOB=\sigOB^{\rm T}$
satisfying (\ref{utOBregf},b) and solving the Oldroyd-B system (\ref{equtOB},b).
By applying Theorem \ref{convfinal} with $\utst=\utOB$ yields
the existence of $\hpsiOB$ satisfying (\ref{hpsiOBqrF}--c)
and solving (\ref{eqpsinconPF}). Finally, by applying Lemma \ref{convsig} with $\utst=\utOB$ yields that
$\sigtt(M\,\hpsiOB) \in L^\infty(0,T;\Ltt^2(\Omega))\cap L^2(0,T;\Htt^1(\Omega))$
satisfies
\begin{align}\label{eqsigpsiF}
&\displaystyle-\int_{0}^{T}
\int_{\Omega} \sigtt(M\,\hpsiOB) : \frac{\partial \xitt}{\partial t}
\dx \dt
\nonumber \\
& \qquad + \int_{0}^T \int_{\Omega}
\left[ \left[ (\utOB \cdot \nabx) \sigtt(M\,\hpsiOB) \right] : \xitt
+ \epsilon \,\nabx \sigtt(M\,\hpsiOB)
::
\nabx \xitt \right] \dx \dt
\nonumber
\\
&\qquad + \int_{0}^T \int_\Omega
\left[ \frac{1}{2\,\lambda}\left( \sigtt(M\,\hpsiOB)- \Itt\right) -
\left((\nabxtt \utOB)\,\sigtt(M\,\hpsiOB) + \sigtt(M\,\hpsiOB)\,(\nabxtt \utOB)^{\rm T} \right)
\right] :
\xitt  \dx  \dt \nonumber \\
&\quad
=
\int_\Omega [\sigtt(M\,\hpsi_0)](\xt) : \xitt(\xt ,0) \dx
\qquad
\forall \xitt \in W^{1,1}(0,T;\Htt^1(\Omega)) \mbox{ {\rm with} $\xitt(\cdot,T)=\zerott$},
\end{align}
where we have noted from Remark \ref{rho=1rem} that $\rho(M\,\hpsiOB)\equiv 1$.
Comparing (\ref{eqsigpsiF}) and (\ref{eqsigOB}), and recalling
the uniqueness result of Theorem \ref{thOBreg}
yields that $\sigOB= \sigtt(M\,\hpsiOB)$.

Hence, $(\utOB,\hpsiOB)$ solve the Hookean dumbbell model, (\ref{equtOB})
and (\ref{eqpsinconPF}); and $(\utOB,\sigOB=\sigtt(M\,\hpsiOB))$ solve the
Oldroyd-B model, (\ref{equtOB},b).
\end{proof}

\section{Concluding remarks: subsolutions and stress-defect measure}\label{remexist}\setcounter{equation}{0}
We close the paper with a discussion aimed at explaining why an alternative,
apparently more direct, approach to proving the existence of large-data global weak
solutions to problem (P), along the lines of our paper \cite{BS2010-hookean}, fails
in the case of the Hookean model, and why we had to resort in this paper to a different
line of reasoning than in \cite{BS2010-hookean}. We shall describe this direct approach
for both $d=2$ and $d=3$ space dimensions, and will also indicate the improvements
that can be achieved in the case of $d=2$. As will be explained below, however, these
improvements are still insufficient to complete the proof of existence of
global weak solutions to the model with that approach, even in the case of
$d=2$. On the other hand, for both $d=2$ and $d=3$, this direct approach still
establishes the existence of global weak \textit{subsolutions},
in a sense to be made precise below.

A direct existence proof for problem (P) would follow the approach in \cite{BS2010-hookean}
based on a discrete-in-time regularization of (P), similar to (FP$_{L}^{\Delta t}$),
(\ref{hpsiOBLncon}):

{\boldmath $({\rm P}_{L}^{\Delta t})$}:
Find $(\utLDtp(t),\hpsiLDtp(t)) \in
\Vt \times (\hat X \cap \hat Z_2)$
s.t.
\begin{subequations}
\begin{align}\label{ULDtcon}
&\displaystyle\int_{0}^{T} \int_\Omega  \frac{\partial \utLDt}{\partial t}\cdot
\wt \dx \dt
+ \int_{0}^T \int_{\Omega}
\left[ \left[ (\utLDtm \cdot \nabx) \utLDtp \right]\,\cdot\,\wt
+ \nu \,\nabxtt \utLDtp
:
\nabx \wt \right] \dx \dt
\nonumber\\
&\hspace{1mm}=\int_{0}^T
\left[ \langle \ft^{\Delta t,+}, \wt\rangle_{H^1_0(\Omega)}
- k\,\int_{\Omega}
\sigtt(M\,
\hpsiLDtp): \nabx
\wt \dx \right] \dt
\qquad \forall \wt \in L^1(0,T;\Vt ),
\end{align}
\begin{align}\label{HPLDtcon}
&\int_{0}^T \int_{\Omega \times D}
M\,\frac{ \partial \hpsiLDt}{\partial t}\,
\hat \varphi \dq \dx \dt
+ \int_{0}^T \int_{\Omega \times D} M\,\left[
\epsilon\, \nabx \hpsiLDtp- \utLDtm\,\hpsiLDtp \right]\cdot\, \nabx
\hat \varphi
\,\dq \dx \dt
\nonumber \\
& \hspace{0.5cm} +
\int_{0}^T \int_{\Omega \times D} M
\left[ \frac{1}{4\,\lambda}
\nabq \hpsiLDtp
-[\,(\nabxtt \utLDtp)
\,\qt]\,
\beta^L(\hpsiLDtp)\right] \cdot \nabq
\hat \varphi
\,\dq \dx \dt = 0
\nonumber \\
& \hspace{3.7in}
\qquad \forall \hat \varphi \in L^1(0,T;\hat X);
\end{align}
\end{subequations}
subject to the initial conditions
$\utLDt(\cdot,0)= \ut^0 \in \Ht$ and
$\hpsiLDt(\cdot,\cdot,0) = \beta^L(\hat \psi^0(\cdot,\cdot)) \in \hat Z_2$.
Following the line of reasoning in \cite{BS2010-hookean} one can establish,
for $\Omega \subset {\mathbb R}^d$, $d=2$ or 3,
with $\partial \Omega \in C^{0,1}$, the existence of
a solution to (P$^{\Delta t}_L$) satisfying the uniform bounds
\begin{align}
&{\rm ess.sup}_{t\in [0,T]} \|\utLDtpm(t)\|_{L^2(\Omega)}^2
+ \frac{1}{\Delta t} \int_0^T \|\utLDtp - \utLDtm\|_{L^2(\Omega)}^2
\dd t + \nu \int_0^T \|\nabxtt \utLDtp\|_{L^2(\Omega)}^2 \dd t\nonumber\\
&\qquad +  {\rm ess.sup}_{t\in [0,T]} \,2k\,\int_{\Omega \times D}\!\!
M \mathcal{F}(\hpsiLDtpm(t)) \dq \dx + \frac{k}{\Delta t\,L}
\int_0^T \int_{\Omega \times D}\!\! M (\hpsiLDtp
- \hpsiLDtm)^2 \dq \dx \dd t
\nonumber \\
&\qquad +\, 8k\,\varepsilon \int_0^T \int_{\Omega \times D} M
|\nabx \sqrt{\hpsiLDtp} |^2 \dq \dx \dd t
+\, \frac{k}{\lambda}  \int_0^T \int_{\Omega \times D}M\,
\,|\nabq \sqrt{\hpsiLDtp}|^2 \,\dq \dx \dd t
\leq C.
\label{PLDtent}
\end{align}
In addition, one can establish that $\hpsiLDtpm \geq 0$ a.e.\ in $\Omega \times D
\times [0,T]$, $\rho(\hpsiLDtpm) =1$ a.e.\ in $\Omega \times [0,T]$, and
\begin{subequations}
\begin{align}
\displaystyle
\left|\int_{0}^{T}\!\! \int_\Omega  \frac{\partial \utLDt}{\partial t}\cdot
\wt \dx \dt\right|&\leq C\,\|\wt\|_{L^2(0,T;V_\mu)}
\qquad \forall \wt \in L^2(0,T;\Vt_\mu), \quad \mbox{for }\mu > \frac{d}{2}+1,
\label{u-time-bound}\\
\left|\int_{0}^T\int_{\Omega \times D} M\, \frac{\partial
\hpsiLDt}{\partial t}\,
\hat \varphi
\dq \dx \dt\right|
& \leq C\,\|\hat \varphi\|_{L^2(0,T;W^{1,\infty}(\Omega \times D)}
\qquad \forall \hat \varphi \in L^2(0,T;W^{1,\infty}(\Omega \times D)),
\label{psi-time-bound}
\end{align}
\end{subequations}
where $\Vt_\mu$ is the completion of
$\{\wt \in \Ct^\infty_0(\Omega) : \nabx \cdot \wt = 0 \mbox{ in } \Omega \}$
in $\Ht_0^1(\Omega)\cap \Wt^{\mu,2}(\Omega)$.
It follows from the fourth bound in (\ref{PLDtent}), (\ref{qt2bd}), (\ref{MN})
and (\ref{U}) that
\begin{align}
{\rm ess.sup}_{t\in [0,T]} \,\int_{\Omega \times D}\!\!
M\,(1+|\qt|^2)\,\hpsiLDtpm(t) \dq \dx \leq C.
\label{HPq2}
\end{align}

Motivated by the form of the denominator of the prefactor of the multiple integral
in the fifth term on the
left-hand side of \eqref{PLDtent} we let $\Delta t = o(L^{-1})$ as $L \rightarrow \infty$,
so as to drive the multiple integral appearing in that term to $0$ in the limit.
Then, one can establish
the existence of a subsequence of $\{(\utLDt, \hpsiLDt)\}_{L >1}$ (not indicated),
and a pair of functions $(\ut, \hpsi)$ such that
\begin{align}\label{Hpsia-u}
\ut \in L^{\infty}(0,T;\Lt^2(\Omega))\cap L^{2}(0,T;\Vt)
\cap H^1(0,T;\Vt'_\mu),\quad \mu > \frac{d}{2} + 1,
\end{align}
and
\begin{align}\label{Hpsia-psi}\hpsi \in L^1(0,T;L^1_{M}(\Omega \times D)) \cap
H^1(0,T; M^{-1}[H^s(\Omega \times D)]'),
\quad s>d+1,
\end{align}
with $\hpsi \geq 0$ a.e. on $\Omega \times D \times [0,T]$,
$\int_D M(\qt)\,\hpsi(\xt,\qt,t) \dq =
1$ for a.e.\ $(x,t) \in \Omega \times [0,T]$,
\begin{equation}\label{Hpsia}
\mathcal{F}(\hpsi) \in L^\infty(0,T;L^1_M(\Omega\times D))\quad
\mbox{and}\quad \sqrt{\hpsi} \in L^{2}(0,T;H^1_M(\Omega \times D)),
\end{equation}
whereby $(1+|\qt|^2)\,\hpsi \in L^\infty(0,T; L^1_{M}(\Omega \times D))$;
such that, as $L\rightarrow \infty$ (and thereby $\Delta t \rightarrow 0_+$),
\begin{subequations}
\begin{alignat}{2}
&\utLDtpm \rightarrow \ut \qquad &&\mbox{weak* in }
L^{\infty}(0,T;{\Lt}^2(\Omega)), \label{5-uwconL2a}\\
&\utLDtpm \rightarrow \ut \qquad &&\mbox{weakly in }
L^{2}(0,T;\Vt), \label{5-uwconH1a}\\
&~\frac{\partial \utLDt}{\partial t} \rightarrow  \frac{\partial \ut}{\partial t}
\qquad &&\mbox{weakly in }
L^2(0,T;\Vt_\mu'), \label{5-utwconL2a}\\
&\utLDtpm \rightarrow \ut \qquad &&\mbox{strongly in }
L^{2}(0,T;\Lt^{\mathfrak{r}}(\Omega)), \label{5-usconL2a}
\end{alignat}
\end{subequations}
where $\mathfrak{r} \in [1,\infty)$ if $d=2$ and
$\mathfrak{r} \in [1,6)$ if $d=3$;
and
\begin{subequations}
\begin{alignat}{2}
\bet
M^{\frac{1}{2}}\,\nabx \sqrt{\hpsiLDtpm} &\rightarrow M^{\frac{1}{2}}\,\nabx
\sqrt{\hpsi}
&&\qquad \mbox{weakly in } L^{2}(0,T;\Lt^2(\Omega\times D)), \label{6-psiwconH1a}\\
\bet
M^{\frac{1}{2}}\,\nabq \sqrt{\hpsiLDtpm} &\rightarrow M^{\frac{1}{2}}\,\nabq
\sqrt{\hpsi}
&&\qquad \mbox{weakly in } L^{2}(0,T;\Lt^2(\Omega\times D)), \label{6-psiwconH1xa}\\
\bet
M\,\frac{\partial \hpsiLDt} {\partial t} &\rightarrow
M\,\frac{\partial \hpsi}{\partial t}
&&\qquad \mbox{weakly in }
L^2(0,T;[H^s(\Omega\times D)]'), \label{6-psitwconL2a}\\
\bet
|\qt|^r\,\beta^L(\hpsiLDtpm),\,
|\qt|^r\,\hpsiLDtpm &\rightarrow
|\qt|^r\,\hpsi
&&\qquad \mbox{strongly in }
L^{p}(0,T;L^{1}_{M}(\Omega\times D)),\label{6-psisconL2a}
\end{alignat}
for all $p \in [1,\infty)$ and any $r \in [0,2)$.
\end{subequations}

The result (\ref{6-psisconL2a}) for $r=0$
follows using Dubinski{\u\i}'s theorem,
Theorem \ref{thm:Dubinski}, and Lemma \ref{le:supplementary},
as in the proof of Theorem
\ref{convfinal}.
The result (\ref{6-psisconL2a}) for $r \in (0,2)$
is proved similarly to
(\ref{qalpha}).
We have for any $\alpha \in [1,2]$ and any $p \in [2,\infty)$ that
\begin{align}
&\|\,|\qt|^{\alpha}\,(\hpsi-\hpsiLDtpm)\,
\|_{L^p(0,T;L^1_M(\Omega \times D))}
\nonumber \\
& \quad \leq \left( \int_0^T  \|\,|\qt|^{2}\,(\sqrt{\hpsi}
+\sqrt{\hpsiLDtpm})^2\,
\|_{L^1_M(\Omega \times D)}^{\frac{p}{2}}\,
\|\,|\qt|^{2(\alpha-1)}\,(\sqrt{\hpsi}-\sqrt{\hpsiLDtpm})^2\,
\|_{L^1_M(\Omega \times D)}^{\frac{p}{2}} \dt \right)^{\frac{1}{p}}
\nonumber \\
& \quad \leq C\,
\|\,|\qt|^{2(\alpha-1)}\,(\hpsi-\hpsiLDtpm)\,
\|_{L^{\frac{p}{2}}(0,T;L^1_M(\Omega \times D))}^{\frac{1}{2}},
\label{qalpha1}
\end{align}
where we have noted (\ref{HPq2})
and that $(1+|\qt|^2)\,\hpsi$
belongs to $L^\infty(0,T;L^1_M(\Omega \times D))$.
Therefore, it follows from (\ref{qalpha1}) that, for any
$\alpha \in [1,2]$ and any $p \in [1,\infty)$,
\begin{align}
\|\,|\qt|^{\alpha}\,(\hpsi-\hpsiLDtpm)\,
\|_{L^{p}(0,T;L^1_M(\Omega \times D))}
\leq \left[ C(T,p)\,
\|\,|\qt|^{2(\alpha-1)}\,(\hpsi-\hpsiLDtpm)\,
\|_{L^{p}(0,T;L^1_M(\Omega \times D))} \right]^{\frac{1}{2}}.
\label{qalphaC}
\end{align}
We deduce from (\ref{qalphaC}) that, for
any $p \in [1,\infty)$ and for any $m \in {\mathbb N}$,
\begin{align}
&\|\,|\qt|^{2-\frac{1}{2^{m-1}}}\,(\hpsi-\hpsiLDtpm)\,
\|_{L^p(0,T;L^1_M(\Omega \times D))}
\leq C(T,p)^{\frac{1}{2^m}} \,
\|\hpsi-\hpsiLDtpm\|_{L^p(0,T;L^1_M(\Omega \times D))}^{\frac{1}{2^{m}}}.
\label{qalpham}
\end{align}
Hence, the desired result (\ref{5-psisconL2a}) follows immediately from (\ref{qalpham}).

Unfortunately, although \eqref{C1} and \eqref{HPq2} imply that
\[ \|\sigtt(M\,\hpsiLDtp)\|_{L^\infty(0,T;L^1(\Omega))} \leq C,\]
where $C$ is independent of $L$ and $\Delta t$, and therefore there exists a symmetric positive semidefinite
$\ztt \in L^{\infty}(0,T;[\Ctt(\overline\Omega)]')$ such that
\begin{equation}
\sigtt(M\,\hpsiLDtp)
\rightarrow
\ztt
\qquad \mbox{weak* in }
L^{\infty}(0,T;[\Ctt(\overline\Omega)]'),
\label{6-CwconL2a3}
\end{equation}
it does not follow that $\ztt = \sigtt(M\,\hpsi)$. The reason why we cannot identify $\ztt$
with $\sigtt(M\,\hpsi)$ is because \eqref{6-psisconL2a} only holds for $r \in [0,2)$ (but
not for $r=2$), and the boundedness of the sequence
$\{|\qt|^2\, \hpsiLDtp\}_{L \geq 1}$ in $L^\infty(0,T;L^1_M(\Omega \times D))$,
which is the strongest uniform bound that we have on $|\qt|^2\, \hpsiLDtp$, does not suffice
to deduce that
\[ \int_D M\, \qt\, \qt^{\rm T}\,  \hpsiLDtp \dd \qt \rightarrow \int_D M\, \qt\, \qt^{\rm T}\, \hpsi \dd \qt
 \quad \mbox{weak* in }
L^{\infty}(0,T;[\Ctt(\overline\Omega)]').
\]
If this were true,  it would of course automatically follow by the uniqueness of the weak* limit that $\ztt = \sigtt(M\,\hpsi)$, thus completing the proof of the existence of global
weak solutions to the Hookean model in both $d=2$ and $d=3$ dimensions.

Nevertheless, motivated by notion of \textit{weak solution with a defect measure} in the work of Feireisl
(cf. in particular Sec. 4.3.2 in \cite{Feir1}, the discussion on p.21 in \cite{FN}, and \cite{FRSZ}, as well as pp.~\!7, 8 in the work of
DiPerna \& Lions \cite{DiL} and Definition 1.3 in the paper of Alexandre \& Villani \cite{AV}),
we will show that this direct approach implies the existence of a \textit{subsolution} to the Navier--Stokes--Fokker--Planck
system, which is a weak solution with a defect measure, in a sense to be made precise below, for both $d=2$ and $d=3$.
We begin by noting that, since $\hpsiLDtp \geq 0$ in $\Omega \times D \times [0,T]$, we have that
\[\int_0^T \int_\Omega \left(\int_D M\, \qt\, \qt^{\rm T}\,  \hpsiLDtp \dd \qt \right) :\xitt \dd x \dd t  \geq
\int_0^T \int_\Omega \left(\int_D \chi_R\, M\, \qt\, \qt^{\rm T}\,  \hpsiLDtp \dd \qt \right):\xitt \dd x \dd t, \]
where $\qt \mapsto\chi_R(\qt)$ is the characteristic function of the ball of radius $R$ in $D$ centred at the origin, and
$\xitt \in L^1(0,T;\Ctt(\overline\Omega))$ is a symmetric positive semidefinite $d \times d$ matrix-valued test function.
By fixing $R>0$ and passing to the limit $L \rightarrow \infty$ (with $\Delta t = o(L^{-1})$ as $L \rightarrow \infty$)
using \eqref{6-psisconL2a} with $r=0$, it then follows that
\[ \int_0^T \langle \ztt , \xitt \rangle_{C(\overline\Omega)} \dd t
\geq  \int_0^T \int_\Omega \left(\int_D \chi_R\, M\, \qt\, \qt^{\rm T}\,  \hpsi \dd \qt \right):\xitt \dd x \dd t, \]
where $\langle \cdot , \cdot \rangle_{C(\overline\Omega)}$ denotes the duality pairing between $[\Ctt(\overline\Omega)]'$
and $\Ctt(\overline\Omega)$.

In order to pass to the limit $R \rightarrow \infty$ in the last inequality, we note that the sequence of real-valued
nonnegative functions
$\{\chi_R\, M\, \qt\, \qt^{\rm T}\,\hpsi :\xitt\}_{R>0}$, where $|\qt|^2\,\hpsi \in L^\infty(0,T;L^1_M(\Omega \times D))$, is nondecreasing and converges to
$M\, \qt\, \qt^{\rm T}\,\hpsi:\xitt$
a.e. on $\Omega \times D \times (0,T)$ as $R \rightarrow \infty$. Hence, by Lebesgue's monotone convergence theorem, we can pass to the limit $R \rightarrow \infty$
in the last inequality, resulting in
\begin{align*}
\int_0^T \langle \ztt , \xitt \rangle_{C(\overline\Omega)} \dd t
&\geq  \int_0^T \int_\Omega \left(\int_D M\, \qt\, \qt^{\rm T}\,  \hpsi \dd \qt \right) : \xitt  \dd x \dd t
\\
&= \int_0^T \int_\Omega \sigtt(M\, \hpsi) : \xitt  \dd x \dd t = \int_0^T \langle \sigtt(M\, \hpsi) , \xitt \rangle_{C(\overline\Omega)} \dd t,
\end{align*}
for all  symmetric positive semidefinite $d \times d$ matrix-valued test functions $\xitt \in L^1(0,T;\Ctt(\overline\Omega))$. In other words,
\[\ztt \geq \sigtt(M\, \hpsi)\qquad \mbox{in $L^\infty(0,T;[\Ctt(\overline\Omega)]')$},\qquad \mbox{with}
\qquad \sigtt(M\, \hpsi)
\in L^\infty(0,T;\Ltt^1_M(\Omega \times D)),\]
where the last inclusion is a direct consequence of the fact that $(1+|\qt|^2)\,\hpsi \in L^\infty(0,T;L^1_M(\Omega \times D))$.

One can now pass to the limit $L\rightarrow \infty$ (and thereby $\Delta t \rightarrow 0_+$)
in all other terms of (\ref{ULDtcon},b) using
(\ref{5-uwconL2a}--d) and (\ref{6-psiwconH1a}--d)
similarly to (\ref{hpsiOBLncon}) as in the proof of Theorem \ref{convfinal}.

Thus we have shown the existence of $(\ut, \ztt, \hpsi)$, such that $\ut$ and $\hpsi$
satisfy \eqref{Hpsia-u}--\eqref{Hpsia},
with $\hpsi \geq 0$ a.e. on $\Omega \times D \times [0,T]$, $\int_D M(\qt)\,\hpsi(\xt,\qt,t) \dd q = 1$
for a.e. $(x,t) \in \Omega \times [0,T]$,
and
\[ \ztt \in L^\infty(0,T;[\Ctt(\overline\Omega)]'), \qquad \ztt = \ztt^{\rm T} \geq \sigtt(M\hpsi) \geq \zerott,\]
satisfying, for $\mu>\frac{d}{2}+1$,
\begin{subequations}
\begin{align}\label{equtNS-direct}
&\displaystyle-\int_{0}^{T}
\int_{\Omega} \ut \cdot \frac{\partial \wt}{\partial t}
\dx \dt
+ \int_{0}^T \int_{\Omega}
\left[ \left[ (\ut \cdot \nabx) \ut \right]\,\cdot\,\wt
+ \nu \,\nabxtt \ut
:
\nabx \wt \right] \dx \dt
\nonumber
\\
&\qquad = \int_{0}^T
\left[ \langle \ft, \wt\rangle_{H^1_0(\Omega)}
- k\langle \ztt , \nabx \wt\rangle_{C(\overline\Omega)}
\right] \dt
+ \int_\Omega \ut_0(\xt) \cdot \wt(\xt ,0) \dx
\nonumber
\\
& \hspace{2.8in}
\forall \wt \in W^{1,1}(0,T;\Vt_\mu) \mbox{ {\rm with} $\wt(\cdot,T)=\zerot$};
\end{align}
%
and, for $s>d+1$,
\begin{align}\label{eqpsinconP-direct}
&-\int_{0}^T
\int_{\Omega \times D} M\,\hpsi\, \frac{\partial \hat\varphi}{\partial t}
\dq \dx \dt
+ \int_{0}^T \int_{\Omega \times D} M\,\left[
\epsilon\, \nabx \hpsi - \ut \,\hpsi \right]\cdot\, \nabx
\hat \varphi
\,\dq \dx \dt
\nonumber \\
&\quad +
\int_{0}^T \int_{\Omega \times D} M\,\left[
\frac{1}{4\,\lambda}\,\nabq \hpsi
-
\left[(\nabxtt \ut)
\,\qt\right] \hpsi \right] \cdot \nabq
\hat \varphi
\dq \dx \dt \nonumber \\
& \quad \quad
= \int_{\Omega \times D} M\, \hat\psi_0(\xt,\qt)\,\hat\varphi(\xt,\qt,0)
\dq \dx
\quad \forall \hat \varphi \in W^{1,1}(0,T;H^s(\Omega\times D))
\mbox{ with $\hat\varphi(\cdot,\cdot,T)=0$}.
\end{align}
\end{subequations}

The triple $(\ut,\ztt,\hpsi)$ can be therefore viewed as a large-data global weak \textit{subsolution} to the Navier--Stokes--Fokker--Planck
system, which becomes a large-data global weak solution if the \textit{stress-defect measure}
\[ 
\ztt - \sigtt(M\hpsi),\]
which, as was proved above, is a symmetric positive semidefinite Radon measure, is in fact equal to $\zerott$ in
$L^\infty(0,T;[\Ctt(\overline\Omega)]')$.

Although with this direct approach one cannot prove the analogue of (\ref{qalphabd}), as one only has a uniform bound on $\utLDtpm$
in $L^\infty(0,T;\Lt^2(\Omega))\cap L^2(0,T;\Ht^1(\Omega))$
and not in $L^1(0,T;\Wt^{1,\infty}(\Omega))$, one can still establish for $8\,L^2\,\Delta t\leq 1$,
in the case $d=2$ at least, the analogue of (\ref{psiGijbdfinlem},b); that is,
\begin{subequations}
\begin{align}
&{\rm ess.sup}_{t\in [0,T]}
\|\sigtt(M\,\hpsiLDtp)(t)\|^2_{L^2(\Omega)}
+ \int_0^T
\|\nabx \,\sigtt(M\,\hpsiLDtp)\|^2_{L^2(\Omega)} \dt
\leq C,
\label{uLbds} \\
&{\rm ess.sup}_{t\in [0,T]}
\|\rho(M\,\hpsiLDtp)(t)\|^2_{L^2(\Omega)}
+ \int_0^T
\|\nabx \,\rho(M\,\hpsiLDtp)\|^2_{L^2(\Omega)} \dt
\leq C.
\label{rhoLbfs}
\end{align}
\end{subequations}
As was noted in Remark \ref{rempsiGij},
in the analogue of (\ref{psiGijbd6})
one bounds
\begin{align*}
&\int_0^{t_{n-1}} \|\nabxtt \utLDtp\|_{L^2(\Omega)}\,\|\hpsiLDtp\|_{L^4(\Omega)}^2 \dd t
\leq C\!\int_0^{t_{n-1}} \|\nabxtt \utLDtp\|_{L^2(\Omega)}\,
\|\hpsiLDtp\|_{L^2(\Omega)}\,\|\hpsiLDtp\|_{H^1(\Omega)} \dd t,
\end{align*}
where we have used (\ref{eqinterp}) as $d=2$.
In addition, we have that
\begin{alignat}{2}
\bet
&
\sigtt(M\,\hpsiLDtp)
\rightarrow
\ztt
&&\quad \mbox{weak* in }
L^{\infty}(0,T;\Ltt^2(\Omega)).
\label{6-CwconL2a}
\end{alignat}
The key difference between (\ref{5-psiwconH1a}--d)
and (\ref{6-psiwconH1a}--d) is (\ref{6-psisconL2a}).

As (\ref{6-psisconL2a}) is not valid for $r=2$, unlike (\ref{5-psisconL2a}),
we still cannot identify the limit $\ztt$ in (\ref{6-CwconL2a}) with $\sigtt(M\,\hpsi)$
using (\ref{sigL1bd}). This is the only step that fails in this direct
existence proof for (P), and therefore the existence of a nonzero stress-defect cannot be ruled out even in the case of $d=2$
by using this direct approach.

The failure of identifying the limit $\ztt$ in (\ref{6-CwconL2a})
with $\sigtt(M\,\hpsi)$ is why for the existence proof in
\cite{BS2010-hookean}, which covered both $d=2$ and $d=3$,
we required that the mapping $s \mapsto U(s)$ had superlinear growth at infinity, recall (\ref{growth1}--c).
If $U$ satisfies (\ref{growth1}--c), then the analogue of Lemma
\ref{qt2bd}, Lemma 4.1 in \cite{BS2010-hookean}, and
the fourth bound in (\ref{PLDtent})
yield that
\begin{align}
{\rm ess.sup}_{t\in [0,T]} \,\int_{\Omega \times D}\!\!
M\,(1+|\qt|)^{2\vartheta}\,\hpsiLDtpm(t) \dq \dx \leq C,
\label{HPq2new}
\end{align}
the analogue of (\ref{HPq2}). Setting $\omega(\qt)=M(\qt)\,
(1+|\qt|)^{2\vartheta}$ and defining $L^2_\omega(\Omega \times D)$
like $L^2_M(\Omega \times D)$, but with $M$ replaced by $\omega$,
one can show that $H^1_M(\Omega \times D) \cap L^2_\omega(\Omega \times D)
\hookrightarrow L^2_\omega(\Omega \times D)$ is compact if $\vartheta>1$,
see Appendix F in \cite{BS2010-hookean}.
One can then establish the strong convergence result (\ref{6-psisconL2a})
for any $r \in [0,2\vartheta]$, and hence one can then identify
$\ztt$ in (\ref{6-CwconL2a})
with $\sigtt(M\,\hpsi)$, and thus prove the existence of large-data global weak solutions to (P)
with $U$ satisfying (\ref{growth1}--c).
Finally, we note that if $\vartheta=1$, e.g.\ $U(s)=s$,
one can demonstrate that the above embedding is not compact by
considering a counterexample; e.g.\ see Remark 3.17 in \cite{ChenL13}.





\bibliographystyle{siam}

\bibliography{polyjwbesrefs}

\bigskip

~\hspace{10.5cm}{\footnotesize \textit{London \& Oxford, \today.}}

\end{document}